\documentclass[11pt,oneside,english]{amsart}
\usepackage[T1]{fontenc}
\usepackage[latin9]{inputenc}
\usepackage{geometry}
\usepackage{booktabs}
\usepackage{textcomp}
\usepackage{mathrsfs}
\usepackage{amsbsy}
\usepackage{amstext}
\usepackage{amsthm}
\usepackage{amssymb}

\makeatletter

\providecommand{\tabularnewline}{\\}

\numberwithin{equation}{section}
\theoremstyle{plain}
\newtheorem{thm}{\protect\theoremname}[section]
\theoremstyle{remark}
\newtheorem{rem}[thm]{\protect\remarkname}
\theoremstyle{plain}
\newtheorem{cor}[thm]{\protect\corollaryname}
\theoremstyle{definition}
\newtheorem{example}[thm]{\protect\examplename}
\theoremstyle{plain}
\newtheorem{lem}[thm]{\protect\lemmaname}
\theoremstyle{plain}
\newtheorem{prop}[thm]{\protect\propositionname}
\theoremstyle{plain}
\newtheorem{conjecture}[thm]{\protect\conjecturename}

\def\makebbb#1{
    \expandafter\gdef\csname#1\endcsname{
        \ensuremath{\Bbb{#1}}}
}\makebbb{R}\makebbb{N}\makebbb{Z}\makebbb{C}\makebbb{H}\makebbb{E}\makebbb{H}\makebbb{P}\makebbb{B}\makebbb{Q}\makebbb{E}\makebbb{E}
\makebbb{H}
\makebbb{F}
\makebbb{A}
\makebbb{G}
\usepackage{babel}

\usepackage{babel}

\makeatother

\usepackage{babel}

\makeatother

\usepackage{babel}
\providecommand{\conjecturename}{Conjecture}
\providecommand{\corollaryname}{Corollary}
\providecommand{\examplename}{Example}
\providecommand{\lemmaname}{Lemma}
\providecommand{\propositionname}{Proposition}
\providecommand{\remarkname}{Remark}
\providecommand{\theoremname}{Theorem}

\begin{document}
\title{Canonical heights, periods and the Hurwitz zeta function}
\author{Rolf Andreasson, Robert J. Berman}
\begin{abstract}
Let $(\mathcal{X},\mathcal{D})$ be a projective log pair over the
ring of integers of a number field such that the log canonical line
bundle $\mathcal{K}_{(\mathcal{X},\mathcal{D})},$ or its dual $-\mathcal{K}_{(\mathcal{X},\mathcal{D})},$
is relatively ample. We introduce a canonical height of $\pm\mathcal{K}_{(\mathcal{X},\mathcal{D})},$
which is finite precisely when the complexifications of $(\mathcal{X},\mathcal{D})$
are K-semistable. When the complexifications of $(\mathcal{X},\mathcal{D})$
are K-polystable, the canonical height is the height of $\pm\mathcal{K}_{(\mathcal{X},\mathcal{D})}$
wrt any volume-normalized Kähler-Einstein metric on the complexifications
of $\pm\mathcal{K}_{(\mathcal{X},\mathcal{D})}.$ The canonical height
is shown to have a number of useful variational properties. Moreover,
it may be expressed as a limit of periods on the $N-$fold products
of the complexifications of $\mathcal{X},$ as $N$ tends to infinity.
In particular, using this limit formula, the canonical height for
the arithmetic log surfaces $(\P_{\Z}^{1},\mathcal{D}),$ where $\mathcal{D}$
has at most three components, is computed explicitly in terms of the
Hurwitz zeta function and its derivative at $s=-1.$ Combining this
explicit formula with previous height formulas for quaternionic Shimura
curves yields a procedure for extracting information about the canonical
integral models of some Shimura curves, such as wild ramification.
Furthermore, explicit formulas for the canonical height of twisted
Fermat curves are obtained, implying explicit Parshin type bounds
for the Arakelov metric. 
\end{abstract}

\email{rolfan@chalmers.se, robertb@chalmers.se}
\maketitle

\section{Introduction}

Let $(\mathcal{X},\mathcal{L})$ be a polarized arithmetic variety
over the ring of integers $\mathcal{O}_{\F}$ of a number field $\F,$
i.e. a projective flat scheme $\mathcal{X}$ over $\mathcal{O}_{\F}$
of relative dimension $n,$ endowed with a relatively ample line bundle
$\mathcal{L}.$ A key role in arithmetic geometry is played by the\emph{
height} $h_{\left\Vert \cdot\right\Vert }(\mathcal{X},\mathcal{L})$
of $(\mathcal{X},\mathcal{L}),$ which is defined with respect to
a given metric $\left\Vert \cdot\right\Vert $ on the $\C-$points
of $\mathcal{L}\rightarrow\mathcal{X}.$ The height is a measure of
the arithmetic complexity of $(\mathcal{X},\mathcal{L})$ and may
be defined in terms of arithmetic intersection theory in the context
of Arakelov geometry \cite{g-s,fa,b-g-s}:
\[
h_{\left\Vert \cdot\right\Vert }(\mathcal{X},\mathcal{L}):=\overline{\mathcal{L}}^{n+1},\,\,\,\,\hat{h}_{\left\Vert \cdot\right\Vert }(\mathcal{X},\mathcal{L}):=\overline{\mathcal{L}}^{n+1}/\left([\F:\Q]L_{\F}^{n}(n+1)\right)
\]
 where $\overline{\mathcal{L}}$ is a shorthand for the metrized line
bundle $(\mathcal{L},\left\Vert \cdot\right\Vert )$ and $\hat{h}_{\left\Vert \cdot\right\Vert }(\mathcal{X},\mathcal{L})$
is called the \emph{normalized height} (which is invariant under base
change). In the classical ``geometric analog'' - where the scheme
$\mathcal{X}$ over $\Z$ corresponds to a fibration $\mathscr{X}\rightarrow\A_{\C}^{1}$
over the complex affine line - the role of the metric $\left\Vert \cdot\right\Vert $
is played by an extension $(\overline{\mathcal{X}},\overline{\mathcal{L}})$
of $(\mathcal{X},\mathcal{L})$ over the compactification $\P_{\C}^{1}$
of $\A_{\C}^{1}.$ Accordingly, the role of the height is played by
the degree of $(\overline{\mathcal{X}},\overline{\mathcal{L}})$. 

From an adelic perspective a metrized polarized arithmetic variety
$(\mathcal{X},\mathcal{L},\left\Vert \cdot\right\Vert )$ over $\mathcal{O}_{\F}$
induces an adelic metric on the adelic extension of the line bundle
$L_{\F}\rightarrow X_{\F}$ (determined by $(\mathcal{X},\mathcal{L})$
and $\left\Vert \cdot\right\Vert $ over the finite and infinite places
of $\F$, respectively) \cite{zh1}. In the case when $X_{\F}$ is
a curve and $K_{X_{\F}}$ is ample there is - after performing a base
change - a canonical relatively ample model $(\mathcal{X},\mathcal{K}_{\mathcal{X}})$
of $(X,K_{X})$ over $\mathcal{O}_{\F},$ which is \emph{stable} (in
the sense of Deligne-Mumford \cite{d-m}). It was proposed by Manin
\cite{man}, and along similar lines by Bost \cite{bo1} and Zhang
\cite{zh2}, that the analog, over the infinite places of $\F$ of
such a canonical model is played by a \emph{Kähler-Einstein metric},
i.e. a Kähler metric $\omega$ on the complex points $X(\C)$ with
constant Ricci curvature (which is negative when $K_{X_{\F}}$ is
ample). Manin's proposal can be made precise using work of Odaka \cite{o},
that will be further developed here in the more general context of
log pairs (see Remark \ref{rem:Manin}).

We will be mainly concerned with the case where $\mathcal{L}$ is
the \emph{log canonical line bundle} $\mathcal{K}_{(\mathcal{X},\mathcal{D})}$
or its dual $-\mathcal{K}_{(\mathcal{X},\mathcal{D})}$ of a\emph{
log pair} $(\mathcal{X},\mathcal{D})$ in the usual sense of the Minimal
Model Program in birational geometry (recalled in Section \ref{subsec:Log-pairs}).
This means that $\mathcal{D}$ is an effective $\R-$ divisor on $\mathcal{X}$
\[
\mathcal{D}=\sum w_{i}\mathcal{D}_{i},\mathrm{\ and}\,\,\,\,\mathcal{K}_{(\mathcal{X},\mathcal{D})}:=\mathcal{K}_{\mathcal{X}}+\mathcal{D},
\]
 where $\mathcal{D}_{i}$ are irreducible effective divisors and $w_{i}$
are non-negative real numbers. The main arithmetic applications concern
the ``orbifold/cusp'' case, where $\mathcal{D}$ has simple normal
crossings and $w_{i}=1-1/m_{i}$ for ``ramification indices'' $m_{i}\in\N\cup\{\infty\}$.
However, allowing general non-negative coefficients $w_{i}$ will
be important in order to apply variational arguments. 

\subsection{Motivation}

In contrast to algebraic degrees, heights can rarely be computed explicitly.
But for special arithmetic varieties $\mathcal{X}$ and metrics -
typically admitting a modular interpretation - it has been conjectured
that the corresponding normalized height can be computed explicitly
in terms of special values of logarithmic derivatives of Dedekind
zeta functions $\zeta{}_{\F}$ (and, more generally, of Artin L-functions).
\cite{m-r1,m-r2,kud}. Very recently, the case of quaternionic Shimura
curves over a totally real field $\F$ was settled by Yuan \cite{yu1}.
The proof, that builds on \cite{y-z-z,y-z}, uses automorphic forms
over the adelic group $GL_{2}(\A_{\F})$ (the case of the classical
modular curve was first shown by Bost and Kuhn \cite{ku}, using modular
forms, and the case when $\F=\Q$ was established by Kudla-Rapoport-Yang
\cite[Thm 1.0.5]{k-r-y} in the context of Kudla's program). Such
a Shimura curve defines a log pair $(X,\Delta)$ over $\F,$ whose
log canonical line bundle $K_{(X,\Delta)}$ is ample. Moreover, by
\cite{y-z}, $X$ admits a canonical model $\mathcal{X}$ over $\mathcal{O_{\F}}.$
Likewise, $K_{(X,\Delta)}$ admits a canonical model over $\mathcal{O_{\F}}$
called the\emph{ Hodge bundle }\cite{y-z}\emph{, }which is is endowed
with the\emph{ Petersson metric}. The Hodge bundle may be identified
with the log canonical line bundle $\mathcal{K}_{(\mathcal{X},\mathcal{D})}$
for a canonical effective $\Q-$divisor $\mathcal{D}$ on $\mathcal{X}.$
Moreover, the Petersson metric can be characterized as the unique
Kähler-Einstein metric on $K_{(X,\Delta)}$ with volume $\pi\deg K_{(X,\Delta)}/2$
(Lemma \ref{lem:The-volume-of Pet}). Yuan's formula \cite{yu1} for
the normalized height of $(\mathcal{X},\overline{\mathcal{K}_{(\mathcal{X},\mathcal{D})}})$
reads as follows: 
\begin{equation}
\hat{h}_{\text{Pet}}(\overline{\mathcal{K}_{(\mathcal{X},\mathcal{D})}})=-\frac{1}{2}-\frac{1}{[\F:\Q]}\frac{\zeta'_{\F}(-1)}{\zeta{}_{\F}(-1)}+\frac{1}{[\F:\Q]}\frac{3N(\frak{\mathfrak{p}})-1}{4(N(\frak{\mathfrak{p}})-1)}\sum_{\frak{\mathfrak{p}}}\log N(\frak{\mathfrak{p}})\label{eq:yuan intro}
\end{equation}
where $\mathfrak{p}$ ranges over the prime ideals in $\mathcal{O}_{\F}$
which are in the ramification locus of the corresponding quaternion
algebra and $N(\mathfrak{p}):=\sharp(\mathcal{O}_{\F}/\mathfrak{p)}.$
When $\F=\Q,$ $\mathcal{X}$ is the coarse moduli scheme parametrizing
all Abelian surfaces over $\Z$ with quaternion multiplication \cite{yu2}
and the divisor $\mathcal{D}$ on $\mathcal{X}$ is the Zariski closure
of $\Delta$ (see Example \ref{exa:course mod space for Q}). 

For general heights, all one can hope for is to obtain explicit bounds.
For example, as shown in the early days of Arakelov geometry by Parshin
\cite{pa} (and further developed in \cite{v,m-b}) an effective version
of the Mordell conjecture (i.e. an effective bound on the height of
$\F-$points of a given curve $X_{\F}$ with $K_{X_{\F}}$ ample)
would follow from a height inequality for stable arithmetic surfaces
$\mathcal{X}$ (in the sense of Deligne-Mumford \cite{d-m}), when
$K_{X(\C)}$ is endowed with the \emph{Arakelov metric }
\begin{equation}
\hat{h}_{\text{Ar}}(\overline{\mathcal{K}_{\mathcal{X}}})\leq c_{0}+c_{1}\sum_{\frak{\mathfrak{p}\text{\ensuremath{\text{bad}}}}}\log N(\frak{\mathfrak{p}})+c_{2}\log\left|D_{\F}\right|\label{eq:parshin intro}
\end{equation}
 for constants $c_{1},c_{2}$ independent of $X$ and a constant $c_{0}$
depending on $X,$ summing over closed points $\frak{\mathfrak{p}}\in\text{Spec \ensuremath{\mathcal{O}_{\F}}}$
of bad reduction and $D_{\F}$ denotes the discriminant of $\F$ (in
particular, $\log\left|D_{\F}\right|\geq0$ with equality when $\F=\Q).$
The geometric analog of this inequality follows, with explicit constants,
from the Miyaoka--Yau inequality. However, the direct arithmetic
analog of Miyaoka--Yau fails for the Arakelov metric, as shown in
\cite{b-m-m-b} (by explicit computations on curves of genus two).
As discussed in \cite{v}, there seems to be no clear understanding
of what the constants $c_{i}$ in the inequality \ref{eq:parshin intro}
should be (but, necessarily, $c_{2}\geq0$, as pointed out in \cite{m-b}).
A suggestion is put forth in Section \ref{subsec:Comparison-with-Parhin's}.

Diophantine aspects of orbifolds $(X,\Delta)$ have also recently
been explored in \cite{ca,ab}. Moreover, in the opposite situation
where $X_{\F}$ is a Fano variety, i.e. $-K_{X_{\F}}$ is ample, some
intriguing conjectural relations between the density of $\F-$points
on $X_{\F}$ and bounds on the height of $-\mathcal{K}_{\mathcal{X}}$
are discussed in \cite{ber1} (for appropriate models $\mathcal{X}$
of $X_{\F}$ over $\mathcal{O}_{\F}$). 

\subsection{Main results }

Let $(\mathcal{X},\mathcal{D})$ be a log pair over $\mathcal{O}_{\F}$
and consider, for simplicity, first the case when $\mathcal{O}_{\F}=\Z.$
Assume that $\pm\mathcal{K}_{(\mathcal{X},\mathcal{D})}$ is relatively
ample, i.e. either the log canonical line bundle $\mathcal{K}_{(\mathcal{X},\mathcal{D})}$
or its dual $-\mathcal{K}_{(\mathcal{X},\mathcal{D})}$ is relatively
ample. Then the complexification $\pm K_{(X,\Delta)}$ of $\pm\mathcal{K}_{(\mathcal{X},\mathcal{D})}$
admits a Kähler-Einstein metric if and only if $\pm K_{(X,\Delta)}$
is \emph{K-polystable} (as recalled in Section \ref{subsec:K-stability}).
For example, when $K_{(X,\Delta)}>0$ any orbifold/cusp pair $(X,\Delta)$
is K-polystable. Accordingly, when $\pm K_{(X,\Delta)}$ is K-polystable
we define the \emph{canonical height} $h_{\text{can}}(\pm\mathcal{K}_{(\mathcal{X},\mathcal{D})})$
of $\pm\mathcal{K}_{(\mathcal{X},\mathcal{D})}$ as the height of
$\pm\mathcal{K}_{(\mathcal{X},\mathcal{D})}$ wrt to any volume-normalized
Kähler-Einstein metric on $\pm K_{(X,\Delta)}.$

\subsubsection{General results}

We show that the canonical height has a number of useful properties,
as the coefficients of $\boldsymbol{w}\in\R^{m}$ of $\mathcal{D}$
are varied, assuming that $\pm\mathcal{K}_{(\mathcal{X},\mathcal{D})}$
stays proportional to one and the same $\R-$line bundle (see Prop
\ref{prop:conc} and Prop\ref{prop:real anal} respectively): 
\begin{itemize}
\item $\pm\hat{h}_{\text{can }}(\pm\mathcal{K}_{(\mathcal{X},\mathcal{D})})$
is\emph{ concave} and\emph{ continuous} wrt $\boldsymbol{w}$ up to
the boundary of the convex domain in $\R^{m}$ where $\pm\mathcal{K}_{(\mathcal{X},\mathcal{D})}$
is K-semistable. 
\item $\pm\hat{h}_{\text{can }}(\pm\mathcal{K}_{(\mathcal{X},\mathcal{D})})$
is \emph{real-analytic} wrt $\boldsymbol{w}$ in the region where
$(X,\Delta)$ is K-stable, if $(X,\Delta)$ is log smooth. 
\end{itemize}
This last statement requires some explanation. So far, we have not
defined $\hat{h}_{\text{can }}(\pm\mathcal{K}_{(\mathcal{X},\mathcal{D})})$
in the ``log Calabi-Yau case'' when $\mathcal{K}_{(\mathcal{X},\mathcal{D})}$
is trivial. But $\hat{h}_{\text{can }}(\pm\mathcal{K}_{(\mathcal{X},\mathcal{D})})$
extends real-analytically over this region and moreover, coincides
with a Faltings type height of $(\mathcal{X},\mathcal{D})$ there,
defined in terms of the period of the trivializing section of $\mathcal{K}_{(\mathcal{X},\mathcal{D})}.$
Likewise, we show that when $\mathcal{K}_{(\mathcal{X},\mathcal{D})}$
is relatively ample, the canonical height of $\mathcal{K}_{(\mathcal{X},\mathcal{D})}$
may be expressed as a limit of periods on large $N-$fold powers $X^{N}$
of a canonical algebraic top form $\alpha$ (see Theorem \ref{thm:period K pos}):
\begin{equation}
\hat{h}_{\text{can }}(\mathcal{K}_{(\mathcal{X},\mathcal{D})})=-\lim_{N\rightarrow\infty}(\frac{i}{2})^{(Nn)^{2}}\log\int_{X^{N}}\alpha\wedge\overline{\alpha},\label{eq:period expr intro}
\end{equation}
 assuming that $(\mathcal{X},\mathcal{D})$ is klt (which is a generic
condition). The proof leverages the probabilistic approach to Kähler-Einstein
metrics introduced in \cite{berm1b,berm1c,berm6b}. A similar limit
formula is established for $-\hat{h}_{\text{can }}(-\mathcal{K}_{(\mathcal{X},\mathcal{D})}),$
conditioned on some conjectures in \cite{berm1c,berm6b}, which hold
when $n=1.$ Anyhow, in practice, $-\hat{h}_{\text{can }}(-\mathcal{K}_{(\mathcal{X},\mathcal{D})})$
may be obtained by analytic continuation of the formula for the canonical
height of relatively ample log canonical line bundles.

The results above naturally extend to any number field $\F,$ by taking
all the complexifications $X_{\sigma}$ of $\mathcal{X}$ into account
(labeled by the embeddings $\sigma$ of $\F$ in $\C$). The real-analyticity
properties above indicate that when switching the sign of $\mathcal{K}_{(\mathcal{X},\mathcal{D})}$
the role of $h_{\text{can}}(\overline{\mathcal{K}_{(\mathcal{X},\mathcal{D})}})$
is played by the \emph{negative }of $h_{\text{can}}(\overline{-\mathcal{K}_{(\mathcal{X},\mathcal{D})}}).$
In Prop \ref{prop:var princi metrics} this phenomenon is illuminated
by establishing a unified variational principle for the invariant
$\pm h_{\text{can }}(\pm\mathcal{K}_{(\mathcal{X},\mathcal{D})}).$
It shows, in particular, that the Kähler-Einstein metrics minimize
a logarithmic generalization of the arithmetic Mabuchi functional
introduced by Odaka \cite{o}. Furthermore, we introduce a notion
of optimality of models: a model $(\mathcal{X}^{o},\mathcal{D}^{o})$
over $\mathcal{O}_{\F}$ for a log pair $(X_{\F},\Delta_{\F})$ is
called\emph{ optimal }if $\mathcal{\pm K}_{(\mathcal{X}^{o},\mathcal{D}^{o})}$
is relatively ample (for some sign) and $(\mathcal{X}^{o},\mathcal{D}^{o})$
minimizes $\pm h_{\text{ }}(\overline{\pm\mathcal{K}_{(\mathcal{X},\mathcal{D})}}):$
\begin{equation}
\pm h_{\text{ }}(\overline{\pm\mathcal{K}_{(\mathcal{X}^{o},\mathcal{D}^{o})}})=\min_{(\mathcal{X},\mathcal{D})}\pm h_{\text{ }}(\overline{\pm\mathcal{K}_{(\mathcal{X},\mathcal{D})}}),\label{eq:optimality cond intro}
\end{equation}
for any fixed metric on $\pm K_{(X,\Delta)(\C)},$ where $(\mathcal{X},\mathcal{D})$
ranges over all models over $\mathcal{O}_{\F}$ for $(X_{\F},\Delta_{\F})$
such that $\mathcal{\pm K}_{(\mathcal{X},\mathcal{D})}$ is relatively
ample. For example, when $K_{X}$ is a curve with $K_{X}>0,$ a stable
model of $X$ over $\mathcal{O}_{\F}$ (in the sense of Deligne-Mumford
\cite{d-m}) is an optimal model for $(X,0).$ This follows from work
of Odaka \cite{od} (see Cor \ref{cor:stable model}). Conjecturally,
in general, our notion of optimality is related to Odaka's notion
of global K-semistability \cite{od} (see Remark \ref{rem:globally K}).
It should be stressed that the optimality condition \ref{eq:optimality cond intro}
is independent of the choice of metric (see Lemma \ref{lem:local h}).

\subsubsection{Explicit formulae for log pairs on $\P_{\Z}^{1}$ with three components}

Consider an effective divisor $\Delta_{\Q}$ on $\P_{\Q}^{1}$ such
that the complexification of $(\P_{\Q}^{1},\Delta_{\Q})$ is K-semistable.
As is well-known \cite{fu2,berm6b}, this amounts to the following
condition on the weights $w_{i}$ of $\Delta_{\Q}:$ 
\begin{equation}
0\leq w_{i}\leq1,\,\,\,w_{i}\leq\frac{1}{2}V+1,\,\,\,V:=\sum_{j}w_{j}-2\label{eq:weight cond intro}
\end{equation}
 (note that, since $V$ is the volume/degree of $K_{(\P_{\Q}^{1},\Delta_{\Q})},$
the second condition is automatic when $K_{(\P_{\Q}^{1},\Delta_{\Q})}$
is semi-ample, i.e. when $V\geq0$). We show (Prop \ref{prop:optimal model for log P one})
that when $\Delta_{\Q}$ is supported on\emph{ three} points, the
optimal model of $(\P_{\Q}^{1},\Delta_{\Q})$ over $\Z$ is $(\P_{\Z}^{1},\mathcal{D}^{o}),$
where $\mathcal{D}^{o}$ is the Zariski closure of the divisor on
$\P_{\Q}^{1}$ supported on $\{0,1,\infty\}$ having the same coefficients
as $\Delta_{\Q}.$ Using the period formula \ref{eq:period expr intro}
the canonical height of $(\P_{\Z}^{1},\mathcal{D}^{o})$ is expressed
explicitly in terms of the Hurwitz zeta function $\zeta(s,x)$ and
its derivative $\zeta'(s,x)$ wrt $s.$ More precisely, setting 
\begin{equation}
\gamma(a,b):=F(b)+F(1-b)-F(a)-F(1-a),\,\,\,F(x):=\zeta(-1,x)+\zeta'(-1,x),\label{eq:def of gamma}
\end{equation}
for $a,b\in[0,1]$ we show:
\begin{thm}
\label{thm:explicit intro}Let $(\P_{\Z}^{1},\mathcal{D}^{o})$ be
as above. When $K_{(\P_{\Q}^{1},\Delta_{\Q})}$ is semi-ample (i.e.
$V\geq0)$ 
\begin{equation}
\hat{h}_{\text{can }}(\mathcal{K}_{(\P_{\Z}^{1},\mathcal{D}^{o})})=f(\boldsymbol{w}):=\frac{1}{2}(1-\log(\pi\frac{V}{2}))-\frac{\gamma(0,\frac{V}{2})-\sum_{i=1}^{3}\gamma(w_{i}-\frac{V}{2},w_{i})}{V}\label{eq:norm height of log K in thm intro}
\end{equation}
and when $-K_{(\P_{\Q}^{1},\Delta_{\Q})}$ is ample (i.e. $V<0)$
\[
\hat{h}_{\text{can }}(-\mathcal{K}_{(\P_{\Z}^{1},\mathcal{D}^{o})})=\frac{1}{2}(1+\log\frac{\pi}{-V/2})+\frac{\gamma(0,\frac{-V}{2})+\sum_{i=1}^{3}\gamma(w_{i},w_{i}-\frac{V}{2})}{V}.
\]
\end{thm}

The theorem applies, in particular, when $\Delta$ is an orbifold/cusp
divisor, i.e. when $w_{i}=1-1/m_{i}$ for $m_{i}\in\N\cup\{\infty\},$
$i\leq3.$ Indeed, in this case $(\P_{\Q}^{1},\Delta_{\Q})$ is always
K-polystable - the corresponding Kähler-Einstein metric is the one
induced by uniformization. In fact, \emph{any} orbifold $(\P_{\Q}^{1},\Delta_{\Q})$
such that $-K_{(\P_{\Q}^{1},\Delta_{\Q})}>0$ has the property that
$\Delta_{\Q}$ is supported on three points and $(\P_{\Q}^{1},\Delta_{\Q})$
isomorphic to the quotient $\P^{1}/G,$ where $G$ is a finite subgroup
of $SU(2)$ \cite[Chapter 8]{ki}. Moreover, when $K_{(\P_{\Q}^{1},\Delta_{\Q})}>0$
and $\Delta$ is an orbifold/cusp divisor supported on three points
the Kähler-Einstein metric for $(\P^{1},\Delta)$ is the one induced
by the action on the upper-half plane $\H$ by a discrete subgroup
$\Gamma$ of $SL(2,\R)$ such that $\overline{\H/\Gamma}\simeq\P_{\C}^{1}$
(known as a triangle group \cite[Prop 1]{tak0,ch}). In this case
$\hat{h}_{\text{can }}(\mathcal{K}_{(\P_{\Z}^{1},\mathcal{D}^{o})})$
appears in the arithmetic Riemann-Roch formula established in \cite{fr-p}
(using a different volume-normalization). Combining the previous theorem
with \cite[Thm 10.1]{fr-p} thus yields an explicit formula for the
derivative at $s=1$ of the corresponding Selberg zeta function $Z(s,\Gamma)$
in terms of the arithmetic degree of the line bundle $\psi_{W}$ defined
in \cite{fr-p} (generalizing the case when $\Gamma=\text{SL }(2,\Z),$
established in \cite[Thm 10.2]{fr-p}). 

In another direction, the previous theorem yields explicit expressions
for Odaka's modular invariant of any polarized log pair $(\P_{\Q}^{1},\Delta_{\Q})$
such that $\Delta_{\Q}$ is supported on three points (see formula
\ref{eq:normal mod invariant of log pairs on p one}).

\subsubsection{Sharp bounds for log pairs on $\P_{\Z}^{1}$}

In \cite{a-b2} a logarithmic arithmetic analog of Fujita's sharp
bound \cite{fu} on the degree of a K-semistable Fano variety was
proposed, which may be formulated as the following bound 

\begin{equation}
h_{\text{can}}(\mathcal{X},\mathcal{D})\leq h_{\text{can}}(\P_{\Z}^{n},0)\,\,\,\,\left(=\frac{1}{2}(n+1)^{n+1}\left((n+1)\sum_{k=1}^{n}k^{-1}-n+\log(\frac{\pi^{n}}{n!})\right)\right)\label{eq:Fuj}
\end{equation}
assuming that $\mathcal{X}$ is a projective scheme over $\Z$ such
that $-\mathcal{K}_{(\mathcal{X},\mathcal{D})}$ is a relatively ample
$\Q-$line bundle and its complexification is K-semistable. Moreover,
for $\mathcal{X}$ normal equality should hold only for $(\P_{\Z}^{n},0).$
The conjecture was, in particular, settled for $n=1$ when $\mathcal{D}\otimes\Q$
is supported at three points. It should, however, be stressed that
the conjectured inequality \ref{eq:Fuj} does\emph{ not} hold when
the canonical height is replaced by its normalization $\hat{h}_{\text{can}}.$
In fact, for all we know it could actually be that $\hat{h}_{\text{can}}(\mathcal{X},\mathcal{D})$
is\emph{ minimal }for $(\P_{\Z}^{n},0)$ among all arithmetic log
Fano varieties $(\mathcal{X},\mathcal{D})$ over $\Z,$ if $(\mathcal{X},\mathcal{D})$
is taken to be an optimal model over $\Z.$ Here we show that this
is, indeed, the case for any arithmetic log Fano surface $(\mathcal{X},\mathcal{D})$
such that $\mathcal{D}\otimes\Q$ is supported at three points. In
this case it was shown in \cite[Section 6]{a-b2} that the optimal
model of $(\mathcal{X},\mathcal{D})$ is of the form $(\P_{\Z}^{1},\mathcal{D}^{o}),$
as appearing in the previous section. The minimality in question thus
follows from the first inequality in the following theorem:
\begin{thm}
\label{thm:sharp bounds intro}Let $(\P_{\Z}^{1},\mathcal{D}^{o})$
be as in the previous section. Then 
\[
\text{\ensuremath{\pm}}\hat{h}_{\text{can}}(\pm\mathcal{K}_{(\P_{\Z}^{1},\mathcal{D}^{o})})\leq-\hat{h}_{\text{can}}(-\mathcal{K}_{(\P_{\Z}^{1},0)})\,\left(=-\frac{1}{2}(1+\log\pi)\right)
\]
 with equality iff $\mathcal{D}^{o}=0.$ Furthermore, if $K_{(\P^{1},\Delta)}$
is semi-ample, then the following more precise inequality holds:
\[
\hat{h}_{\text{can}}(\mathcal{K}_{(\P_{\Z}^{1},\mathcal{D}^{o})})\leq-\frac{1}{2}\log(\pi)+\frac{3}{2}\log\frac{\Gamma(\frac{2}{3})}{\Gamma(\frac{1}{3})}\,\,(<0)
\]
 where equality holds iff $K_{(\P^{1},\Delta)}$ is trivial and all
the coefficients of $\mathcal{D}^{o}$ equal $2/3.$ More generally,
the latter inequality holds when the Kähler-Einstein metric is replaced
by any volume-normalized continuous metric on $K_{(\P^{1},\Delta)}.$ 
\end{thm}

The proof combines the explicit formula in Theorem \ref{thm:explicit intro}
with the concavity of $\pm\hat{h}_{\text{can }}(\pm\mathcal{K}_{(\mathcal{X},\mathcal{D})}).$
Interestingly, the second inequality in the previous theorem is reminiscent
of the result in \cite[page 29]{de} saying that the maximal value
of Faltings' stable height of elliptic curves is attained at the semistable
reduction of the Néron model $\mathcal{X}_{0}$ of any elliptic curve
$X$ with vanishing $j-$invariant. Incidentally, after a base change,
such an elliptic curve is a Galois cover of the log pair $(\P^{1},\Delta)$
saturating the second inequality in the previous theorem. 

\subsubsection{\label{subsec:Specific-values-of}Specific values of canonical heights }

In some orbifold cases the terms involving the Hurwitz zeta function
in Theorem \ref{thm:explicit intro} can be eliminated in favor of
a single logarithmic derivative of a Dedekind zeta function $\zeta{}_{\F}(s)$
of an appropriate totally real number field $\F.$ More precisely,
when the degree $V$ of $K_{(\P_{\Q}^{1},\Delta_{\Q})}$ is positive,
using the classical multiplication theorem for the Hurwitz zeta function
leads - after some computations - to the results in Table 1, where
$(m_{1},m_{2},m_{3})$ denotes the ramification indices of $\mathcal{D}^{o}\otimes\Q$
and $\hat{h}_{\text{Pet }}:=\hat{h}_{\text{can }}(\mathcal{K}_{(\P_{\Z}^{1},\mathcal{D}_{0}})-\log\frac{\pi V}{2}$
(which is the height wrt the Peterson metric). In some of these cases
the coefficients in front of the log $p-$ terms carry information
about the canonical integral model of the Shimura curve attached to
$\F,$ as explained in the following section.

\begin{table}[bh]
\begin{tabular}{cc}
\toprule 
$(m_{1},m_{2},m_{3})$ & $\hat{h}_{\text{Pet }}+\frac{1}{2}+\frac{1}{[\F:\Q]}\frac{\zeta'_{\F}(-1)}{\zeta{}_{\F}(-1)}_{\text{ }}$\tabularnewline
\midrule
\midrule 
$(2,3,\infty)$ & $-\frac{1}{2}\log2-\frac{1}{4}\log3,\ \F=\Q$\tabularnewline
\midrule 
$(6,2,6)$ & $-\frac{1}{6}\log2+\frac{1}{8}\log3,\,\F=\Q$\tabularnewline
\midrule 
$(4,4,4)$ & $-\frac{19}{12}\log2,\ \F=\mathbb{Q}(\sqrt{2})$\tabularnewline
\midrule 
$(3,3,6)$ & $-\frac{5}{6}\log2-\frac{13}{16}\log3,\ \F=\mathbb{Q}(\sqrt{3})$\tabularnewline
\midrule 
$(2,4,12)$ & $-\frac{5}{3}\log2-\frac{7}{16}\log3,\ \mathbb{F}=\mathbb{Q}(\sqrt{3})$\tabularnewline
\midrule 
$(6,6,6)$ & $-\frac{5}{6}\log2-\frac{7}{16}\log3,\ \mathbb{F}=\mathbb{Q}(\sqrt{3})$\tabularnewline
\midrule 
$(5,5,5)$ & $\frac{25}{48}\log5,\ \mathbb{F}=\mathbb{Q}(\sqrt{5})$\tabularnewline
\midrule 
$(3,4,6)$ & $-\frac{9}{16}\log3-\frac{11}{12}\log2,\ \mathbb{F}=\mathbb{Q}(\sqrt{6})$\tabularnewline
\midrule 
$(7,7,7)$ & $-\frac{95}{12^{2}}\log7,\ \F=\mathbb{Q}(\cos(\pi/7))$\tabularnewline
\midrule 
$(9,9,9)$ & $-\frac{31}{24}\log3,\ \mathbb{F}=\mathbb{Q}(\cos(\pi/9))$\tabularnewline
\bottomrule
\end{tabular}\label{table: log canonically polarized curves}\bigskip{}

\caption{Ramification indices and the corresponding normalized height for some
log canonically polarized orbifolds.}
\end{table}
In the opposite case of Fano orbifold curves (which are automatically
K-polystable) further cancellations take place (see Section \ref{subsec:spec log Fano}).

\subsubsection{Application to the canonical integral models of some quaternionic
Shimura curves }

Given a quaternionic Shimura curve $X$ over a totally real field
$\F,$ consider the corresponding canonical model $(\mathcal{X},\mathcal{D})$
of $(X,\Delta)$ over $\mathcal{O}_{\F},$ appearing in formula \ref{eq:yuan intro}.
The log pair $(X,\Delta)$ over $\F$ is \emph{stable,} in the sense
of the Minimal Model Program (MMP) in birational geometry (generalizing
Deligne-Mumford's notion of stability to log pairs; see Section \ref{subsec:Optimal-models}).
However, in general, there exist prime ideals $\mathfrak{p}$ such
that the corresponding log pair $(\mathcal{X},\mathcal{D})$ over
the field $(\mathcal{O}_{\F}/\mathfrak{p})$ is\emph{ not} stable
(even when $\mathcal{D}$ is horizontal). That is to say that the
fibers of the log pair $(\mathcal{X},\mathcal{D})$ over the closed
fibers over $\text{Spec \ensuremath{\mathcal{O}_{\F}}}$ are not stable,
in general. However, there exists, after perhaps performing a finite
base change $\F\subset\F',$ a unique model $(\mathcal{X}^{o},\mathcal{D}^{o})$
for $(X,\Delta)$ over $\mathcal{O}_{\F'}$ all of whose fibers are
stable (see Section \ref{subsec:Var pr for models}). Such a model
is, in fact, an optimal model in the sense of formula \ref{eq:optimality cond intro}.
In order to get a measure of how much $(\mathcal{X},\mathcal{D})$
differs from $(\mathcal{X}^{o},\mathcal{D}^{o})$ one can fix a metric
on $K_{(X,\Delta)(\C)}$ and consider the corresponding height difference
\begin{equation}
h(\overline{\mathcal{K}_{(\mathcal{X},\mathcal{D})}\otimes_{\mathcal{O}_{\F}}\mathcal{O}_{\F'}})-h(\overline{\mathcal{K}_{(\mathcal{X}^{o},\mathcal{D}^{o})}})=\sum_{\mathfrak{p}}h(\frak{\mathfrak{p}})\log N(\frak{\mathfrak{p}})\label{eq:height difference intro-1}
\end{equation}
for a finite set $\mathfrak{p}$ of closed points in $\text{Spec \ensuremath{\mathcal{O}_{\F'}}.}$
The numbers $h(\frak{\mathfrak{p}})$ are independent of the choice
of metric. Indeed, they may be expressed as algebro-geometric intersection
numbers on the fiber $\mathcal{Y}_{\mathfrak{p}}$ over $\mathfrak{p}$
of any given normal model $\mathcal{Y}$ of $X$ over $\mathcal{O}_{\F'},$
dominating $\mathcal{X}\otimes_{\mathcal{O}_{\F}}\mathcal{O}_{\F'}$
and $\mathcal{X}^{o}$ (see Lemma \ref{lem:local h}). Moreover, $h(\frak{\mathfrak{p}})$
is non-negative and vanishes iff $\mathcal{X}\otimes_{\mathcal{O}_{\F}}\mathcal{O}_{\F'}$
is isomorphic to $\mathcal{X}^{o}$ locally around $\frak{\mathfrak{p}}$
(see Remark \ref{rem:local ineq}). 

We will consider some cases where $(\mathcal{X}^{o},\mathcal{D}^{o})$
is of the form $(\P_{\mathcal{O}_{\F'}}^{1},\mathcal{D}^{o}$) for
a horizontal divisor $\mathcal{D}^{o}$ on $\P_{\mathcal{O}_{\F'}}^{1}.$
Given a prime number $p,$ we will compute the rational numbers
\[
h(p):=\frac{1}{[\F':\Q]}\sum_{p\in\mathfrak{p}_{i}}f_{i}h(\frak{\mathfrak{p}_{i}}),\,\,\,\,N(\mathfrak{p}_{i}):=p^{f_{i}},
\]
that are invariant under base change. To the best of our knowledge
these numbers have not been computed before, except for the indefinite
quaternion algebra over $\Q$ with the smallest discriminant $(=1).$
In this case, $\mathcal{X}=\P_{\Z}^{1}$ and $\mathcal{D}$ is an
explicit horizontal divisor supported at three $\Z-$points, as follows
from classical results for the $j-$invariant of elliptic curves.
Accordingly, the numbers $h(\frak{p})$ may be computed directly from
the intersection-theoretic formula in Lemma \ref{lem:local h}. The
result is that $h(p)$ vanishes unless $p=2$ or $p=3$ and 
\begin{equation}
h(2)=1/2,\,\,\,\,h(3)=1/4.\label{eq:h for modular curve}
\end{equation}
In particular, the log pairs $(\mathcal{X},\mathcal{D})\otimes\Z/(p)$
are not stable when $p=2$ or $p=3,$ although $\mathcal{X}$ is smooth
over $\Z$ (see Remark \ref{rem:not stable log pair}). It may also
be worth pointing out that in this case Theorem \ref{thm:explicit intro}
yields a new proof of the height formula \ref{eq:yuan intro}, that
does not use automorphic (or modular) forms, nor uniformization (see
Section \ref{subsec:modular curve}). 

For the indefinite quaternion algebra over $\Q$ with the next to
smallest discriminant $(=6),$ we compute the corresponding rational
numbers $h(p)$ indirectly, by combining Theorem \ref{thm:explicit intro}
with formula \ref{eq:yuan intro}. In this case, by a result of Ihara,
the canonical model of the corresponding quaternionic Shimura curve
$X_{\Q}$ is isomorphic to $\P^{1}$ over $\Q(\sqrt{3},i).$ Moreover,
$\Delta$ is supported on four $\Q(\sqrt{3},i)-$points with ramification
indices $(3,3,2,2)$ and cross ratio $-1$ \cite{el}. We show that
the corresponding unique optimal model of $(X_{\Q},\Delta_{\Q})\otimes\Q(\sqrt{3},i)$
over $\mathcal{O}_{\Q(\sqrt{3},i)}$ is given by $(\P_{\mathcal{O}_{\Q(\sqrt{3},i)}}^{1},\mathcal{D}^{o}),$
where $\mathcal{D}^{o}$ denotes the Zariski closure of the divisor
on $\P_{\Q(\sqrt{3},i)}^{1}$ supported on the pair of two points
$\{\infty,0;1,-1\},$ having the same ramification indices $(3,3;2,2)$
as the divisor $\Delta$ under the action of an automorphism of $\P^{1}.$ 
\begin{thm}
\label{thm:Shimura intro}Consider the quaternion algebra over $\Q$
with discriminant $6$ and denote by $(\mathcal{X},\mathcal{D})$
the canonical model over $\Z$ of the corresponding Shimura curve
$(X_{\Q},\Delta_{\Q}).$ Then $h(p)=0$ unless $p=2$ or $p=3$ and
\[
h(2)=11/18,\,\,\,h(3)=7/12.
\]
\end{thm}

It follows that $(\mathcal{X}\otimes_{\Z}\mathcal{O}_{\Q(\sqrt{3},i)},\mathcal{D})$
is isomorphic to $(\P_{\mathcal{O}_{\Q(\sqrt{3},i)}}^{1},\mathcal{D}^{o})$
away from the fibers over $p=2$ and $p=3$ (see Remark \ref{rem:local ineq}).
This isomorphism also follows from the explicit model for $(\mathcal{X},\mathcal{D})$
over $\Z[1/6]$ established in \cite{h-l}. Since there is a single
prime ideal $\frak{\mathfrak{p}_{2}}$ in $\mathcal{O}_{\Q(\sqrt{3},i)}$
over $2$ and a single one $\frak{\frak{\mathfrak{p}_{3}}}$ over
$3$ the previous theorem also allows one to compute $h(\frak{\mathfrak{p}_{2}})$
and $h(\frak{\mathfrak{p}_{3}}).$ It also follows from the previous
theorem (together with Prop \ref{prop:optimal model for log P one})
that the fibers of the log arithmetic surface $(\mathcal{X},\mathcal{D})\otimes_{\Z}\mathcal{O}_{\Q(\sqrt{3},i)}$
are not stable over $\frak{\mathfrak{p}_{2}}$ and $\frak{\mathfrak{p}_{3}}.$
In the proof of the theorem a covering argument is used to replace
the divisor $\mathcal{D}^{o}$ with the Zariski closure of the divisor
supported on the points $\{0,1,\infty\}$ with ramification indices
$(6,2,6).$ For the later divisor the corresponding the explicit formula
in Table 1 can be employed.

More generally, by the classification results in \cite[Table 3]{tak},
there are $19$ Shimura curves $X_{\F}$ that are isomorphic to $\P_{\bar{\Q}}^{1}$
over $\bar{\Q}$ and such that the corresponding divisor $\Delta_{\bar{\Q}}$
is supported on at most three $\bar{\Q}-$points, up to taking finite
covers. The cases when $\F=\Q$ are precisely the two cases discussed
above. Another case is considered in the following result (see also
Theorem \ref{thm:sqrt six} for one more case).
\begin{thm}
\label{thm:Shimura not Q intro}Consider the quaternion algebra over
$\mathbb{Q}(\sqrt{3})$ that is ramified precisely over the unique
prime ideal in \emph{$\mathcal{O}_{\mathbb{Q}(\sqrt{3})}$ }containing\emph{
3 }and denote by $(\mathcal{X},\mathcal{D})$ the canonical model
over \emph{$\mathcal{O}_{\mathbb{Q}(\sqrt{3})}$} of the corresponding
Shimura curve $(X_{\mathbb{Q}(\sqrt{3})},\Delta_{\mathbb{Q}(\sqrt{3})}).$\emph{
}Fix a finite field extension\emph{ $\F$ }of $\mathbb{Q}(\sqrt{3})$
such that $X_{\mathbb{Q}(\sqrt{3})}\otimes\F$ is isomorphic to $\P_{\F}^{1}$
and $\Delta_{\F}$ is supported on three $\F-$points. Then the optimal
model of $(X_{\mathbb{Q}(\sqrt{3})},\Delta_{\mathbb{Q}(\sqrt{3})})\otimes\F$
over $\mathcal{O}_{\mathbb{Q}(\sqrt{3})}$ is given by $(\P_{\mathcal{O}_{\F}}^{1},\mathcal{D}^{o}),$
where $\mathcal{D}^{o}$ denotes the Zariski closure of the divisor
on $\P_{\F}^{1}$ supported on $\{0,1,\infty\}$ having the same ramification
indices $(2,4,12)$ as the divisor $\Delta_{\F}.$ Moreover $h(p)=0$
unless $p=2$ or $p=3$ and 
\[
h(2)=\frac{5}{9},\ h(3)=\frac{15}{48}.
\]
 
\end{thm}

It follows from the vanishing of $h(p)$ for $p\neq2,3$ in the previous
theorem that the base change of $(\mathcal{X},\mathcal{D})$ to $\mathcal{O}_{\F}$
is isomorphic to $(\P_{\mathcal{O}_{\F}}^{1},\mathcal{D}^{0})$ away
from the fibers of $\mathcal{X}$ over $\{(2),(3)\}\in\text{Spec}\Z$
(see Remark \ref{rem:local ineq}). As a consequence, over this Zariski
open subset the following three properties hold: $(i)$ $\mathcal{X}$
is (geometrically) smooth, $(ii)$ $\mathcal{D}$ is horizontal and
$(iii)$ the irreducible components of $\mathcal{D}$ are mutually
non-intersecting. In fact, $(i)$ and $(ii)$ hold over any prime
ideal $\frak{\mathfrak{p}}$ of \emph{$\mathcal{O}_{\mathbb{Q}(\sqrt{3})}$}
which is not in the ramification locus of the corresponding quaternion
algebra, i.e. when $\frak{\mathfrak{p}}\neq\frak{\mathfrak{p}}_{3}.$
This follows from general results in \cite{yu1} (see the beginning
of Section \ref{subsec:Implications-for-wild}). In particular, they
also hold for $\frak{\mathfrak{p}}=\frak{\mathfrak{p}_{2}}.$ However,
the non-vanishing of $h(2),$ in the previous theorem, implies that
some of the components of $\mathcal{D}$ must intersect over $\frak{\mathfrak{p}_{2}}.$
Furthermore, in Section \ref{subsec:Implications-for-wild}, we show,
building on \cite{yu1}, that there is wild ramification over $\mathcal{X}_{\frak{\mathfrak{p}_{2}}}$
(in the sense of stacks).
\begin{rem}
In order to extend Theorem \ref{thm:Shimura intro} to further Shimura
curves in \cite[Table 3]{tak}, one needs to extend Table 1 to the
other ramification indices in column 5 in \cite[Table 3]{tak}. See
the discussion in Section \ref{subsec:extension of the explicit formula}.
\end{rem}

\subsubsection{Application to twisted Fermat curves}

Given integers $a_{0},a_{1}$ and $a_{2},$ consider the corresponding
twisted Fermat curve $X_{a}^{(m)}$ of degree $m$ in $\P_{\Q}^{2},$
cut out by the polynomial $a_{0}x_{0}^{m}+a_{1}x_{1}^{m}+a_{2}x_{2}^{m},$
for $m\geq3.$ When $a_{1}=a_{2}=1$ and $a_{0}=-1$ this is the classical
Fermat curve that we shall denote by $X^{(m)}.$ Using that $X^{(m)}$
is a Galois cover of $\P^{1}$ of degree $m$ with branching divisor
$\Delta$ supported at $\{0,1,\infty\}$ with ramification indices
$(m,m,m),$ we deduce the following result from Theorems \ref{thm:explicit intro},
\ref{thm:sharp bounds intro}:
\begin{thm}
\label{thm:Fermat intro}Denoting by $\mathcal{X}_{a}^{(m)}$ the
Zariski closure of $X_{a}^{(m)}$ in $\P_{\Z}^{2}$, 
\[
\hat{h}_{\text{can }}(\mathcal{K}_{\mathcal{X}_{a}^{(m)}})=f(1-1/m,1-1/m,1-1/m)+\log m+\left(\frac{(m-3)}{2}+1\right)\frac{1}{m}\sum_{i}\log|a_{i}|,
\]
 where $f$ is defined in \ref{eq:norm height of log K in thm intro}.
As a consequence, for any stable model $\mathcal{Y}_{a}^{(m)}$ of
$X_{a}^{(m)}$ over $\mathcal{O}_{\F}$ and any volume-normalized
metric on $K_{X_{a}^{(m)}(\C)}$
\begin{equation}
\hat{h}_{\text{ }}(\overline{\mathcal{K}_{\mathcal{Y}_{a}^{(m)}}})\leq f(1-1/m,1-1/m,1-1/m)+\log m\,\left(\leq\log m\right).\label{eq:ineq in thm Fermat}
\end{equation}
\end{thm}

Consider in particular the Zariski closure in $\mathcal{X}^{(m)}$
$\P_{\Z}^{2}$ of the Fermat curve of degree $m(\geq4).$ For $m\in\{4,5,6,7,9\}$
Table 1 yields an explicit expression for $\hat{h}_{\text{can }}(\mathcal{K}_{\mathcal{X}^{(m)}})$
in terms of the logarithmic derivatives of Dedekind zeta functions
of the following form: 
\begin{equation}
\hat{h}_{\text{can }}(\mathcal{K}_{\mathcal{X}^{(m)}})=\log\frac{\pi(1-\frac{3}{m})}{2}-\frac{1}{2}-\frac{1}{[\F_{m}:\mathbb{Q}]}\frac{\zeta_{\F_{m}}^{'}(-1)}{\zeta_{\F_{m}}(-1)}-c_{m}\log m,\label{eq:fermat formula}
\end{equation}
 where $\text{\F}_{m}:=\mathbb{Q}\left(\cos(\pi/m)\right),$ $c_{m}\in\Q$
and the sum ranges over all primes $p$ dividing $m.$ However, there
are reasons to doubt that this formula holds for \emph{all} $m\geq4$
- even for $m=8$ (see the discussion in Section \ref{subsec:extension of the explicit formula}). 

In another direction, using the relations between the Kähler-Einstein
metric and the Arakelov metric in \cite{j-k1,j-k2} we deduce the
following general inequalities from the previous theorem:
\begin{cor}
\label{cor:Arak }Denoting by $\mathcal{X}^{(m)}$ the Zariski closure
in $\P_{\Z}^{2}$ of the Fermat curve of degree $m(\geq4),$ 
\begin{equation}
\hat{h}_{\text{Ar}}(\overline{\mathcal{K}_{\mathcal{X}^{(m)}}})\leq\frac{1}{2}(1-\log(\frac{1-3/m}{2}))-\frac{\gamma(0,\frac{1-3/m}{2})-3\gamma(\frac{1}{2}+\frac{1}{2m},1-\frac{1}{m})}{1-3/m}+2\log m+\epsilon_{m}\leq\label{eq:ineq in thm Fermat-1}
\end{equation}
\[
-\frac{1}{2}-\frac{13}{12}\log2-\frac{1}{2}\frac{\zeta'_{\mathbb{Q}(\sqrt{2})}(-1)}{\zeta_{\mathbb{Q}(\sqrt{2})}(-1)}+\epsilon_{m}+2\log m,
\]
 where $\epsilon_{m}$ is the following number, decreasing to $0,$
as $m\rightarrow\infty,$ 
\[
\epsilon_{m}:=\frac{1}{2}\frac{4\log((m-1)(m-2)-2)+1}{(m-1)(m-2)/2-1}+\frac{1}{2}\log(\frac{(m-1)(m-2)-2}{m^{2}}).
\]
\end{cor}

In general, $\hat{h}_{\text{Ar}}(\overline{\mathcal{K}_{\mathcal{X}}})\geq0$,
when $\mathcal{K}_{\mathcal{X}}$ is relatively ample, \cite{fa84}.
Accordingly, the previous corollary gives (using that $\epsilon_{m}\leq\epsilon_{4}$
and by evaluating the constants in question) that 
\begin{equation}
0\leq\hat{h}_{\text{Ar}}(\overline{\mathcal{K}_{\mathcal{X}^{(m)}}})\leq-0.88...+2\log(m).\label{eq:bound Fermat arak intro}
\end{equation}
 Since $\hat{h}_{\text{ }}(\overline{\mathcal{K}_{\mathcal{Y}^{(m)}}})\leq\hat{h}_{\text{\ensuremath{} }}(\overline{\mathcal{K}_{\mathcal{X}^{(m)}}})$
for any stable model $\mathcal{Y}^{(m)}$ of $X^{(m)}$ we thus deduce
an explicit Parshin inequality for $\mathcal{Y}^{(m)}$ (as in inequality
\ref{eq:parshin intro}). The inequality \ref{eq:bound Fermat arak intro}
also holds when $\mathcal{X}^{(m)}$ is replaced by the minimal regular
model $\mathcal{X}_{\text{min}}^{(m)}$ attached to any given finite
field extension $\F$ of $\Q$ (by Prop \ref{prop:minimal model}).
Upper bounds on $\hat{h}_{\text{Ar}}(\mathcal{K}_{\mathcal{X}_{\text{min}}^{(m)}})$
have previously been obtained in \cite{ku2,c-k,cu-m}, when $\F=\Q(\zeta_{m})$
where $\zeta_{m}$ denotes an $m-$th root of unity (assuming that
$m$ is a prime number or square-free). However, the bounds in \cite{ku2,c-k,cu-m}
involve two non-explicit constants $\kappa_{1}$ and $\kappa_{2},$
appearing in the analytic contribution $\kappa_{1}\log m+\kappa_{1}$
to the bounds in \cite{ku2,c-k,cu-m} (originating in \cite[Thm 2.10]{ku2}).
Explicit bounds on $\hat{h}_{\text{\ensuremath{\psi_{\text{Ar}}} }}(\overline{\mathcal{K}_{\mathcal{Y}^{(m)}}})$
that are polynomial in $m$ are contained in \cite[Cor 1.5.1]{ja}.

\subsection{Acknowledgments}

We are deeply grateful to Noam Elkies, Dennis Eriksson, Gerard Freixas
i Montplet, Christian Johansson, John Voight and Xinyi Yuan for very
helpful discussions and feedback. This work was supported by a Wallenberg
Scholar grant from the Knut and Alice Wallenberg foundation.

\section{\label{sec:Setup}Setup}

Henceforth, $\mathcal{X}$ will denote an \emph{arithmetic variety}
(over $\mathcal{O}_{\F}$), i.e. a projective flat scheme $\mathcal{X}$
over the ring of integers $\mathcal{O}_{\F}$ of a number field $\F$
of relative dimension $n$ such that $\mathcal{X}$ is reduced, pure
dimensional, satisfies Serre's conditions $S_{2}$ and has a relative
canonical sheaf $\omega_{\mathcal{X}/B}$ \cite[Condition 1.6.1]{ko}
- the $\Q-$divisor corresponding to $\omega_{\mathcal{X}/B}$ will
be denoted by $\mathcal{K}_{\mathcal{X}}.$ For example, these conditions
are satisfied if $\mathcal{X}$ is normal. We will denote by $\pi$
the corresponding structure morphism from \emph{$\mathcal{X}$} to
$\text{Spec \ensuremath{\mathcal{O}_{\F}}}.$ The corresponding scheme
over $\F,$ $\mathcal{X}\otimes_{\mathcal{O}_{\F}}\F,$ will be denote
by $X_{\F}.$ Furthermore, we will denote by $X_{\sigma}$ the complex
varieties corresponding to $\mathcal{X},$ labeled by the embeddings
$\sigma:\,\F\hookrightarrow\C.$ The $\C-$points of $\mathcal{X}$
will be denoted by $X(\C):$ 
\[
X(\C)=\bigsqcup_{\sigma}X_{\sigma},\,\,\,\,X_{\sigma}:=\mathcal{X}\otimes_{\sigma}\C.
\]
Throughout the paper we will assume that $X_{\sigma}$ is normal.
Given a line bundle $\mathcal{L}\rightarrow\mathcal{X}$ we will denote
by $L^{n}$ the corresponding algebraic top intersection over the
generic fiber of $\mathcal{X}$ (or, equivalently, over the complexifications
$X_{\sigma}$ for any $\sigma).$ We will use additive notation for
tensor products of line bundles and say that $\pm\mathcal{L}$ is
relatively ample if either $\mathcal{L}$ or its dual $-\mathcal{L}$
is relatively ample.

\subsection{\label{subsec:Log-pairs}Log pairs and models}

A \emph{log pair} $(\mathcal{X},\mathcal{D})$ over $\mathcal{O}_{\F}$
(also called an\emph{ arithmetic log variety}) of relative dimension
$n$ is an arithmetic variety $\mathcal{X}$ endowed with an effective
$\R-$divisor $\mathcal{D}$ on $\mathcal{X},$ not contained in the
singular locus of $\mathcal{X},$ such that $\mathcal{K}_{\mathcal{X}}+\mathcal{D}$
is $\R-$Cartier (i.e. a real multiple of $\mathcal{K}_{\mathcal{X}}+\mathcal{D}$
defines a line bundle). See \cite[Section 1.1]{ko} where log pairs
are defined over any excellent ring for $\Q-$divisors and the same
setup applies to $\R-$divisors \cite[Remark 2.20]{ko}. The complexifications
of $(\mathcal{X},\mathcal{D})$ will be denoted by $(X_{\sigma},\Delta_{\sigma}).$
A triple $(\mathcal{X},\mathcal{D};\mathcal{L})$ consisting of a
log pair $(\mathcal{X},\mathcal{D})$ over $\mathcal{O}_{\F}$ and
a relatively ample $\Q-$line bundle $\mathcal{L}$ on $\mathcal{X}$
will be called a \emph{polarized log pair over $\mathcal{O}_{\F}.$ }

Given a polarized log pair $(X,\Delta;L)$ over $\F,$ a \emph{model
for $(X,\Delta;\mathcal{L})$ over $\mathcal{O}_{\F}$} consists,
by definition, of a polarized log pair $(\mathcal{X},\mathcal{D};\mathcal{L})$
over $\mathcal{O}_{\F}$ and an isomorphism between $(\mathcal{X},\mathcal{D};\mathcal{L})\otimes_{\mathcal{O}_{\F}}\F$
and $(X,\Delta;L).$ 

\subsubsection{\label{subsec:Singularities-of-log}Singularities of log pairs}

Given a log pair $(\mathcal{X},\mathcal{D})$ over an excellent ring,
with $\mathcal{X}$ normal, consider a blow-up morphism $p:\mathcal{Y}\rightarrow\mathcal{X}$
from a normal scheme $\mathcal{Y}$ to $\mathcal{X}$ and decompose 

\[
\mathcal{K}_{\mathcal{Y}/\mathcal{X}}-p^{*}\mathcal{D}=\sum_{i}a_{i}E_{i},\,\,\,a_{i}\geq-1,\,\,\,\mathcal{K}_{\mathcal{Y}/\mathcal{X}}:=\mathcal{K}_{\mathcal{Y}}-p^{*}\mathcal{K}_{\mathcal{X}},
\]
 where the prime divisor $E_{i}$ is either an exceptional divisor
of $p$ or the proper transform of a component of $\mathcal{D}.$
Following \cite[Section 2]{ko} $(\mathcal{X},\mathcal{D})$ is said
to be \emph{log canonical (lc)} if $a_{i}\geq-1$ for any such $p:\mathcal{Y}\rightarrow\mathcal{X}$
and \emph{Kawamata Log Terminal (klt)} if $a_{i}>1.$ Without assuming
that $\mathcal{X}$ is normal there is also a notion of\emph{ semi-log
canonical} pairs $(\mathcal{X},\mathcal{D})$ (coinciding with lc
pairs when $\mathcal{X}$ is normal)\cite{ko}. When $(\mathcal{X},0)$
is lc (klt) $\mathcal{X}$ is said to have lc (klt) singularities.

For example, when $(\mathcal{X},\mathcal{D}$) is \emph{log smooth,}
i.e. $\mathcal{X}$ is regular and $\mathcal{D}$ has simple normal
crossings, $(\mathcal{X},\mathcal{D})$ is lc if $w_{i}\leq1$ for
all coefficients $w_{i}$ of $\mathcal{D}$ and klt if $w_{i}<1$
(by \cite[Cor 2.13]{ko}). Moreover, in general, if $\mathcal{X}$
is a normal scheme of dimension one over a perfect field, then $(\mathcal{X},\mathcal{D})$
is lc (klt) iff $w_{i}\leq1$ $(w<1)$ for all coefficients $w_{i}$
of $\mathcal{D}$ \cite[page 43]{ko}. 

\subsection{Metrics and measures}

In this section $X$ will denote a compact complex manifold. 

\subsubsection{\label{subsec:Local-representations-of}Local representations of
metrics and measures}

As in \cite{a-b,a-b2} we will use additive notation for metrics on
holomorphic line bundles $L\rightarrow X.$ This means that we identify
a continuous Hermitian metric $\left\Vert \cdot\right\Vert $ on $L$
with a collection of continuous local functions $\phi_{U}$ associated
to a given covering of $X$ by open subsets $U$ and trivializing
holomorphic sections $e_{U}$ of $L\rightarrow U,$ $\phi_{U}:=-\log(\left\Vert e_{U}\right\Vert ^{2}).$
The curvature current of the metric may then, locally, be expressed
as 
\[
dd^{c}\phi_{U}:=\frac{i}{2\pi}\partial\bar{\partial}\phi_{U}.
\]
Accordingly, as is customary, we will symbolically denote by $\phi$
a given continuous Hermitian metric on $L$ and by $dd^{c}\phi$ its
curvature current. More generally, a \emph{singular metric} $\phi$
on $L$ is defined by the condition that $\phi_{U}\in L_{\text{loc}}^{1}.$
When $dd^{c}\phi_{U}\geq0$ such a metric is called a \emph{psh metric}
(when $\phi_{U}$ is taken to be strongly upper semi-continuous). 

To a log pair $(X,\Delta)$ together with a (multi-valued) section
$s_{\Delta}$ cutting out $\Delta$ and a continuous metric $\phi$
on $\pm K_{(X,\Delta)}$ we attach a measure $\mu_{\phi}$ on $X,$
in the following standard way. First, by definition this measure puts
no mass on $X-X_{reg},$ where $X_{reg}$ denotes the regular locus
of $X.$ Next, locally on $X_{reg}$ the measure $\mu_{\phi}$ is
defined by
\[
\mu_{\phi}=e^{\pm\phi_{U}}\left|s_{U}\right|^{-2}(\frac{i}{2})^{n^{2}}dz\wedge d\bar{z},\,\,\,dz:=dz_{1}\wedge\cdots\wedge dz_{n}
\]
 by taking $e_{U}=\partial/\partial z_{1}\wedge\cdots\wedge\partial/\partial z_{n}\otimes e_{\Delta}$
where $e_{\Delta}$ is a local trivialization of the $\Q-$line bundle
over $X_{reg}$ corresponding to the divisor $\Delta$ and $s_{U}e_{\Delta}$
is the (multi-valued) holomorphic section cutting out $\Delta.$ This
measure is globally well-defined and gives finite mass to $X$ iff
$(X,\Delta)$ is klt \cite[Section 3.1]{bbegz}. Accordingly, a metric
$\phi$ on $\pm K_{(X,\Delta)}$ is called \emph{volume-normalized}
if $\int_{X}\mu_{\phi}=1.$ More generally, if $X$ has several components
$X_{\sigma},$ then $\phi$ is called volume-normalized if $\int_{X_{\sigma}}\mu_{\phi}=1$
for all components. 

\subsubsection{The complex Monge-Ampère measure and finite energy metrics.}

Let $L$ be a semi-ample line bundle over $X$ and fix a continuous
metric $\phi_{0}$ on $L$ with positive curvature current. We define
the complex \emph{Monge-Ampère measure} of a singular metric $\phi$
on $L$ as the $n-$fold product
\[
\text{MA(\ensuremath{\phi)}}:=(dd^{c}\phi)^{n}/L^{n},
\]
 using the notion of non-pluripolar products of positive currents,
introduced in \cite{b-e-g-z}. A psh metric $\phi$ on $L$ is said
to have \emph{finite energy,} if $\text{MA(\ensuremath{\phi)}}$ is
a probability measure and $\text{\ensuremath{\int_{X}}(\ensuremath{\phi-\phi_{0})}MA(\ensuremath{\phi)}}<\infty$
(see \cite{b-e-g-z,bbegz}). For any such metric 
\begin{equation}
\mathcal{E}(\phi):=\mathcal{E}_{\phi_{0}}(\phi):=\int_{X}\sum_{j=0}^{n}(\phi-\phi_{0})(dd^{c}\phi)^{j}\wedge(dd^{c}\phi_{0})^{n-j}<\infty.\label{eq:def of beautif E}
\end{equation}

\begin{rem}
When $n=1$ a psh metric $\phi$ has finite energy iff, locally, the
gradient $\nabla\phi\in L_{\text{loc}}^{2}.$ 
\end{rem}

Given a log pair $(X,\Delta)$ and a psh metric $\phi$ on $\pm K_{(X,\Delta)}$
of finite energy we obtain, just as when $\phi$ is continuous, a
measure $\mu_{\phi}$ on $X.$ If $(X,\Delta)$ is klt then the measure
$\mu_{\phi}$ still gives finite total mass to $X$ \cite{bbegz,b-g}. 
\begin{example}
\label{exa:log log sing}Let $(X,\Delta)$ be a log smooth lc pair
such that $K_{(X,\Delta)}$ is ample and denote by $s_{i}$ the holomorphic
sections cutting out the components $\Delta_{i}$ of $\Delta$ with
coefficient $w_{i}=1.$ A psh metric $\phi$ on $K_{(X,\Delta)}$
is said to have log-log singularities if $\phi$ is locally of the
form $-\sum_{i}\log(\log|s_{i}|^{-2})+O(1).$ Such a psh metric $\phi$
has finite energy \cite[Prop 2.3]{gu}. The corresponding measure
$\mu_{\phi}$ gives finite total mass to $X.$ In contrast, if $\phi$
is locally bounded then $\int_{X}\mu_{\phi}<\infty\iff\text{\ensuremath{\Delta}}$
is klt.
\end{example}

\subsubsection{Kähler-Einstein metrics}

Given a projective log pair $(X,\Delta)$ over $\C$ a metric $\phi$
on $\pm K_{(X,\Delta)}$ is said to be a\emph{ Kähler-Einstein metric,}
if $\phi$ has finite energy and its curvature current $dd^{c}\phi$
induces a Kähler metric with constant positive Ricci curvature on
the complement of $\Delta$ in $X_{reg}$ \cite{bbegz,b-g}. In particular,
by \cite{bbegz,b-g}, a Kähler-Einstein metric $\phi$ on $\pm K_{(X,\Delta)}$
is volume-normalized iff 
\begin{equation}
\text{MA(\ensuremath{\phi})=\ensuremath{\mu_{\phi}}}\label{eq:KE eq}
\end{equation}
By the resolution of the Yau-Tian-Donaldson conjecture $(X,\Delta)$
admits a Kähler-Einstein metric iff $(X,\Delta)$ is \emph{K-polystable}
(as defined in the following section). When $K_{(X,\Delta)}>0$ this
follows from combining the characterization of K-stability in \cite{od,b-h-j}
with \cite{b-g} and when $-K_{(X,\Delta)}>0$ it follows - in the
general singular setup - from the combination of \cite{li1} and \cite{l-x-z}. 
\begin{example}
Let $(X,\Delta)$ be a log smooth lc pair such that $K_{(X,\Delta)}>0.$
Then $K_{(X,\Delta)}$ admits a Kähler-Einstein metric $\phi$ (unique
up to scalings) and $\phi$ has log-log singularities (see Example
\ref{exa:log log sing}). 
\end{example}

Given a variety $X_{\F}$ defined over $\F$ we will say that a metric
$\phi$ on $\pm K_{(X,\Delta)(\C)}$ is Kähler-Einstein if the restriction
of $\phi$ to each component $X_{\sigma}$ is a Kähler-Einstein metric
on $\pm K_{(X_{\sigma},\Delta_{\sigma})}.$

\subsection{\label{subsec:K-stability}K-stability}

We next recall the definition of K-stability in terms of intersection
numbers (see the survey \cite{x} for more background). Let $(X,\Delta)$
be a log pair over $\C$ and $L$ an ample line bundle over $X.$
A \emph{test configuration} for a polarized log pair $(X,L)$ is a
$\C^{*}-$equivariant normal model $(\mathscr{X},\mathscr{L})$ for
$(X,L)$ over the complex affine line $\A_{\C}^{1}.$ More precisely,
$\mathscr{X}$ is a normal complex variety endowed with a $\C^{*}-$action
$\rho$, a $\C^{*}-$equivariant holomorphic surjection $\pi$ to
$\A_{\C}^{1}$ and a relatively ample $\C^{*}-$equivariant $\Q-$line
bundle $\mathscr{L}$ (endowed with a lift of $\rho$): 
\begin{equation}
\pi:\mathcal{\mathscr{X}}\rightarrow\A_{\C}^{1},\,\,\,\,\,\mathscr{L}\rightarrow\mathscr{X},\,\,\,\,\,\,\rho:\,\,\mathscr{X}\times\C^{*}\rightarrow\mathscr{X}\label{eq:def of pi for test c}
\end{equation}
such that the fiber of $\mathscr{X}$ over $1\in\A_{\C}^{1}$ is equal
to $(X,L).$ A log pair $(X,\Delta)$ is said to be \emph{K-semistable}
if $\text{DF}_{\Delta}(\mathscr{X},\mathscr{L})\geq0$ for any test
configuration $(\mathscr{X},\mathscr{L}),$ where\emph{ }$\text{DF}_{\Delta}(\mathscr{X},\mathscr{L})$
is the\emph{ Donaldson-Futaki invariant:} 
\begin{equation}
n!\text{DF}_{\Delta}(\mathscr{X},\mathscr{L}):=\frac{a}{(n+1)!}\overline{\mathscr{L}}^{n+1}+\mathscr{K}_{(\mathcal{\mathscr{\overline{X}}},\mathscr{D})/\P_{\C}^{1}}\cdot\mathcal{\overline{\mathscr{L}}}^{n},\,\,\,a=-n(K_{(X,\Delta)}\cdot L^{n-1})/L^{n}\label{eq:df}
\end{equation}
 where $\overline{\mathscr{L}}$ denotes the $\C^{*}-$equivariant
extension of $\mathscr{L}$ to the $\C^{*}-$equivariant compactification
$\mathscr{\overline{X}}$ of $\mathscr{X}$ over $\P_{\C}^{1}$ and
$\mathscr{K}_{(\mathcal{\mathscr{\overline{X}}},\mathscr{D})/\P_{\C}^{1}}$
denotes the relative log canonical divisor of the pair $(\mathscr{\overline{X}}$,$\mathscr{D}$)
with $\mathscr{D}$ denoting the Zariski closure in $\mathscr{\overline{X}}$
of the $\C^{*}-$orbit of the divisor $\Delta.$ Furthermore, $(X,\Delta;L)$
is said to be \emph{K-polystable }if $\text{DF}_{\Delta}(\mathscr{X},\mathscr{L})\geq0$
with equality iff $\mathcal{X}\simeq X\times\A_{\C}^{1}$ and\emph{
K-stable} if equality only holds when $\mathcal{X}\simeq X\times\A_{\C}^{1}$
for a $\C^{*}-$equivariant isomorphism. 

In the case that $\pm K_{(X,\Delta)}>0$ we will say that $(X,\Delta)$
is K-polystable if $(X,\Delta;\pm K_{(X,\Delta)})$ is K-polystable
(and likewise for K-semistability). We recall the following results
from \cite{od,od-s,b-h-j}: 
\begin{itemize}
\item When $K_{X}>0$ $(X,\Delta)$ is K-polystable iff it is K-semistable
iff $(X,\Delta)$ is log canonical
\item When $kK_{(X,\Delta)}$ is trivial for some $k,$ $(X,\Delta;L)$
is K-polystable for any $L$ iff $(X,\Delta)$ is klt and K-semistable
for any $L$ iff $(X,\Delta)$ is lc.
\item When $-K_{X}>0$ the K-semistability of $(X,\Delta)$ implies that
$(X,\Delta)$ is klt (however, the converse does not hold, in general). 
\end{itemize}
When $X$ is defined over $\F$ we will say that $X(\C)$ is K-polystable
(etc) if $X_{\sigma}$ is K-polystable (etc) for all complexifications
$X_{\sigma}.$ 

\subsection{Canonical heights and optimal models}

\subsubsection{Canonical heights}

A\emph{ metrized line bundle} $\overline{\mathcal{L}}$ is a line
bundle $\mathcal{L}\rightarrow\mathcal{X}$ over an arithmetic variety
$\mathcal{X}$ such that the corresponding line bundle $L(\C)\rightarrow X(\C)$
is endowed with a metric, that we shall denote by $\phi$ (as in Section
\ref{subsec:Local-representations-of}); $\overline{\mathcal{L}}:=\left(\mathcal{L},\phi\right).$
We will assume that $\phi$ has finite energy. When $\phi$ is continuous
the height $h_{\phi}(\mathcal{X},\mathcal{L})$ and the normalized
height $\hat{h}_{\phi}(\mathcal{X},\mathcal{L})$ are defined by 
\[
h_{\phi}(\mathcal{X},\mathcal{L}):=\overline{\mathcal{L}}^{n+1},\,\,\,\hat{h}_{\phi}(\mathcal{X},\mathcal{L}):=\frac{\overline{\mathcal{L}}^{n+1}}{[\F:\Q]L^{n}(n+1)},
\]
 expressed in terms of the arithmetic top intersection numbers of
$\overline{\mathcal{L}}$ \cite{g-s,fa,b-g-s,zh1}. The normalized
height is equivariant under scalings of the metric, 
\begin{equation}
\hat{h}_{\phi+c}=\hat{h}_{\phi}+c/2,\,\,\,\forall c\in\R.\label{eq:normal height under scaling of metric}
\end{equation}
and invariant under base change, induced by finite extensions of $\F$
\cite[Section 3.1.4]{b-g-s}. The definition of $h_{\phi}(\mathcal{X},\mathcal{L})$
extends naturally to any metrized $\R-$line bundle $\mathcal{L},$
by imposing homogeneity. 
\begin{lem}
\label{lem:local h}Let $(X,L)$ be a polarized projective normal
scheme over $\F.$ Consider two metrized models $(\mathcal{X},\overline{\mathcal{L}})$
and $(\mathcal{X}',\overline{\mathcal{L}'})$ of $(X,L)$ over $\mathcal{O}_{\F}.$
Assume that the induced isomorphism between $(\mathcal{X},\mathcal{L})$
and $(\mathcal{X}',\mathcal{L}')$ yields an isometry between $\overline{\mathcal{L}}(\C)$
and $\overline{\mathcal{L}'}(\C).$ Then there exist integers $h(\frak{\mathfrak{p}}),$
where $\frak{\mathfrak{p}}$ ranges over a finite number of closed
points of $\text{Spec}\ \ensuremath{\mathcal{O}_{\F},}$ such that
\[
h_{\phi'}(\mathcal{X}',\mathcal{L}')-h_{\phi}(\mathcal{X},\mathcal{L})=\sum_{\mathfrak{p}}h(\frak{\mathfrak{p}})\log N(\frak{\mathfrak{p}})
\]
Moreover, fixing a model $\mathcal{Y}$ of $X_{\F}$ over $\mathcal{O}_{\F}$
dominating both $\mathcal{X}$ and $\mathcal{X}'$ and identifying
$\mathcal{L}$ and $\mathcal{L}'$ with their pull-backs to $\mathcal{Y},$
\[
h(\frak{\mathfrak{p}})=\sum_{0\leq j\leq n}(\mathcal{L}{}_{|\mathcal{Y}_{\mathfrak{p}}}^{'j}\cdot\mathcal{L}_{|\mathcal{Y}_{\mathfrak{p}}}{}^{n-j})\cdot E_{\mathfrak{p}},\,\,\,\,\sum_{\mathfrak{p}}E_{\mathfrak{p}}:=(s=0)
\]
 where $(s=0)$ denotes the zero-divisor on $\mathcal{Y}$ of the
rational section $s$ of $\overline{\mathcal{L}'}-\overline{\mathcal{L}}$
whose restriction to the generic fiber of $\mathcal{X}$ equals $1\in H^{0}(X_{\F},\mathcal{O}_{X_{\F}})(=\F)$
and the intersection numbers are computed on the projective scheme
$\mathcal{Y}_{\mathfrak{p}}$ over the residue field of $\mathfrak{p}.$
More generally, the formulas above extend, by homogeneity, to the
case when $L_{\F}$ is an $\R-$line bundle.
\begin{proof}
This follows from basic properties of arithmetic intersection numbers.
For future reference we provide a proof. Using the multilinearity
of arithmetic intersection numbers,
\[
h(\overline{\mathcal{L}'})-h(\overline{\mathcal{L}})=(\sum_{j=0}^{n}\overline{\mathcal{L}'}^{j}\cdot\overline{\mathcal{L}}{}^{n-j})\cdot(\overline{\mathcal{L}'}-\overline{\mathcal{L}}).
\]
Now pull back $\mathcal{L}$ and $\mathcal{L}'$ to a model $\mathcal{Y}$
as described in the lemma. By assumption, the restriction of $(\overline{\mathcal{L}'}-\overline{\mathcal{L}})$
to the generic fiber $X_{\F}$ of $\mathcal{Y}\rightarrow\text{Spec}\ \ensuremath{\mathcal{O}_{\F}}$
may by identified with the trivial line bundle $\mathcal{O}_{X_{\F}}\rightarrow X_{\F}$
endowed with its standard metric. The restriction formula for (generalized)
arithmetic intersection numbers \cite[Prop 2.3.1]{b-g-s} (\cite[Prop 6.3]{fr})
thus gives
\[
\overline{\mathcal{L}'}^{j}\cdot\overline{\mathcal{L}}{}^{n-j}\cdot(\overline{\mathcal{L}'}-\overline{\mathcal{L}})=(\overline{\mathcal{L}'}^{j}\cdot\overline{\mathcal{L}}{}^{n-j})\cdot(s=0)-\int_{X(\C)}\log\left\Vert s\right\Vert (dd^{c}\phi)^{n-j}\wedge(dd^{c}\phi')^{j}.
\]
Since $\left\Vert s\right\Vert =1$ on $X(\C)$ and $(s=0)$ is a
vertical divisor on $\mathcal{Y}$ this concludes the proof. 
\end{proof}
\end{lem}

Following \cite{ber-f}, the functional $\phi\mapsto h_{\phi}(\mathcal{X},\mathcal{L})$
admits a canonical extension to a functional on the space of all singular
metrics $\psi$ on $\mathcal{L}$ with positive curvature current
(using that $h_{\phi}(\mathcal{X},\mathcal{L})$ is increasing in
$\phi):$ 
\[
h_{\psi}(\mathcal{X},\mathcal{L}):=\sup_{\phi\leq\psi}h_{\phi}(\mathcal{X},\mathcal{L}),
\]
 where $\phi$ is assumed to be a continuous metric on $L(\C)$ with
positive curvature current. As observed in \cite{ber-f} $h_{\psi}(\mathcal{X},\mathcal{L})$
is finite iff $\psi$ has finite energy and then, for any fixed continuous
metric $\phi_{0}$ on $L(\C),$ 
\begin{equation}
h_{\psi}(\mathcal{X},\mathcal{L})=h_{\phi_{0}}(\mathcal{X},\mathcal{L})+\frac{1}{2}\mathcal{E}_{\phi_{0}}(\psi),\,\,\,\,\,\mathcal{E}_{\phi_{0}}(\phi)=\sum_{\sigma}\mathcal{E}_{\phi_{0}^{\sigma}}(\phi^{\sigma}),\label{eq:change of metrics formula for height}
\end{equation}
where $\phi^{\sigma}$ and $\phi_{0}^{\sigma}$ denote the restrictions
of the metrics $\phi$ and $\phi_{0}$ to $L_{\sigma}\rightarrow X_{\sigma}$
and $\mathcal{E}$ denotes the functional \ref{eq:def of beautif E}.
When $n=1$ and $\nabla\psi\in L_{\text{loc}}^{2}$ this shows that
the height $h_{\psi}(\mathcal{X},\mathcal{L})$ coincides with the
height defined wrt the generalized arithmetic intersection theory
in \cite{bo2}. 

Let now $(\mathcal{X},\mathcal{D})$ be an arithmetic log pair such
that $\pm\mathcal{K}_{(\mathcal{X},\mathcal{D})}$ is relatively ample.
We define the canonical height of $\pm\mathcal{K}_{(\mathcal{X},\mathcal{D})}$
as
\begin{equation}
h_{\text{can }}(\pm\mathcal{K}_{(\mathcal{X},\mathcal{D})}):=\sup_{\phi}h_{\phi}(\mathcal{X},\pm\mathcal{K}_{(\mathcal{X},\mathcal{D})}),\label{eq:def of can height}
\end{equation}
 where the sup ranges over all volume-normalized psh metrics $\phi$
on $\pm K_{(X(\C),\Delta(\C))}$ of finite energy.

\subsubsection{\label{subsec:Optimal-models}Optimal models and the canonical height
over $\F$}

A model $(\mathcal{X}^{o},\mathcal{D}^{o})$ over $\mathcal{O}_{\F}$
for a log pair $(X_{\F},\Delta_{\F})$ will said to be \emph{optimal
}if $\mathcal{\pm K}_{(\mathcal{X}^{o},\mathcal{D}^{o})}$ is relatively
ample (for some sign) and 
\[
\pm h_{\text{ }}(\overline{\pm\mathcal{K}_{(\mathcal{X}^{o},\mathcal{D}^{o})}})=\min_{(\mathcal{X},\mathcal{D})}\pm h_{\text{ }}(\overline{\pm\mathcal{K}_{(\mathcal{X},\mathcal{D})}}),
\]
for any fixed metric on $\pm K_{(X,\Delta)},$ where $(\mathcal{X},\mathcal{D})$
ranges over all models over $\mathcal{O}_{\F}$ for $(X_{\F},\Delta_{\F})$
such that $\mathcal{\pm K}_{(\mathcal{X},\mathcal{D})}$ is relatively
ample. This definition is independent of the choice of metric, by
Lemma \ref{lem:local h}.

\subsection{The arithmetic Mabuchi functional and Odaka's modular invariant}

Let $\overline{\mathcal{L}}\rightarrow\mathcal{X}$ be a metrized
relatively ample line bundle over an arithmetic variety $\mathcal{X}$
over $\mathcal{O}_{\F}.$ When $X_{\F}$ is non-singular and the metric
on $L(\C)$ is smooth, then the corresponding \emph{arithmetic Mabuchi
(K-energy) functional} is defined as follows (in terms of Gillet-Soulé's
arithmetic intersection numbers \cite{g-s}):
\begin{equation}
\mathcal{M}_{\mathcal{X}}(\overline{\mathcal{L}}):=\frac{a}{(n+1)!}\overline{\mathcal{L}}^{n+1}+\frac{1}{n!}\overline{\mathcal{K}}_{\mathcal{X}}\cdot\overline{\mathcal{L}}^{n},\,\,\,\,a=-n(K_{X_{\F}}\cdot L_{\F}^{n-1})/L_{\F}^{n},\label{eq:def of arithm Mab intro-1}
\end{equation}
 where $K_{X(\C)}$ is endowed with the metric induced by the normalized
volume form $\omega^{n}/L^{n}$ of the curvature form $\omega$ of
$\overline{\mathcal{L}}$ (giving total volume one to $X).$ 
\begin{rem}
We have followed the normalizations adopted in \cite{a-b}, which
differ from Odaka's arithmetic Mabuchi functional \cite{o} which
uses the metric on $K_{X}$ induced by the non-normalized volume form
$\omega^{n}/n!$ (as explained in \cite[Section 6.4]{a-b}, when $X$
is Fano, and further discussed in Remark \ref{rem:not volume-norm}).
\end{rem}

Let now $(\mathcal{X},\mathcal{D})$ be a log pair over $\mathcal{O}_{\F}$
and $\mathcal{\overline{\mathcal{L}}}\rightarrow\mathcal{X}$ a metrized
relatively ample line bundle over $\mathcal{X}.$ When $(X_{\F},\Delta_{\F}$)
is log smooth, log canonical and the metric $\phi$ has pre-log-log
singularities in the sense of \cite{b-k-k} (along the non-klt components
of $\Delta$) we define the \emph{arithmetic log Mabuchi functional}
as follows, using the arithmetic intersection theory in \cite{b-k-k,b-b-k}
(see also \cite{ku} for the case $n=1)$:
\begin{equation}
\mathcal{M}_{(\mathcal{X},\mathcal{D})}(\overline{\mathcal{L}}):=\frac{a}{(n+1)!}\overline{\mathcal{L}}^{n+1}+\frac{1}{n!}(\overline{\mathcal{K}}_{(\mathcal{X},\mathcal{D})})\cdot\overline{\mathcal{L}}^{n},\,\,\,\,a=-n(K_{(X,\Delta)}\cdot L^{n-1})/L^{n},\label{eq:def of arithm Mab log smooth pre log}
\end{equation}
 where $K_{(X,\Delta)(\C)}$ is endowed with the normalized volume
form $\omega^{n}/L^{n}$ of the curvature form $\omega$ of $\overline{\mathcal{L}},$
tensored with the singular metric on the $\Q-$line bundle $\Delta,$
induced by the (multivalued) holomorphic section cutting out $\Delta.$
The definition of $\mathcal{M}_{(\mathcal{X},\mathcal{D})}(\overline{\mathcal{L}})$
mimics the definition of the Donaldson-Futaki invariant \ref{eq:df}.

Note that in the case that $\mathcal{L}=\pm\mathcal{K}_{(\mathcal{X},\mathcal{D})},$
that we shall focus on here, 
\begin{equation}
\mathcal{M}_{(\mathcal{X},\mathcal{D})}(\pm\overline{\mathcal{K}}_{(\mathcal{X},\mathcal{D})})=-\pm\frac{n}{(n+1)!}\overline{\mathcal{L}}^{n+1}+\frac{1}{n!}(\overline{\mathcal{K}}_{(\mathcal{X},\mathcal{D})})\cdot\overline{\mathcal{L}}^{n}\label{eq:def of arithm Mab log smooth plus minus K}
\end{equation}
The\emph{ normalized arithmetic log Mabuchi functional} is defined
by

\[
\mathcal{\hat{M}}_{(\mathcal{X},\mathcal{D})}(\overline{\mathcal{L}}):=\frac{\mathcal{M}_{(\mathcal{X},\mathcal{D})}(\overline{\mathcal{L}})}{[\F:\Q]L^{n}/n!}.
\]
It follows readily from the definition that if $q:\mathcal{Y}\rightarrow\mathcal{X}$
is a birational morphism over $\mathcal{O}_{\F}$ with $\mathcal{Y}$
and $\mathcal{X}$ normal, then
\begin{equation}
\mathcal{M}_{(\mathcal{X},\mathcal{D})}(\overline{\mathcal{L}})=\mathcal{M}_{(\mathcal{Y},q^{*}\mathcal{D})}(\overline{q^{*}\mathcal{L}}).\label{eq:pullback Mab}
\end{equation}

\subsubsection{\label{subsec:Odaka's-modular-invariant}Odaka's modular invariant }

Consider now a polarized log pair $(X_{\F},D_{\F};L_{\F})$ over a
number field $\F.$ Following \cite{o} (but using our different normalizations)
we define its \emph{normalized modular invariant} by 
\[
\hat{\mathcal{M}}(X_{\F},D_{\F};L_{\F}):=\inf\mathcal{\hat{M}}_{(\mathcal{X},\mathcal{D})}(\overline{\mathcal{L}})\in]-\infty,\infty[
\]
where the infimum runs over all metrized polarized log pairs $(\mathcal{X},\mathcal{D};\overline{\mathcal{L}})$
over $\mathcal{O}_{\F'}$ where $\F'$ is a finite field extension
of $\F.$ The (non-normalized) \emph{modular invariant} $\mathcal{M}(X_{\F},D_{\F};L_{\F})$
is defined by $\mathcal{\mathcal{\hat{M}}}(X_{\F},D_{\F};L_{\F})[\F:\Q]L^{n}/n!$ 
\begin{example}
\label{exa:mod inv as Faltings height }When $X_{\F}$ is an abelian
variety and $D_{\F}=0,$ our normalizations ensure that $\hat{\mathcal{M}}(X_{\F},D_{\F};L_{\F})$
is precisely Faltings' height \cite{fa} of $X_{\F},$ as follows
from \ref{prop:inf Mab for log CY}, combined with \cite[Thm 2.14]{o}. 
\end{example}

\section{Variational principles }

\subsection{The arithmetic Mabuchi functional when $L_{\F}=\pm K_{(X_{\F},\Delta_{\F})}$
and finite energy metrics}

Consider a general polarized log pair $(\mathcal{X},\mathcal{D};\mathcal{L})$
such that $L_{\F}=\pm K_{(X_{\F},\Delta_{\F})}.$ Denote by $E_{\pm}$
the vertical divisor on $\mathcal{X}$ cut out by the rational section
of $\mathcal{L}-\pm\mathcal{K}_{(\mathcal{X},\mathcal{D})}$ whose
restriction to the generic fiber $X_{\F}$ of $\mathcal{X}$ coincides
with $1\in H^{0}(X_{\F},\mathcal{O}_{\F}).$ Given a finite energy
metric $\phi$ on $L(\C)$  we then define $\mathcal{\hat{M}}_{(\mathcal{X},\mathcal{D})}(\overline{\mathcal{L}})$
as follows,
\begin{equation}
\mathcal{\hat{M}}_{(\mathcal{X},\mathcal{D})}(\mathcal{L},\phi):=\pm\hat{h}_{\phi}(\mathcal{X},\mathcal{L})+\frac{1}{2}\text{Ent \ensuremath{(\text{MA}(\ensuremath{\phi)}}|\ensuremath{\mu_{\phi})}}-\frac{E_{\pm}\cdot\mathcal{L}^{n}}{L_{\F}^{n}}\label{eq:def norm Mab for finite energy}
\end{equation}
 where $\text{Ent}(\mu|\ensuremath{\mu_{0})}$ denotes the \emph{entropy}
of a measure $\mu$ relative to a measure $\mu_{0}:$ 
\begin{equation}
\text{Ent}(\mu|\ensuremath{\mu_{0})}:=\int_{X(\C)}\log\frac{\mu}{\mu_{0}}\mu\,\left(=\sum_{\sigma}\int_{X_{\sigma}}\log\frac{\mu}{\mu_{0}}\mu\right),\label{eq:def of ent}
\end{equation}
 if $\mu$ is absolutely continuous wrt $\mu_{0}$ and $\text{Ent}(\mu|\ensuremath{\mu_{0})}:=\infty,$
otherwise. Note that since $\phi$ has finite energy the height term
is always finite. Moreover, since $E_{\pm}$ is a vertical divisor
the last term in formula \ref{eq:def norm Mab for finite energy}
is independent of the metric on $\mathcal{L}.$
\begin{lem}
\label{lem:finite energy Mab}Assume that $(X_{\F},\Delta_{\F})$
is log smooth and log canonical, that $L_{\F}=\pm K_{(X_{\F},\Delta_{\F})}$
and that $\phi$ is a metric on $L(\C)$ with pre-log-log singularities.
Then the definitions \ref{eq:def of arithm Mab log smooth pre log}
and \ref{eq:def of arithm Mab log smooth pre log} are compatible 
\end{lem}

\begin{proof}
Let $\phi$ be a psh metric with pre-log-log singularities. Then $\phi$
has finite energy (see Example \ref{exa:log log sing}) and, as a
consequence, $\mu_{\phi}$ has total mass. We rewrite the definition
\ref{eq:def of arithm Mab log smooth plus minus K} of $\mathcal{M}_{(\mathcal{X},\mathcal{D})}(\overline{\mathcal{L}})$,
where $\overline{\mathcal{L}}=(\mathcal{L},\phi),$ as
\[
\mathcal{M}_{(\mathcal{X},\mathcal{D})}(\overline{\mathcal{L}})=-\pm\frac{n}{(n+1)n!}\overline{\mathcal{L}}^{n+1}+\frac{\pm}{n!}\overline{\mathcal{L}}^{n}+\frac{1}{n!}(\overline{\mathcal{K}}_{(\mathcal{X},\mathcal{D})}-\pm\overline{\mathcal{L}})\cdot\overline{\mathcal{L}}^{n}=
\]
\[
\pm\frac{1}{(n+1)!}\overline{\mathcal{L}}^{n}+\frac{1}{n!}(\overline{\mathcal{K}}_{(\mathcal{X},\mathcal{D})}-\pm\overline{\mathcal{L}})\cdot\overline{\mathcal{L}}^{n}.
\]
Denote by $\psi$ the induced metric on $K_{(X,\Delta)(\C)}.$ Since
$\overline{\mathcal{K}}_{(\mathcal{X},\mathcal{D})}-\pm\overline{\mathcal{L}}=:-E_{\pm}$
is a vertical divisor the restriction formula \cite[Prop 6.3]{fr}
yields
\[
(\overline{\mathcal{K}}_{(\mathcal{X},\mathcal{D})}-\pm\overline{\mathcal{L}})\cdot\overline{\mathcal{L}}^{n}=\int_{X(\C)}(\psi-\pm\phi)(dd^{c}\phi)^{n}-E_{\pm}\cdot\mathcal{L}^{n}.
\]
 Moreover, since the measure $(dd^{c}\phi)^{n}$ does not charge $X(\C)-\text{supp\ensuremath{(\Delta(\C))}}$
and $\psi$ is represented by $\log\text{MA(\ensuremath{\phi)}}$
on $X-\text{supp\ensuremath{(\Delta)}}$ it follows that
\[
\frac{1}{L_{\F}^{n}}(\overline{\mathcal{K}}_{(\mathcal{X},\mathcal{D})}-\pm\overline{\mathcal{L}})\cdot\overline{\mathcal{L}}^{n}=\frac{1}{2}\int_{X(\C)}\log\frac{\text{MA}(\ensuremath{\phi)}}{\mu_{\phi}}\text{MA}(\ensuremath{\phi)}-E_{\pm}\cdot\mathcal{L}^{n},
\]
 which concludes the proof.
\end{proof}

\subsection{\label{subsec:Variational-principles-for metr}Variational principles
for metrics}
\begin{lem}
Let $(\mathcal{X},\mathcal{D})$ be a log pair such that $\pm\mathcal{K}_{(\mathcal{X},\mathcal{D})}$
is relatively ample and $\phi$ a volume-normalized psh metric on
$\pm K_{(X,\Delta)(\C)}$ with finite energy. Then 
\[
\mathcal{\hat{M}}_{(\mathcal{X},\mathcal{D})}(\pm\mathcal{K}_{(\mathcal{X},\mathcal{D})},\phi)\geq\pm\hat{h}_{\phi}(\pm\mathcal{K}_{(\mathcal{X},\mathcal{D})})
\]
 with equality iff $\phi$ is a Kähler-Einstein metric. 
\end{lem}

\begin{proof}
When $\mathcal{L}=\pm\mathcal{K}_{\mathcal{X}}$ we have that $E_{\pm}=0$
in formula \ref{eq:def norm Mab for finite energy}. Indeed, $\mathcal{L}\pm\mathcal{K}_{\mathcal{X}}$
is the trivial line bundle and $1\in H^{0}(\mathcal{X},\mathcal{O}_{\mathcal{X}})$
has no zeroes on $\mathcal{X}$ under our assumptions on $\mathcal{X}$
(as shown precisely as in the case $\F=\Q$ considered in \cite[Lemma 2.3]{a-b}).
The lemma thus follows from combining the expression \ref{eq:def norm Mab for finite energy}
for $\mathcal{\hat{M}}_{(\mathcal{X},\mathcal{D})}(\pm\mathcal{K}_{(\mathcal{X},\mathcal{D})},\phi)$
with the Kähler-Einstein equation \ref{eq:KE eq}, using that, for
any given probability measures $\mu$ and $\mu_{0},$ $\text{Ent}(\mu|\ensuremath{\mu_{0})}\geq0$
with equality iff $\mu=\mu_{0}$ (by Jensen's inequality).
\end{proof}
\begin{prop}
\label{prop:var princi metrics}Let $(\mathcal{X},\mathcal{D})$ be
a log pair, whose complexification is klt, such that either $\mathcal{K}_{(\mathcal{X},\mathcal{D})}$
or $-\mathcal{K}_{(\mathcal{X},\mathcal{D})}$ is relatively ample.
Then 
\begin{equation}
\inf_{\phi}\mathcal{\hat{M}}_{(\mathcal{X},\mathcal{D})}(\pm\mathcal{K}_{(\mathcal{X},\mathcal{D})},\phi)=\pm\sup_{\phi}\left(\hat{h}(\pm\mathcal{K}_{(\mathcal{X},\mathcal{D})},\phi):\,\:\text{\ensuremath{\phi\,\text{vol-normalized}}}\right):=\pm\hat{h}_{\text{can}}(\pm\mathcal{K}_{(\mathcal{X},\mathcal{D})}),\label{eq:inf is plus minus sup}
\end{equation}
 where $\phi$ ranges over all finite energy psh metrics on $\pm K_{(X,\Delta)(\C)}.$
Moreover, the inf and sup above are attained iff $\phi$ is a Kähler-Einstein
metric. In particular, 
\begin{equation}
\inf_{\phi}\mathcal{\hat{M}}_{(\mathcal{X},\mathcal{D})}(\pm\mathcal{K}_{(\mathcal{X},\mathcal{D})},\phi)=\pm\hat{h}_{\text{can}}(\pm\mathcal{K}_{(\mathcal{X},\mathcal{D})})=\pm\hat{h}_{\phi_{\text{KE}}}(\pm\mathcal{K}_{(\mathcal{X},\mathcal{D})})\label{eq:inf Mab is canonical height is height of KE}
\end{equation}
for any volume-normalized Kähler-Einstein metric $\phi_{KE},$ if
such a metric exists (i.e. if $\pm K_{(X,\Delta)}$ is K-polystable).
More generally, if $\mathcal{K}_{(\mathcal{X},\mathcal{D})}$ is relatively
ample and $(\mathcal{X},\mathcal{D})$ is log canonical (equivalently,
$(X,\Delta)$ is K-stable and $K_{(X,\Delta)}$ admits a Kähler-Einstein
metric) then the identities \ref{eq:inf Mab is canonical height is height of KE}
still hold.
\end{prop}

\begin{proof}
Introducing the normalized \emph{arithmetic log Ding functional }defined
by\emph{
\begin{equation}
\hat{\text{\emph{D}}}_{(\mathcal{X},\mathcal{D})}((\pm\mathcal{K}_{(\mathcal{X},\mathcal{D})},\phi)=-\hat{h}(\pm\mathcal{K}_{(\mathcal{X},\mathcal{D})},\phi)\pm\sum_{\sigma}\frac{1}{2[\F:\C]}\log\int_{X_{\sigma}}\mu_{\phi},\label{eq:arithm Ding}
\end{equation}
}it is equivalent (by scaling the restrictions of $\phi$ to $X_{\sigma}$)
to prove that 
\[
\inf_{\phi}\mathcal{\hat{M}}_{(\mathcal{X},\mathcal{D})}(\pm\mathcal{K}_{(\mathcal{X},\mathcal{D})},\phi)=-\pm\inf_{\phi}\hat{\text{\emph{D}}}((\pm\mathcal{K}_{(\mathcal{X},\mathcal{D})},\phi),
\]
 where $\phi$ ranges over all finite energy metrics on $\pm K_{(X,\Delta)(\C)}.$
Now set $\mathcal{L}=\pm\mathcal{K}_{(\mathcal{X},\mathcal{D})}$
and fix a reference metric $\phi_{0}$ of finite energy on $L(\C)$
(for example a continuous psh metric). We can then, using formula
\ref{eq:def norm Mab for finite energy}, rewrite
\begin{equation}
\mathcal{\hat{M}}_{(\mathcal{X},\mathcal{D})}(\mathcal{L},\phi):=\pm\frac{1}{2}\left(2\hat{h}(\mathcal{L},\phi)-\int_{X}(\phi-\phi_{0})\text{MA}(\ensuremath{\phi)}\right)+\frac{1}{2}\text{Ent \ensuremath{(\text{MA}(\ensuremath{\phi)}}|\ensuremath{\mu_{\phi_{0}})} }\label{eq:norm Mab for finite energy-1}
\end{equation}
(the klt assumption ensures that $\mu_{\phi_{0}}$ has finite total
mass). Hence, 
\begin{equation}
\mathcal{\hat{M}}_{(\mathcal{X},\mathcal{D})}(\mathcal{L},\phi)=\frac{1}{2}\mathcal{\hat{M}}_{\phi_{0}}(\phi)\pm\hat{h}(\mathcal{L},\phi_{0})\label{eq:arith Mab in terms of Mab}
\end{equation}
 where $\mathcal{M}_{\phi_{0}}(\phi)$ is defined by replacing $(\mathcal{L},\phi)^{n+1}$
in formula \ref{eq:norm Mab for finite energy-1} with $\mathcal{E}_{\phi_{0}}(\phi)/2.$
Likewise, 
\begin{equation}
\hat{\text{\emph{D}}}_{(\mathcal{X},\mathcal{D})}((\pm\mathcal{K}_{(\mathcal{X},\mathcal{D})},\phi)=\frac{1}{2}\hat{\text{\emph{D}}}_{\phi_{0}}(\phi)-\hat{h}(\mathcal{L},\phi_{0}),\label{eq:arithm Ding is Ding}
\end{equation}
 where $\hat{\text{\emph{D}}}_{\phi_{0}}(\phi)$ is defined by replacing
$(\mathcal{L},\phi)^{n+1}$ in formula \ref{eq:arithm Ding} with
$\mathcal{E}_{\phi_{0}}(\phi)/2.$ All in all, by decomposing 
\[
\mathcal{\hat{M}}_{\phi_{0}}(\phi)=\sum_{\sigma}\mathcal{\hat{M}}_{\phi_{0}^{\sigma}}(\phi^{\sigma}),\,\,\,\hat{\text{\emph{D}}}_{\phi_{0}}(\phi)=\sum_{\sigma}\mathcal{\hat{M}}_{\phi_{0}^{\sigma}}(\phi^{\sigma}),
\]
 where $\phi^{\sigma}$ is the restriction of $\phi$ to $X_{\sigma}$
and $\mathcal{\hat{M}}_{\phi_{0}^{\sigma}}(\phi^{\sigma})$ and $\hat{\text{\emph{D}}}_{\phi_{0}^{\sigma}}(\phi^{\sigma})$
are defined by decomposing both terms appearing in the definitions
of $\mathcal{\hat{M}}_{\phi_{0}}(\phi)$ and $\hat{\text{\emph{D}}}_{\phi_{0}}(\phi)$
wrt $\sigma.$ All in all, this means that it is equivalent to prove
the following identity:
\begin{equation}
\inf_{\phi^{\sigma}}\mathcal{\hat{M}}_{\phi_{0}^{\sigma}}(\phi^{\sigma})=-\pm\inf_{\phi^{\sigma}}\hat{\text{\emph{D}}}_{\phi_{0}^{\sigma}}(\phi^{\sigma}),\label{eq:inf M hat is D hat}
\end{equation}
 where $\phi_{0}^{\sigma}$ ranges over all psh metrics on $L_{\sigma}$
with finite energy. But this identity follows from results in \cite{berm1,bbegz}.
For future reference we recall the reduction to \cite{berm1,bbegz},
which uses the thermodynamical formalism introduced in \cite{berm1}.
Let $X$ be a complex projective variety and assume that $L=\pm K_{(X,\Delta)}$
is ample. Given a reference metric $\phi_{0}$ on $L$ consider the
functional $E$ on the space\emph{ $\mathcal{P}(X)$ }of all probability
measures $\mu$ on $X$ defined by
\[
E(\mu)=\sup_{\phi}\left(\frac{\mathcal{E}_{\phi_{0}}(\phi)}{(n+1)L^{n}}-\int_{X}(\phi-\phi_{0})\mu\right),
\]
where the sup ranges over all psh metrics $\phi$ on $L$ with finite
energy. In the terminology introduced in \cite{begz} $E(\mu)$ is
the\emph{ pluricomplex energy }of $\mu$ (relative to $dd^{c}\phi_{0}$).
\footnote{we have followed the notation in \cite{berm1}, which differs from
the notation in \cite{begz} where the pluricomplex energy is denoted
$E^{*}(\mu)$).} Next, given $\beta\in\R$ the corresponding \emph{free energy functional
}$F_{\beta}$ on $\mathcal{P}(X)$ is defined by 
\begin{equation}
F_{\beta}(\mu)=\beta E(\mu)+\text{Ent }(\mu|\mu_{\phi_{0}}),\label{eq:def of free}
\end{equation}
if $E(\mu)<\infty.$ Otherwise, $F_{\beta}(\mu):=\infty$ \footnote{we have followed the notation in \cite{ber2} which differs from the
notation in \cite{berm1} where the role of $F_{\beta}$ is played
by $\beta F_{\beta}$. } Recall that $\text{Ent }(\mu|\mu_{0})$ is the relative entropy defined
in formula \ref{eq:def of ent}. By \cite[Thm A]{begz}, $E(\mu)<\infty$
iff there exists a finite energy psh metric $\phi_{\mu}$ solving
$\text{MA(\ensuremath{\phi_{\mu})=\mu.}}$ Moreover, the sup defining
$E(\mu)$ is then attained at $\phi_{\mu}.$ Hence, if $\phi$ has
finite energy, we can express
\begin{equation}
\mathcal{\hat{M}}_{\phi_{0}}(\phi)=F_{\pm1}(\text{MA(\ensuremath{\phi)}}).\label{eq:Mab as free}
\end{equation}
The identity\ref{eq:inf M hat is D hat}  thus follows from the following
identity, applied to $\beta=\pm1:$
\begin{equation}
\inf_{\mu\in\mathcal{P}(X)}F_{\beta}(\mu)=-\beta\inf_{\phi}\hat{\text{\emph{D}}}_{\phi_{0}}(\phi).\label{eq:inf F is inf D}
\end{equation}
 When $\beta=-1$ this identity follows from \cite[Thm 1.1]{berm1}
for $X$ non-singular and the same argument applies in general (see
\cite[Lemma 4.4]{bbegz}). When $\beta=1$ the identity \ref{eq:inf F is inf D}
follows from \cite[Thm 3.3]{berm1} when $X$ is non-singular and,
again, the same argument applies in general. Anyhow, we will prove
the case $\beta=1$ directly in the more general setup of log canonical
pairs. But we first note that the statement in the proposition about
Kähler-Einstein metrics follows from the well-known fact that the
optimizers of both the Mabuchi functional $\mathcal{\hat{M}}_{\phi_{0}}$
and the Ding functional $\hat{\text{\emph{D}}}_{\phi_{0}}$ (relative
to $\phi_{0})$ are precisely the Kähler-Einstein metrics on $\pm K_{(X,\Delta)}.$
See \cite[Thm 3.3]{berm1} for the case $\beta=-1$ and \cite[Thm C]{begz}
and for the case $\beta=1.$ 

Finally, assume that $\mathcal{K}_{(\mathcal{X},\mathcal{D})}$ is
relatively ample and $(\mathcal{X},\mathcal{D})$ is log canonical.
By \cite{b-g}, $K_{(X,\Delta)(\C)}$ admits a unique Kähler-Einstein
metric $\phi_{\text{KE}}$ with finite energy. In particular, the
corresponding measure $\mu_{\phi_{\text{KE}}}$ has finite total mass
(as follows from the Kähler-Einstein equation \ref{eq:KE eq}). In
fact, as shown in \cite{b-g}, $\phi_{\text{KE}}$ minimizes the corresponding
Ding functional $\hat{\text{\emph{D}}}_{\phi_{0}},$ appearing in
formula \ref{eq:arithm Ding is Ding}. It will thus be enough to show
that $\phi_{\text{KE}}$ also minimizes $\mathcal{\hat{M}}_{\phi_{0}},$
or equivalently: that $\text{MA}(\phi_{\text{KE}})$ minimizes the
corresponding free energy functionals $F_{1}$ for any $X_{\sigma}.$
To this end restrict to $X_{\sigma}$ and set $\phi_{0}:=\phi_{\text{KE}},$
assuming that $\phi_{\text{KE}}$ is volume-normalized, i.e. that
$\mu_{\phi_{0}}$ is a probability measure. This implies (by Jensen's
inequality) that $\text{Ent }(\mu|\mu_{\phi_{0}})\geq0$ iff $\mu=\mu_{\phi_{0}}.$
But, in general, we also have $E(\mu)=0$ iff $\text{\ensuremath{\mu=\text{MA (\ensuremath{\phi_{0})}}}}$
\cite{begz,bbegz}. Hence, $F(\mu)\geq0$ with equality iff $\mu=\mu_{\phi_{0}}.$
Since $\phi_{0}$ is assumed to be Kähler-Einstein metric this concludes
the proof.
\end{proof}
We also note the following
\begin{lem}
\label{lem:sup over all cont phi}The sup defining $\hat{h}_{\text{can}}(\pm\mathcal{K}_{(\mathcal{X},\mathcal{D})})$
may, equivalently, be taken over all continuous psh metrics on $\pm K_{(X,\Delta)(\C)}$
and when $n=1$ the sup may be taken over \emph{all} continuous metrics.
Moreover, if $(X,\Delta)(\C)$ is log smooth and klt then both the
inf and the sup in formula \ref{eq:inf is plus minus sup} may, equivalently,
be taken over all log smooth psh metrics $\phi$ (i.e. such that the
curvature form $\omega_{\phi}$ of $\phi$ has conical singularities
along $\Delta$).
\end{lem}

\begin{proof}
To prove the first result recall that when $L$ is an ample line bundle
over a normal complex projective variety $X$ any psh metric $\psi$
on $L$ is the decreasing limit of continuous (and even smooth) psh
metrics $\psi_{j}$ \cite[Cor C]{c-g-z}. Hence, the first statement
of the lemma follows from the fact that the Ding function $\hat{\text{\emph{D}}}_{\phi_{0}},$
appearing in formula \ref{eq:arithm Ding is Ding}, is continuous
under decreasing limits (indeed, for the integral term this follows
from the monotone convergence theorem in integration theory and for
the term $\mathcal{E}_{\phi_{0}}(\phi)$ this follows from \cite[Thm 2.17]{b-e-g-z}).
Next consider the case when $n=1.$ Following \cite{b-b}, given a
continuous psh metric $\phi$ on $L:=\pm K_{(X,\Delta)(\C)}$ denote
by $P_{X}\phi$ the continuous psh metric on $L$ defined as the sup
of all continuous psh metrics $\psi$ on $L$ satisfying $\psi\leq\phi.$
Then $P_{X}\phi\leq\phi,$ giving $\pm\log\mu_{P_{X}\phi}(X_{\sigma})\leq\pm\log\mu_{\phi}(X_{\sigma}).$
Hence, by formula \ref{eq:arithm Ding}, it is enough to show that
$h_{\phi}(\mathcal{L})\leq h_{P_{X}\phi}(\mathcal{L}).$ But, by formula
\ref{eq:change of metrics formula for height}, this follows from
\[
\int_{X(\C)}(P_{X}\phi-\phi)(dd^{c}P_{X}\phi+dd^{c}\phi)=\int_{X(\C)}(P_{X}\phi-\phi)(-dd^{c}P_{X}\phi+dd^{c}\phi)\geq0,
\]
 using in the first equality that $\int_{X(\C)}(P_{X}\phi-\phi)(dd^{c}P_{X}\phi)=0$
(by \cite[Prop 2.10]{b-b}). The inequality then follows by integrating
by parts to get $\int_{X(\C)}d(P_{X}\phi-\phi)\wedge d^{c}(P_{X}\phi-\phi),$
which is an $L^{2}-$norm and thus non-negative. Finally, to prove
the statement concerning pairs $(X,\Delta)(\C)$ that are log smooth
and klt first note that, as in the proof of the previous proposition,
it is enough to prove the corresponding statement for the log Mabuchi
functional $\mathcal{\hat{M}}_{\phi_{0}}$ and log Ding functional
$\hat{\text{\emph{D}}}_{\phi_{0}}.$ But the latter property follows
from essentially well-known regularization results for $\hat{\text{\emph{D}}}_{\phi_{0}}$
and $\mathcal{\hat{M}}_{\phi_{0}}.$ For example, when $\Delta=0,$
the regularization result in question for $\mathcal{\hat{M}}_{\phi_{0}}$
appears in \cite[Lemma 3.1]{bdl1} and the case when $\Delta\neq0$
is shown in precisely the same way, but replacing the use of the Calabi
theorem in the proof of \cite[Lemma 3.1]{bdl1} with \cite[Thm A]{g-p}
(with $\mu=0).$ 
\end{proof}
It should be stressed that, in general, the finiteness of $\hat{h}_{\text{can}}(-\mathcal{K}_{(\mathcal{X},\mathcal{D})})$
does not imply that $-K_{(X,\Delta)}$ admits a Kähler-Einstein metric,
or equivalently, that $(X,\Delta)$ is K-polystable. For example,
when $\Delta=0,$ it was shown in \cite[Thm 2.4]{a-b} that the finiteness
in question is equivalent to the K-semistability of $X$ (which, in
general, is weaker than K-polystability). More generally, we have:
\begin{thm}
Let $(\mathcal{X},\mathcal{D})$ be a log pair such that $\pm\mathcal{K}_{(\mathcal{X},\mathcal{D})}$
is relatively ample. Then $\hat{h}_{\text{can}}(\pm\mathcal{K}_{(\mathcal{X},\mathcal{D})})<\infty$
if and only if $(X,\Delta)$ is K-semistable.
\end{thm}

\begin{proof}
When $-\mathcal{K}_{(\mathcal{X},\mathcal{D})}$ is relatively ample
this is shown in, essentially, the same way as in case $\Delta=0$,
considered in \cite[Thm 2.4]{a-b}. Next, for log pairs over $\C$
such that $K_{(X,\Delta)}$ is ample it is shown in \cite[Thm 2.4]{b-g}
that the inf of the corresponding Ding functional $\mathcal{D}_{\phi_{0}}(\phi)$
of all psh metrics $\phi$ of finite energy is finite iff $(X,\Delta)$
is lc, which concludes the proof using formula \ref{eq:arithm Ding is Ding}
and the results described in Section \ref{subsec:K-stability}.
\end{proof}
\begin{cor}
\label{cor:h finite implies klt etc}Let $(\mathcal{X},\mathcal{D})$
be a log pair. If $-\mathcal{K}_{(\mathcal{X},\mathcal{D})}$ is relatively
ample and $\hat{h}_{\text{can}}(-\mathcal{K}_{(\mathcal{X},\mathcal{D})})$
is finite, then $(X,\Delta)$ is klt. If $\mathcal{K}_{(\mathcal{X},\mathcal{D})}$
is relatively ample and $\hat{h}_{\text{can}}(\mathcal{K}_{(\mathcal{X},\mathcal{D})})$
is finite, then $(X,\Delta)$ is lc. 
\end{cor}

\begin{proof}
This follows from the previous theorem, using the relations between
the K-semistability of $(X,\Delta)$ and the singularities of $(X,\Delta)$
recalled in Section \ref{subsec:K-stability}. Alternatively, a  direct
analytic proof can be given using that for any given finite energy
psh metric $\phi$ on $-K_{(X,\Delta)}$ ($K_{(X,\Delta)}$) the total
mass $\mu_{\phi}(X)$ is finite iff $(X,\Delta)$ is klt (lc) \cite{b-g}.
\end{proof}

\subsubsection{Intermezzo: the log Calabi-Yau case and Faltings' height}

Next, assume that $(\mathcal{X},\mathcal{D})$ is a log Calabi-Yau
pair, in the sense that there exists a positive integer $k$ such
that $k\mathcal{K}_{(\mathcal{X},\mathcal{D})}$ is trivial. Denote
by $\alpha$ the multivalued meromorphic top form on $X_{\sigma}$
defined as the tensor product of the $k$:th root of a generator of
$H^{0}(\mathcal{X},k\mathcal{K}_{(\mathcal{X},\mathcal{D})})$ with
the inverse of the (multivalued) section $s_{\mathcal{D}}$ cutting
out $\mathcal{D}.$ Then one can define a \emph{Faltings' height}
of $(\mathcal{X},\mathcal{D})$ by 
\begin{equation}
h_{\text{Falt}}(\mathcal{X},\mathcal{D}):=-\frac{1}{2[\F:\Q]}\log\prod_{\sigma}(\frac{i}{2})^{n^{2}}\int_{X_{\sigma}}\alpha_{\sigma}\wedge\bar{\alpha}_{\sigma}\in[\infty,\infty[,\label{eq:Falting height}
\end{equation}
 which is finite iff $(X,\Delta)(\C)$ is klt (as follows directly
from the analytic characterization of klt pairs). When $\mathcal{X}$
is an abelian variety and $\mathcal{D}=0$ this is the usual definition
of the Faltings height \cite{fa} (see also \cite{de}, where a different
normalization is adopted). 
\begin{prop}
\label{prop:inf Mab for log CY}Assume that some tensor power of $\mathcal{K}_{(\mathcal{X},\mathcal{D})}$
is trivial. Then, for any relatively ample line bundle $\mathcal{L}$
over \emph{$\mathcal{X},$}
\begin{equation}
\inf_{\psi}\mathcal{\hat{M}}_{(\mathcal{X},\mathcal{L})}(\psi)=h_{\text{Falt}}(\mathcal{X},\mathcal{D}),\label{eq:inf Mab is Falt}
\end{equation}
where the inf ranges over all psh metrics on $L(\C)$ of finite energy.
In particular, the inf above is finite iff $(X,\Delta)(\C)$ is klt.
\end{prop}

\begin{proof}
In the case that $(X,\Delta)$ is klt the proof proceeds as in the
case when $\Delta=0,$ considered in \cite[Prop 6.5]{a-b2}. Next,
when $(X,\Delta)$ is not klt we need to prove that the inf in the
lemma equals $-\infty.$ To this end fix a sequence of increasing
compact sets $C_{j}$ exhausting the complement in $X$ of the support
of $\Delta.$ We can take $C_{j}$ to be the closure of open domains
in $X$ and consider the probability measures 
\[
\mu_{j}:=1_{C_{j}}(i/2)^{n^{2}}\alpha\wedge\bar{\alpha}/\int_{C_{j}}(i/2)^{n^{2}}\alpha\wedge\bar{\alpha},
\]
 where $1_{C_{j}}$ denotes the characteristic function of $C_{j}.$
By \cite[Thm B]{b-e-g-z} there exists psh metrics $\phi_{j}$ on
$L$ of finite energy such that $\text{MA(\ensuremath{\phi_{j})}=\ensuremath{\mu_{j}} }$.
Indeed, $\phi_{j}$ is even locally bounded. A slight variant of Lemma
\ref{lem:finite energy Mab} gives
\[
\mathcal{M}_{(\mathcal{X},\mathcal{L})}(\phi_{j})=\frac{1}{2}\text{Ent \ensuremath{(\text{MA}(\ensuremath{\phi_{j})}}|\ensuremath{(i/2)^{n^{2}}\alpha\wedge\bar{\alpha})} }=-\log\int_{C_{j}}(i/2)^{n^{2}}\alpha\wedge\bar{\alpha},
\]
 which converges to $-\log\int_{X}(i/2)^{n^{2}}\alpha\wedge\bar{\alpha},$
as $j\rightarrow\infty,$ by the monotone convergence theorem. Finally,
since, by the analytic characterization of klt pairs, $\int_{X}i^{n^{2}}\alpha\wedge\bar{\alpha}$
is finite iff $(X,\Delta)$ is klt, this concludes the proof. 
\end{proof}
In the light of the variational principles in Prop \ref{prop:var princi metrics}
and Lemma \ref{prop:inf Mab for log CY} it is thus natural to define
\[
\pm\hat{h}_{\text{can }}(\pm\mathcal{K}_{(\mathcal{X},\mathcal{D})}):=h_{\text{Falt}}(\mathcal{X},\mathcal{D})
\]
when $(\mathcal{X},\mathcal{D})$ is log Calabi-Yau. 
\begin{rem}
As recalled in Section \ref{subsec:K-stability}, a polarized log
Calabi-Yau pair $(X,\Delta,L)$ is K-semistable iff $(X,\Delta)$
is lc. Hence, the previous lemma reveals that - in contrast to the
case when $L=\pm K_{(X,\Delta)}$ - K-semistability is \emph{not}
equivalent to the finiteness of the inf of $\mathcal{\hat{M}}_{(\mathcal{X},\mathcal{L})}(\phi)$
over all metrics on $L$ of finite energy, in the log Calabi-Yau case
(only K-polystability is).
\end{rem}

\subsection{\label{subsec:Var pr for models}Variational principles for models}

In this section we will, for simplicity, assume that all arithmetic
varieties are\emph{ normal. }Consider two metrized models $(\mathcal{X},\mathcal{D};\mathcal{\overline{L}})$
and $(\mathcal{X}',\mathcal{D}';\overline{\mathcal{L}'})$ for $(X_{\F},\Delta_{\F},L_{\F})$
over $\mathcal{O}_{\F}.$ Assume that the induced isomorphism between
$\mathcal{\overline{L}}(\C)$ and $\mathcal{\overline{L}}'(\C)$ is
an isometry. Then the difference $\mathcal{M}_{(\mathcal{X},\mathcal{D})}(\overline{\mathcal{L}})-\mathcal{M}_{(\mathcal{X},\mathcal{D})}(\overline{\mathcal{L}})$
is independent of the induced metric on $L_{\F}(\C):$ 
\begin{equation}
\mathcal{M}_{(\mathcal{X}',\mathcal{D}')}(\overline{\mathcal{L}'})-\mathcal{M}_{(\mathcal{X},\mathcal{D})}(\overline{\mathcal{L}})=\sum_{b}m(\frak{b})\log N(\frak{b})\label{eq:difference Man in terms of b}
\end{equation}
 for a finite number of closed points $b\in\text{Spec \ensuremath{\mathcal{O}_{\F}}, }$
where $N(b)$ denotes the cardinality of the residue field of $b$
and $m(\frak{b})$ may be expressed in terms of intersection numbers
over the fiber $\mathcal{Y}_{b}$ of any fixed model $\mathcal{Y}$
dominating both $\mathcal{X}$ and $\mathcal{X}'.$ This is shown
precisely as in the proof of Lemma \ref{lem:local h}. In fact, the
difference \ref{eq:difference Man in terms of b} is even independent
of the choice of a fixed metric on $K_{(X,\Delta)}(\C)$ in formula
\ref{eq:def of arithm Mab log smooth pre log}. Accordingly, in this
section we shall fix any pair of metrics on $L(\C)$ and $K_{(X,\Delta)}(\C)$
and denote by $\mathcal{M}_{(\mathcal{X},\mathcal{D})}(\mathcal{L})$
the corresponding (generalized) arithmetic Mabuchi functional. 

We recall the following result from \cite[Thm 2.14]{o}: 
\begin{thm}
\label{thm:(Odaka).-Given-a}(Odaka). Given a projective scheme $(X_{\F},K_{X_{\F}})$
such that $K_{X_{\F}}$ defines an ample $\Q-$line bundle, assume
that $\mathcal{X}$ is a model of $X_{\F}$ over $\ensuremath{\mathcal{O}_{\F}}$
such that $\mathcal{X}$ is normal, $\mathcal{K}_{\mathcal{X}}$ is
relatively ample and $(\mathcal{X},\mathcal{X}_{b})$ is log canonical
for any closed point $b$. Then
\[
\mathcal{M}_{\mathcal{X}^{o}}(\mathcal{K}_{\mathcal{X}^{o}})\leq\mathcal{M}_{\mathcal{X}'}(\mathcal{L}')
\]
 for any relatively ample model $(\mathcal{X}',\mathcal{L}')$ of
$(X_{\F},K_{X_{\F}})$ over $\text{Spec}\ \ensuremath{\mathcal{O}_{\F}}$
and metric on $K_{X}(\C).$ 
\end{thm}

It follows from the previous theorem (applied to $\mathcal{L}'$ of
the form $\mathcal{K}_{\mathcal{X}'}$ ) that $\mathcal{X}^{o}$ is
an optimal integral model for $X_{\F}$ (in the sense of \ref{subsec:Optimal-models}).
Using inversion of adjunction we also deduce the following corollary,
where $\mathcal{X}_{s}$ is a stable model in the sense of Deligne-Mumford
\cite{d-m} (which always exists, after a base change):
\begin{cor}
\label{cor:stable model}Let $X_{\F}$ be a non-singular projective
curve over $\F$ such that $K_{X_{\F}}>0$ and $\mathcal{X}^{s}$
a stable model of $X_{\F}$ over $\text{Spec}\ \ensuremath{\mathcal{O}_{\F}}.$
Then 
\[
\mathcal{M}_{\mathcal{X}^{s}}(\mathcal{K}_{\mathcal{X}^{s}})\leq\mathcal{M}_{\mathcal{X}'}(\mathcal{L}')
\]
 for any relatively ample model $(\mathcal{X}',\mathcal{L}')$ of
$(X_{\F},K_{X_{\F}})$ over $\text{Spec}\thinspace\ensuremath{\mathcal{O}_{\F}}.$
In particular, $\mathcal{X}^{s}$ is an optimal model for $X_{\F}$
(in the sense of section \ref{subsec:Optimal-models}).
\end{cor}

\begin{proof}
Recall that $\mathcal{K}_{\mathcal{X}^{s}}$ is relatively ample \cite[page  78]{d-m}.
Next, by inversion of adjunction for surfaces $X$ over excellent
rings \cite[Thm 5.1]{tan} a log pair $(X,C),$ where $C$ is assumed
to be a reduced divisor, is log canonical iff the scheme $C$ has
semi-log canonical singularities, i.e. iff the log pair $(C^{\nu},D^{\nu})$
is log canonical, where $C^{\nu}$ denotes the normalization of $C$
and $D^{\nu}$ denotes the reduced divisor on $C^{\nu}$ defined by
the conductor. Now, by the very definition of stable models in \cite{d-m}
the scheme $\mathcal{X}_{b}^{s}$ is geometrically reduced and thus,
in particular, reduced. Moreover, since the scheme $\mathcal{X}_{b}^{s}$
has only ordinary double points (by definition), its normalization
is regular. $C^{\nu}$ being reduced, it thus follows that $((\mathcal{X}_{b}^{s})^{\nu},D^{\nu})$
is log canonical, as desired. 
\end{proof}
\begin{rem}
\label{rem:Manin}Let $X_{\F}$ be as in the previous corollary. Combining
the previous corollary with Prop \ref{prop:var princi metrics} reveals
that $\mathcal{M}_{\mathcal{X}}(\mathcal{L},\phi)$ is minimal when
$(\mathcal{X},\mathcal{L})=(\mathcal{X}_{s},\mathcal{K}_{\mathcal{X}^{s}})$
and $\phi$ is a Kähler-Einstein metric on $K_{X(\C)}.$ Likewise,
the minimum of $h(\mathcal{K}_{\mathcal{X}})$ over all models $\mathcal{X}$
of $X_{\F}$ with relatively ample $\mathcal{K}_{\mathcal{X}}$ and
volume-normalized continuous psh metrics $\phi$ is attained for $\mathcal{X}=\mathcal{X}^{s}$
and $\phi$ the unique volume-normalized Kähler-Einstein metric on
$K_{X(\C)}.$ This is in line with the suggestion put forth in \cite[ Section 3.1]{man}.
\end{rem}

The previous theorem can be generalized to log pairs $(\mathcal{X},\mathcal{D})$
such that: 
\begin{equation}
(1)\,(\mathcal{X},\mathcal{D}+\mathcal{X}_{b})\,\text{is \emph{lc} for any closed \ensuremath{b\in\text{Spec}\,\mathcal{O}_{\F},\,\,\,(2)} \ensuremath{\mathcal{K}_{(\mathcal{X},\mathcal{D})}}}\,\text{is relatively ample}\label{eq:cond one and two}
\end{equation}
But for our purposes it will be enough to consider the case of arithmetic
surfaces:
\begin{prop}
\label{prop:cond one and two}Let $(\mathcal{X},\mathcal{D})$ be
an arithmetic log pair over $\mathcal{O}_{\F}$ satisfying conditions
$1$ and $2$ above. Then 
\[
\mathcal{M}_{(\mathcal{X},\mathcal{D})}(\mathcal{K}_{(\mathcal{X},\mathcal{D})})\leq\mathcal{M}_{(\mathcal{X}',\mathcal{D}')}(\mathcal{L}')
\]
 for any relatively ample model $(\mathcal{X}',\mathcal{D}';\mathcal{L}')$
of $(X_{\F},\Delta_{\F};K_{X_{\F}})$ over $\text{Spec}\ensuremath{\mathcal{O}_{\F}}.$
More precisely, $m(b)\geq0$ for any closed point $b\in\text{Spec}\ensuremath{\mathcal{O}_{\F}}$
(where $m(b)$ is the number appearing in formula \ref{eq:difference Mab in pf var pr}).
In particular, if $\mathcal{L}'=\mathcal{K}_{(\mathcal{X}',\mathcal{D}')},$
then $h(b)\geq0$ for any closed point $b$ and, as a consequence,
$(\mathcal{X},\mathcal{D})$ is an optimal model for $(X_{\F},\Delta_{\F}).$ 
\end{prop}

\begin{proof}
We will generalize the proof of Cor \ref{cor:stable model} and Thm
\ref{thm:(Odaka).-Given-a}, following the argument in \cite[Section 6]{a-b2}.
Set $\mathcal{L}:=\mathcal{K}_{(\mathcal{X},\mathcal{D})}.$ By Step
1 in \cite[Sections 6.2, 6.3.1]{a-b2}, there exists a regular arithmetic
surface $\mathcal{Y}$ with birational morphisms $p$ and $q$ to
$\mathcal{X}$ and $\mathcal{X}',$ respectively (which are isomorphisms
over the generic point of $\text{Spec}\ensuremath{\mathcal{O}_{\F}})$
such that 
\[
q^{*}\mathcal{L}'=p^{*}\mathcal{L}-E,\,\,\,\,\,\left(\implies\mathcal{M}_{(\mathcal{X}',\mathcal{D}')}(\mathcal{L}')=\mathcal{M}_{(\mathcal{Y},q^{*}\mathcal{D}')}(p^{*}\mathcal{L}-E)\right)
\]
for a $p-$exceptional effective $\Q-$divisor $E$ on $\mathcal{Y},$
which vanishes iff $p$ is an isomorphism \ref{eq:pullback Mab} (using
the pull-back formula \ref{eq:pullback Mab} for the implication).
A direct computation gives 
\begin{equation}
\mathcal{M}_{(\mathcal{X}',\mathcal{D}')}(\mathcal{L}')-\mathcal{M}_{(\mathcal{X},\mathcal{D})}(\mathcal{L})=\frac{1}{2}q^{*}\mathcal{L}'\cdot E+q^{*}\mathcal{L}'\cdot\left(\mathcal{K}_{\mathcal{Y}/\mathcal{X}}-p^{*}\mathcal{D}+q^{*}\mathcal{D}'\right).\label{eq:difference Mab in pf var pr}
\end{equation}
The first term above is non-negative, since $E$ is effective and
$q^{*}\mathcal{L}'$ is semi-ample. Thus all that remains is to verify
that $\mathcal{K}_{\mathcal{X}^{'}/\mathcal{X}}-p^{*}\mathcal{D}+q^{*}\mathcal{D}'$
is effective, under the assumptions on $\mathcal{D}.$ The condition
$1$ is, by inversion of adjunction on excellent surfaces \cite[Thm 5.1]{tan},
equivalent (since $\mathcal{X}$ is normal) to the following property:
$(\mathcal{X},\mathcal{X}_{b}+\mathcal{D})$ is log canonical for
any $b.$ But then it follows from \cite[Lemma 7.2 (4)]{k-m} that
$\mathcal{K}_{\mathcal{X}^{'}/\mathcal{X}}-p^{*}\mathcal{D}+q^{*}\mathcal{D}'$
is effective. This proves the inequality \ref{eq:ineq prop var pr mod}.
\end{proof}
In general, morphisms $(\mathcal{X},\mathcal{D})\rightarrow B$ satisfying
the condition $1$ are called \emph{log canonical (lc)} in the context
of the Minimal Model Program (MMP) \cite{k-m}. The fibers $\mathcal{X}_{b}$
are automatically reduced and if condition $2$ also holds, then the
restricted log pair $(\mathcal{X}_{b},\mathcal{D}_{b})$ is a\emph{
stable pair} in the sense of the MMP (i.e. $(\mathcal{X}_{b},\mathcal{D}_{b})$
is semi-log canonical and $\mathcal{K}_{(\mathcal{X}_{b},\mathcal{D}_{b})}$
is ample \cite{kol}). When $B=\text{Spec \ensuremath{\mathcal{O}_{\F}}}$
and $n=1$ the existence of a model satisfying $1$ and $2$ above,
after a base change, follows from \cite[Prop 3.7]{ha} (under some
regularity assumptions). The existence in any dimension is shown in
\cite[Cor 1.5]{h-x}, when the ground field is $\C.$ Here we will
focus on the following simple case:
\begin{lem}
\label{lem:optimal model for three and four pts}Consider a log canonical
pair $(\P_{\F}^{1},\Delta_{\F})$ (i.e. the coefficients $w_{i}$
of $\Delta_{\F}$ are in $[0,1]$) such that $K_{(\P_{\F}^{1},\Delta_{\F})}>0.$
Assume that either, (a) $\Delta_{\F}$ is supported on three $\F-$
points $(p_{0},p_{1},p_{\infty})$ in $\P_{\F}^{1}$, or, (b) $\Delta_{\F}$
is supported on four $\F-$points $(p_{\infty},p_{0};p_{1},p_{-1})$
in $\P_{\F}^{1}$ with cross ratio $-1,$ such that $w_{1}+w_{-1}\leq1.$
Then there exists an automorphism $f$ of $\P_{\F}^{1}$ mapping $(p_{0},p_{1},p_{\infty})$
to $(0,1,\infty)$ and $(p_{\infty},p_{0};p_{1},p_{-1})$ to $(\infty,0;1,-1)$
respectively and the Zariski closure $\mathcal{D}$ of $f_{*}(\Delta_{\F})$
in $\P_{\mathcal{O}_{\F}}^{1}$ has the property that $(\P_{\mathcal{O}_{\F}}^{1},\mathcal{D})$
satisfies conditions $1$ and $2$ above. 
\end{lem}

\begin{proof}
The existence of $f$ is a classical fact. By inversion of adjunction
for excellent surfaces (see the proof of the second point below) it
is enough to show that $\mathcal{D}_{b}$ is log canonical for any
$b$ in $\text{Spec}\ensuremath{\mathcal{O}_{\F}},$ i.e. that $\mathcal{D}\otimes_{\mathcal{O}_{\F_{b}}}\F_{b}$
has coefficients in $[0,1],$ where $\F_{b}$ denotes the residue
field of $b.$ But this is immediate, since $\{0,1,\infty\}$ corresponds
to three distinct points in $\P_{\F_{b}}^{1}$ for any $b.$ When
the divisor $\mathcal{D}\otimes_{\mathcal{O}_{\F}}\F$ on $\P_{\F}^{1}$
is supported on $\{0,1,-1,\infty\}$ the only case where $\{0,1,-1,\infty\}$
does not correspond to four distinct points in $\P_{\F_{b}}^{1}$
is when $1$ and $-1$ correspond to the same point in $\P_{\F_{b}}^{1}$
which thus acquires the weight $w_{-1}+w_{1}.$ Hence, if $w_{-1}+w_{1}\leq1,$
then $\mathcal{D}_{|\P_{\F_{b}}^{1}}$ has coefficients in $[0,1],$
as desired. 
\end{proof}
The following proposition shows, in particular, that the log pair
$(\P_{\mathcal{O}_{\F}}^{1},\mathcal{D})$ featuring in the previous
lemma is the unique optimal model. 
\begin{prop}
\label{prop:optimal model for log P one} Consider a log pair $(\P_{\F}^{1},\Delta_{\F})$
over a number field $\F$ with coefficients $w_{i}\in[0,1]$ such
that $\pm K_{(\P_{\F}^{1},\Delta_{\F})}>0.$ Let $(\P_{\mathcal{O}_{\F}}^{1},\mathcal{D};\pm\mathcal{K}_{(\P_{\mathcal{O}_{\F}}^{1},\mathcal{D})})$
be a relatively ample model for $(\P_{\F}^{1},\Delta_{\F};\pm K_{(\P_{\F}^{1},\Delta_{\F})})$
over $\mathcal{O}_{\F}$ satisfying the following conditions: 
\begin{itemize}
\item When $K_{(\P_{\F}^{1},\Delta_{\F})}>0:$ $(\P_{\F_{b}}^{1},\mathcal{D}_{|\P_{\F_{b}}^{1}})$
is log canonical for any closed point $b\in\text{Spec}\ensuremath{\mathcal{O}_{\F}},$
where $\F_{b}$ denotes the residue field of $b$ (i.e. the coefficients
of $\mathcal{D}_{|\P_{\F_{b}}^{1}}$ are in $[0,1])$ or equivalently:
$1$ and $2$ in formula \ref{eq:cond one and two} holds.
\item When $-K_{(\P_{\F}^{1},\Delta_{\F})}>0:$ $\mathcal{D}$ is the Zariski
closure of the divisor on $\P_{\F}^{1}$ supported on $\{0,1,\infty\}$
having the same coefficients as $\Delta_{\Q}$ and $(\P^{1},\Delta)(\C)$
is K-semistable (i.e. the weight conditions \ref{eq:weight cond intro}
hold)
\end{itemize}
Then 
\begin{equation}
\mathcal{M}_{(\P_{\mathcal{O}_{\F}}^{1},\mathcal{D})}(\pm\mathcal{K}_{(\P_{\mathcal{O}_{\F}}^{1},\mathcal{D})})\leq\mathcal{M}_{(\mathcal{X}',\mathcal{D}')}(\mathcal{L}')\label{eq:ineq prop var pr mod}
\end{equation}
for any model $(\mathcal{X}',\mathcal{D}';\mathcal{L}')$ of $(\P_{\F}^{1},\Delta_{\F};\pm K_{(\P_{\F}^{1},\Delta_{\F})})$
over $\mathcal{O}_{\F}$ such that $\mathcal{L}'$ is relatively ample.
When $(\P_{\C}^{1},\Delta_{\C})$ is K-stable (which is automatic
if $K_{(\P_{\F}^{1},\Delta_{\F})}>0)$ equality holds in \ref{eq:ineq prop var pr mod}
iff $(\mathcal{X}',\mathcal{D}')=(\P_{\mathcal{O}_{\F}}^{1},\mathcal{D})$
and $\mathcal{L}'=\pm\mathcal{K}_{(\P_{\mathcal{O}_{\F}}^{1},\mathcal{D})}+\pi^{*}M$
for some line bundle $M\rightarrow\text{Spec \ensuremath{\mathcal{O}_{F}} .}$ 
\end{prop}

\begin{proof}
The case when $-K_{(\P_{\F}^{1},\Delta_{\F})}>0$ is shown in \cite[Section 6]{a-b2}
when $\F=\Q$ (see \cite[Remark 6.5]{a-b2} for the equality case)
and the proof in the general case is essentially the same. In the
case $K_{(\P_{\F}^{1},\Delta_{\F})}>0$ the inequality follows from
proposition \ref{prop:cond one and two}. We thus assume that equality
holds. Then it follows from the proof of proposition \ref{prop:cond one and two}
that $q^{*}\mathcal{L}'\cdot E=0,$ which means that $E^{2}=0.$ Since
$E$ is a vertical divisor and $\mathcal{X}$ is regular this can
only happen if $E=\sum_{b}\lambda_{b}\mathcal{Y}_{b}$ for some $\lambda_{b}\in\R.$
It follows that $E=0$, since $E$ is $p-$exceptional. This means
that $p$ is an isomorphism and thus $\mathcal{Y}\simeq\mathcal{X}.$
Next, since all the fibers of $\mathcal{X}(=\P_{\mathcal{O}_{\F}}^{1})$
over $\text{Spec}\ensuremath{\mathcal{O}_{\F}}$ are reduced and irreducible
it then follows that $q$ is also an isomorphism, $\mathcal{X}'\simeq\P_{\mathcal{O}_{\F}}^{1}$
and $\mathcal{L}'\simeq\mathcal{L}+\pi^{*}M,$ if $\mathcal{X}'$
is normal. The vanishing of the right hand side in formula \ref{eq:difference Mab in pf var pr}
then forces $\mathcal{D}'=\mathcal{D},$ as desired (since $\mathcal{K}_{\mathcal{X}^{'}/\mathcal{X}}-p^{*}\mathcal{D}+q^{*}\mathcal{D}'$
is effective, as shown in the proof of proposition \ref{prop:cond one and two}).
\end{proof}
\begin{rem}
\label{rem:local ineq}The proof of the previous proposition shows
that, in fact, $m(b)\geq0$ for any closed point $b,$ where $m(b)$
is the number appearing in formula \ref{eq:difference Mab in pf var pr}.
Moreover, if $m(b)=0$ for all $b$ in a open subset $U$ of $\text{Spec}\mathcal{O}_{\F},$
then $\mathcal{X}'$ is isomorphic to $\P_{\mathcal{O}_{\F}}^{1}$
over $U$ and, under such an isomorphism, $\mathcal{D}'=\mathcal{D}$
over $U.$
\end{rem}

In the light of the discussion following Remark \ref{rem:Manin} it
seems natural to pose the following conjecture:
\begin{conjecture}
Given a number field $\F$ and a log pair $(X_{\F},D_{\F})$ such
that $\pm K_{(X_{\F},\Delta_{\F})}$ is ample
\[
\inf_{(\mathcal{X},\mathcal{D};\mathcal{L})}\mathcal{\hat{M}}_{(\mathcal{X},\mathcal{D})}(\mathcal{L})=\inf_{(\mathcal{X},\mathcal{D})}\mathcal{\hat{M}}_{(\mathcal{X},\mathcal{D})}(\pm\mathcal{K}_{(\mathcal{X},\mathcal{D}})
\]
 where $(\mathcal{X},\mathcal{D};\mathcal{L})$ ranges over all polarized
models of $(X_{\F},\Delta_{\F};\pm K_{(X_{\F},\Delta_{\F})})$ over
$\mathcal{O}_{\F}$ and $(\mathcal{X},\mathcal{D})$ ranges over all
models of $(X_{\F},\Delta_{\F})$ such that $\pm\mathcal{K}_{(\mathcal{X},\mathcal{D}}$
is relatively ample. 
\end{conjecture}

Assuming the validity of this conjecture and the existence of optimal
models $(\mathcal{X}^{o},\mathcal{D}^{o})$ (defined in section \ref{subsec:Optimal-models})
for sufficiently large field extensions, it follows that the normalized
modular invariant $\mathcal{\hat{M}}(X_{\F},D_{\F};\pm K_{(X_{\F},\Delta_{\F})})$
coincides with $\pm h_{\text{can}}(\mathcal{X}^{o},\mathcal{D}^{o}).$
For example, the previous conjecture holds for the log pairs $(\P_{\F}^{1},\Delta_{\F})$
appearing in the previous proposition. 
\begin{rem}
\label{rem:globally K}According to a conjecture of Odaka \cite{o},
the infimum of $\mathcal{\hat{M}}_{\mathcal{X}}(\mathcal{L})$ over
$(\mathcal{X},\mathcal{L})$ is attained at any globally K-semistable
model, i.e. $(\mathcal{X},\mathcal{L}),$ i.e. a model all whose fibers
over $\text{Spec \ensuremath{\mathcal{O}_{\F}}}$ are K-semistable
(see \cite{h-o} for recent progress on this conjecture). In particular,
if this conjecture holds, then any globally K-semistable model of
the form $(\mathcal{X},\pm\mathcal{K}_{\mathcal{X}})$ is optimal
in the sense of Section \ref{subsec:Optimal-models}.
\end{rem}

\subsubsection{Relatively minimal models }

Consider a non-singular projective curve $X$ over a number field
$\F$ such that $K_{X}>0.$ It admits a unique regular model $\mathcal{X}_{\text{min}}$
over $\mathcal{O}_{\F}$ which is \emph{minimal,} or equivalently:\emph{
relatively minimal} (obtained by repeatedly blowing down vertical
$(-1)-$curves). $\mathcal{K}_{\mathcal{X}_{\text{min}}}$ is nef.
Contracting the vertical $(-2)-$curves in $\mathcal{X}_{\text{min}}$
yields a birational morphism from $\mathcal{X}_{\text{min}}$ to a
projective normal scheme over $\mathcal{O}_{\F},$ called the \emph{canonical
model }$\mathcal{X}_{\text{can}}$ of $X$ (not to be confused with
the canonical model of a Shimura curve). $\mathcal{X}_{\text{can}}$
is Gorenstein and $\mathcal{K}_{\mathcal{X}_{\text{can}}}$ is relatively
ample \cite[Cor 4.18]{l-e}.
\begin{prop}
\label{prop:minimal model}Let $X_{\F}$ be a non-singular projective
curve over $\F$ such that $K_{X_{\F}}>0.$ Then 
\[
\mathcal{M}_{\mathcal{X}_{\text{can}}}(\mathcal{K}_{\mathcal{X}_{\text{can}}})=\mathcal{M}_{\mathcal{X}_{\text{min}}}(\mathcal{K}_{\mathcal{X}_{\text{min}}})\leq\mathcal{M}_{\mathcal{X}'}(\mathcal{L}')
\]
 for any relatively ample model $(\mathcal{X}',\mathcal{L}')$ of
$(X_{\F},K_{X_{\F}})$ over $\text{Spec}\thinspace\ensuremath{\mathcal{O}_{\F}}.$
In particular, $\mathcal{X}_{\text{can}}$ is an optimal model for
$X_{\F}$ (in the sense of section \ref{subsec:Optimal-models}).
\end{prop}

\begin{proof}
This follows from results outlined in \cite{o}. For completeness
we provide a proof. Denoting by $\nu:\mathcal{X}'_{\nu}\rightarrow\mathcal{X}'$
the normalization of $\mathcal{X}',$ one first observes that $\mathcal{M}_{\mathcal{X}'_{\nu}}(\nu^{*}\mathcal{K}_{\mathcal{X}'_{\nu}})\leq\mathcal{M}_{\mathcal{X}'}(\mathcal{K}_{\mathcal{X}'}),$
using that the Weil divisor $\mathcal{K}_{\mathcal{X}'_{\nu}/\mathcal{X}'}$
is anti-effective (just as in the proof of\cite[Prop 2.17]{o}). Next,
fixing a birational morphism $f:\,\mathcal{X}\rightarrow\mathcal{X}'_{\nu}$
from a regular arithmetic surface $\mathcal{X},$ yields, by \ref{eq:pullback Mab},
$\mathcal{M}_{\mathcal{X}'_{\nu}}(\nu^{*}\mathcal{K}_{\mathcal{X}'_{\nu}})=\mathcal{M}_{\mathcal{X}}(\mathcal{L}),$
where $\mathcal{L}$ is the relatively semi-ample line bundle on $\mathcal{X}$
obtained by pulling back $\nu^{*}\mathcal{K}_{\mathcal{X}'_{\nu}}$
to\emph{ $\mathcal{X}.$} Finally, it follows from \cite[Thm 2.20]{o}
that 
\begin{equation}
\mathcal{M}_{\mathcal{X}_{\text{min}}}(\mathcal{K}_{\mathcal{X}_{\text{min}}})\leq\mathcal{M}_{\mathcal{X}}(\mathcal{L}).\label{eq:ineq M for X min}
\end{equation}
by running the Minimal Model Program with scaling \cite[Thm 2.20]{o}.
For completeness we detail the proof of the inequality \ref{eq:ineq M for X min}
in the present setup. Let $\mathcal{L}^{(0)}$ be a relatively nef
line bundle over a regular model $\mathcal{X}^{(0)}$ of $X$ and
set $\mathcal{L}_{t}^{(0)}:=\mathcal{L}^{(0)}+t(\mathcal{K}_{\mathcal{X}^{(0)}}-\mathcal{L}^{(0)}).$
A direct computation reveals that
\[
d\mathcal{M}_{\mathcal{X}^{(0)}}(\mathcal{L}_{t}^{(0)})/dt=(\mathcal{K}_{\mathcal{X}^{(0)}}-\mathcal{L}^{(0)})\cdot(\mathcal{K}_{\mathcal{X}^{(0)}}-\mathcal{L}^{(0)})\leq0,
\]
 using in the last step that $\mathcal{K}_{\mathcal{X}^{(0)}}-\mathcal{L}^{(0)}$
is a vertical divisor on a regular arithmetic surface. In particular,
if $\mathcal{K}_{\mathcal{X}^{(0)}}$ is relatively nef, then we are
done, since then $\mathcal{K}_{\mathcal{X}^{(0)}}$ is the unique
relative minimal model ($\mathcal{X}^{(0)}$ cannot contain any $(-1)-$curves,
by adjunction). Otherwise, denote by $t_{0}$ the sup over all $t\in[0,1]$
such that $\mathcal{L}_{t}^{(0)}$ is relatively nef. By assumption,
$t_{0}\in[0,1[.$ Note that the cone of all effective vertical divisors
on $\mathcal{X}^{(0)},$ modulo numerical equivalence, is generated
by a finite number of extremal effective divisors $C_{i}$ (for elementary
reasons; see \cite[Lemma 2.13]{tan} for a more general statement).
As a consequence, $t_{0}$ is rational and there exists an extremal
effective divisor $C$ such that $\mathcal{L}_{t_{0}}^{(0)}\cdot C=0$
and $\mathcal{K}_{\mathcal{X}^{(0)}}\cdot C<0$ (just as in \cite[Lemma 3.10.8]{bchm}).
It follows that there exists a birational morphism 
\[
f_{0}:\mathcal{X}^{(0)}\rightarrow\mathcal{X}^{(1)}
\]
 to a regular model $\mathcal{X}^{(1)}$ of $X$ that contracts precisely
$C.$ Indeed, in general, $C$ induces an extremal contraction $f:\mathcal{X}\rightarrow\mathcal{Z}$
and since $K_{X}>0$ this contraction is birational and $f(C)$ is
$0-$dimensional (see \cite[Thm 4.4]{tan} and \cite[Remark 3.10.9]{bchm}).
Thus, by \cite[Thm 1.5]{ko}, there exists a birational morphism $f_{0}:\mathcal{X}^{(0)}\rightarrow\mathcal{X}^{(1)}$
with $\mathcal{X}^{(1)}$ regular, mapping $C$ to closed point $x$
on $\mathcal{X}^{(1)}$ and such that $f_{0}$ restricts to an isomorphism
from $\mathcal{X}^{(0)}-C$ to $\mathcal{X}^{(1)}-\{x\}.$ Denote
by $\mathcal{L}^{(1)}$ the $\Q-$line bundle on $\mathcal{X}^{(1)}$
defined as the push-forward of $\mathcal{L}_{t_{0}}^{(0)}$ under
$f_{0}.$ It satisfies $f_{0}^{*}\mathcal{L}^{(1)}=\mathcal{L}_{t_{0}}^{(0)}$
and thus $\mathcal{M}_{\mathcal{X}^{(0)}}(\mathcal{L}_{t_{0}}^{(0)})=\mathcal{M}_{\mathcal{X}^{(1)}}(\mathcal{L}^{(1)}),$
by \ref{eq:pullback Mab}. Since $\mathcal{L}^{(1)}$ is relatively
nef we can repeat this procedure in a finite number of steps until
$\mathcal{K}_{\mathcal{X}^{(j)}}$ is nef (using that the Picard number
decreases at each step) and thus $\mathcal{X}^{(j)}$ is the unique
relative minimal model of $X.$ This proves the inequality \ref{eq:ineq M for X min}.
Finally, since the canonical morphism $F:\,\mathcal{X}_{\text{min}}\rightarrow\mathcal{X}_{\text{can}}$
satisfies $F^{*}\mathcal{K}_{\mathcal{X}_{\text{can}}}=\mathcal{K}_{\mathcal{X}_{\text{\ensuremath{\text{min}}}}}$
\cite[Cor 4.18]{l-e} this concludes the proof of the proposition. 
\end{proof}
\begin{rem}
The previous proposition yields an alternative proof of Cor \ref{cor:stable model}.
Indeed, if $X$ admits a semistable model over $\mathcal{O}_{\F}$
(in the sense of Deligne-Mumford), then the regular minimal model
$\mathcal{X}_{\text{min}}$ of $X$ over $\mathcal{O}_{\F}$ is also
semistable and $\mathcal{X}_{\text{can}}$ is the stable model \cite[Thm 10.3.34]{l-e}. 
\end{rem}

\section{\label{subsec:Variations-of-}Variations of the canonical height
with respect to the coefficients of $\mathcal{D}$}

We will say that a set of log pairs $\left\{ (\mathcal{X},\mathcal{D})\right\} $
is a \emph{linear family }if $\mathcal{X}$ and the irreducible components
of $\mathcal{D}$ are fixed and $\mathcal{K}_{(\mathcal{X},\mathcal{D})}$
is proportional to one and the same relatively ample line bundle,
i.e. $\mathcal{K}_{(\mathcal{X},\mathcal{D})}\cong s\mathcal{L}_{0}$
for some $s\in\R$ (depending on the coefficients $w_{i}$ of $\mathcal{D})$
and some relatively ample line bundle $\mathcal{L}_{0}$ (independent
of $w_{i})$. 
\begin{prop}
\label{prop:conc}Let $\left\{ (\mathcal{X},\mathcal{D})\right\} $
be a linear family of log canonical (lc) pairs. Then $\pm\hat{h}_{\text{can }}(\pm\mathcal{K}_{(\mathcal{X},\mathcal{D})})$
is concave wrt the coefficients $\boldsymbol{w}\in\R^{m}$ of $\Delta$
(assuming that the sign is chosen so that $\pm\mathcal{K}_{(\mathcal{X},\mathcal{D})}$
is relatively ample). In particular, $\pm\hat{h}_{\text{can }}(\pm\mathcal{K}_{(\mathcal{X},\mathcal{D})})$
is continuous wrt $\boldsymbol{w}$ in the interior $\dot{C}$ of
the convex set $C\Subset\R^{m}$ of all $\boldsymbol{w}$ for which
$\pm\hat{h}_{\text{can }}(\pm\mathcal{K}_{(\mathcal{X},\mathcal{D})})$
is finite. Furthermore, $\pm\hat{h}_{\text{can }}(\pm\mathcal{K}_{(\mathcal{X},\mathcal{D})})$
is continuous along any affine segment $I$ in $C,$ homeomorphic
to $]0,1],$ if the interior of $I$ is contained in $\dot{C}.$ 
\end{prop}

\begin{proof}
By the variational principles in Prop \ref{prop:var princi metrics}
and Prop \ref{prop:inf Mab for log CY} we can, express
\begin{equation}
\pm\hat{h}_{\text{can }}(\pm\mathcal{K}_{(\mathcal{X},\mathcal{D})})=\inf_{\psi_{0}}\mathcal{\hat{M}}_{(\mathcal{X},\mathcal{D})}(\pm\mathcal{K}_{(\mathcal{X},\mathcal{D})},s\psi_{0}),\label{eq:inf in Prop concave}
\end{equation}
where $\psi_{0}$ ranges over all psh metrics on $L_{0}(\C)$ of finite
energy. Indeed, since we have assumed $\mathcal{K}_{(\mathcal{X},\mathcal{D})}\cong s\mathcal{L}_{0},$
any psh metric $\psi$ on $\pm\mathcal{K}_{(X,\Delta)(\C)}$ may be
expressed as $\psi=s\psi_{0}$ for some psh metric $\psi_{0}$ on
$L_{0}(\C)$ (namely, $\psi_{0}:=\psi/s).$ Moreover, since we are
assuming that the infimum is finite we may as well assume that $\text{Ent \ensuremath{(\text{MA}(\ensuremath{\psi_{0})}}|\ensuremath{\mu_{s\psi_{0}})}}<\infty.$
To prove the concavity of $\pm\hat{h}_{\text{can }}(\pm\mathcal{K}_{(\mathcal{X},\mathcal{D})})$
wrt $\boldsymbol{w}$ it will thus be enough to show that for a \emph{fixed}
such psh metric $\psi_{0}$ on $L_{0},$ $\mathcal{\hat{M}}_{(\mathcal{X},\mathcal{D})}(\pm\mathcal{K}_{(\mathcal{X},\mathcal{D})},\pm s\psi_{0})$
is affine with respect to $\boldsymbol{w}.$ To this end we will exploit
the expression \ref{eq:def norm Mab for finite energy}. Setting $\phi=\pm s\psi_{0}$
we have 
\begin{equation}
\text{MA}(\ensuremath{\pm s\psi_{0})=\text{MA}(\ensuremath{\psi_{0})}},\,\,\,\mu_{\phi}=e^{s\psi_{0}}|s_{1}|^{-2w_{1}}\ldots(\frac{i}{2})^{n^{2}}dz\wedge d\bar{z}\label{eq:MA and meaure in pf conca}
\end{equation}
 using homogeneity in the first equality and, in the second equality,
a local representation as in section, where $s_{i}$ are the sections
cutting out the irreducible components of $\mathcal{D}.$ Since since
$s$ is affine in $\boldsymbol{w}$ it we deduce that $\log\frac{\text{MA}(\ensuremath{s\psi_{0})}}{\mu_{s\psi_{0}}}$
is affine in $\boldsymbol{w},$ which implies that $\text{Ent \ensuremath{(\text{MA}(\ensuremath{s\psi_{0})}}|\ensuremath{\mu_{s\psi_{0}})}}$
is affine in $\boldsymbol{w}$. Finally, by homogeneity, 
\[
\pm h_{\pm s\psi_{0}}(\pm\mathcal{K}_{(\mathcal{X},\mathcal{D})})=sh_{\psi_{0}}(\mathcal{L}_{0}).
\]
which concludes the proof that $\mathcal{\hat{M}}_{(\mathcal{X},\mathcal{D})}(\pm\mathcal{K}_{(\mathcal{X},\mathcal{D})},\pm s\psi_{0})$
is affine, showing that $\hat{h}_{\text{can }}(\pm\mathcal{K}_{(\mathcal{X},\mathcal{D})})$
is concave. Since any convex functions is continuous on a open subset
where it is finite it follows that $\hat{h}_{\text{can }}(\pm\mathcal{K}_{(\mathcal{X},\mathcal{D})})$
is continuous in $\dot{C}.$ The last continuity statement in the
proposition also follows from elementary properties of convex functions
(see \cite[Lemma 2.10]{a-b}). 
\end{proof}
\begin{rem}
\label{rem:not volume-norm}If one were to instead metrize $K_{(X,\Delta)}$
with the metric induced by the volume form $\omega_{\text{KE}}^{n}/n!$
of the Kähler-Einstein metric $\omega_{\text{KE }}$ (without normalizing
the volume), then the corresponding normalized height $\pm\hat{h}(\pm\mathcal{K}_{(\mathcal{X},\mathcal{D})})$
would always diverge as $K_{(X,\Delta)}$ approaches the trivial line
bundle. Indeed, by the scaling relation \ref{eq:normal height under scaling of metric}
\[
\pm\hat{h}(\pm\mathcal{K}_{(\mathcal{X},\mathcal{D})})=\pm\hat{h}_{\text{can}}(\pm\mathcal{K}_{(\mathcal{X},\mathcal{D})})\pm\frac{1}{2}\log\frac{c_{1}(\pm K_{(X,\Delta)})^{n}}{n!},
\]
 where the second term diverges as $K_{(X,\Delta)}$ approaches the
trivial line bundle. 
\end{rem}

\begin{prop}
\label{prop:real anal}Let $\left\{ (\mathcal{X},\mathcal{D}_{\boldsymbol{w}})\right\} $
be a linear family of log smooth klt pairs such that $\pm K_{(X,\Delta_{\boldsymbol{w}})}$
is K-stable and assume that the coefficients $\boldsymbol{w}$ range
over an open subset $G$ of $\R^{m}.$ Then $\pm\hat{h}_{\text{can }}(\pm\mathcal{K}_{(\mathcal{X},\mathcal{D}_{\boldsymbol{w}})})$
is real-analytic wrt $\boldsymbol{w}$ in $G.$ 
\end{prop}

\begin{proof}
By Hartog's classical theorem on separate holomorphicity it is enough
to consider the case when $\left\{ (\mathcal{X},\mathcal{D}_{\boldsymbol{w}})\right\} $
is a one-parameter family: i.e. $\boldsymbol{w}$ depends linearly
on a parameter $t\in]0,1[.$ To simplify the notation we assume that
$\F=\Q$ so that there is only one complex embedding $\sigma$ of
$\F.$ But the proof in the general case is essentially the same.
We will write $X=X(\C)$ and $L=L(\C).$ As recalled in Section \ref{subsec:K-stability},
the K-stability assumption is equivalent to the existence of a unique
volume-normalized Kähler-Einstein metric $\phi_{t}$ on $\pm K_{(X,\Delta_{\boldsymbol{w}(t)})}$
(which, by \cite{da-r} and \cite[Thm 2.19]{b-b-j}, is equivalent
to the properness of the Mabuchi functional appearing in formula \ref{eq:arith Mab in terms of Mab}).
Expressing $\phi_{t}=\pm s\psi$ for $\psi$ a metric on $L_{0},$
just as in the proof of the previous proposition (where $\psi$ was
denoted $\psi_{0})$, the Kähler-Einstein equation \ref{eq:KE eq}
for $\phi_{t}$ translates (using \ref{eq:MA and meaure in pf conca})
into 
\begin{equation}
\text{MA}(\ensuremath{\psi)}=e^{s(t)\psi}|s_{1}|^{-2w_{1}(t)}\ldots(\frac{i}{2})^{n^{2}}dz\wedge d\bar{z}.\label{eq:MA eq pf real ana}
\end{equation}
 Note that the right hand side of this equation depends real-analytically
on ($t,\psi),$ since $s(t)$ and $w_{i}(t)$ depend linearly on $t.$
Hence, assuming that one can apply the implicit function theorem in
an appropriate Banach space, the real-analyticity of $\pm\hat{h}_{\text{can }}(\pm\mathcal{K}_{(\mathcal{X},\mathcal{D})})$
wrt $t$ then follows precisely as in the proof of \cite[Thm 7.9]{ber2}
(where a different family of twisted Kähler-Einstein equations was
considered, where the role $s$ is played by $\beta$ and $\Delta=0).$
Finally, as explained in\cite[Section 2.4.3]{berm6b}, the implicit
function theorem can indeed be applied under the assumptions of the
proposition, using the theory for linearizations of equations of the
form \ref{eq:MA eq pf real ana}, established in \cite{do2,jmr} in
the case that the components of $\Delta$ do not intersect and announced
in \cite{m-r} in the log smooth case.
\end{proof}
\begin{example}
\label{exa:linear ex a la Donaldson}Let $\mathcal{X}$ be an arithmetic
Fano variety and $\mathcal{D}_{1}$ a divisor cut out by an element
in $H^{0}(\mathcal{X},-\mathcal{K}_{\mathcal{X}}).$ Then, for $w\in\R,$
$(\mathcal{X},w\mathcal{D}_{1})$ is a linear family as above with
$s=1-w$ and $\mathcal{L}_{0}=-\mathcal{K}.$ When $\Delta_{1}$ is
defined by an irreducible non-singular hypersurface it follows from
a conjecture of Donaldson, established in \cite{berm1}, that the
log Fano variety $(X,(1-s)\Delta_{1})$ admits a unique Kähler-Einstein
metric $\omega_{s}$ for any sufficiently small positive number $s$
(corresponding to a psh metric $\phi_{s}$ on $\mathcal{L}$ of finite
energy). As a consequence, by the previous two propositions $\hat{h}_{\text{can }}(-\mathcal{K}_{(\mathcal{X},(1-s)\mathcal{D})})$
is real-analytic and concave for $s$ sufficiently small. As another
example (where $\mathcal{D}_{1}$ is not irreducible) let $\mathcal{X}$
be the canonical model over $\Z$ of a toric Fano variety $X_{\Q}$
and let $\mathcal{D}_{1}$ be the standard torus invariant anti-canonical
divisor on $\mathcal{X}.$ Then $(X,(1-s)\Delta_{1})$ is K-polystable
for any $s\in[0,1]$ and, by \cite[Lemma 3.2]{a-b2}, $-\hat{h}_{\text{can }}(-\mathcal{K}_{(\mathcal{X},(1-s)\mathcal{D})})=-\hat{h}_{\text{can }}(-\mathcal{K}_{\mathcal{X}})+\frac{n}{2}\log s,$
which is, indeed concave wrt $s$ and continuous as $s\rightarrow1$
and $s\rightarrow0$ (since the log Calabi-Yau $(\mathcal{X},\mathcal{D})$
is not klt) in accordance with Prop \ref{prop:conc}. However, while
$-\hat{h}_{\text{can }}(-\mathcal{K}_{(\mathcal{X},(1-s)\mathcal{D})})$
is real-analytic wrt $s,$ this does not follow from Prop \ref{prop:real anal},
since $(X,(1-s)\Delta)$ is not K-stable in this case (but it seems
likely that the real-analyticity could be deduced from a generalization
of Prop \ref{prop:real anal} taking a maximal compact subgroup of
the automorphism group of $X$ into account).
\end{example}

\section{Canonical heights in terms of periods}

We start with some notation. Given a $\Q-$divisor $\Delta$ on a
complex projective variety $X$ and a positive integer $N$ we will
use the same notation $\Delta$ for the divisor on the $N-$fold product
$X^{N}$ of $X$ defined as the sum of the $N$:th pull backs of the
divisor $\Delta$ on $X$ under the $N$ projections onto the different
factors of $X^{N}.$ We will denote by $s_{\Delta}$ the corresponding
(multi-valued) holomorphic section over $X_{\text{reg}}^{N}$ cutting
out the restriction of $\Delta$ to $X_{\text{reg}}^{N}$ (where $X_{\text{reg}}$
denotes the regular locus of $X).$

\subsection{The case $K_{(X,\Delta)}>0$}

Let $(\mathcal{X},\mathcal{D})$ be an arithmetic log pair over $\mathcal{O}_{\F}$
such that that $\mathcal{K}_{(\mathcal{X},\mathcal{D})}$ is a relatively
ample $\Q-$line bundle over $\mathcal{X}$ and assume that $(X_{\Q},\Delta_{\Q})$
is klt. Given a positive real number $k$ such that $k\mathcal{K}_{(\mathcal{X},\mathcal{D})}$
is a line bundle (i.e. Cartier) denote by $N_{k}$ the rank of the
$\mathcal{O}_{\F}-$module $H^{0}(\mathcal{X},k\mathcal{K}_{(\mathcal{X},\mathcal{D})}):$
\[
N_{k}:=\dim_{\C}\left(H^{0}(\mathcal{X},k\mathcal{K}_{(\mathcal{X},\mathcal{D})})\otimes_{\sigma}\C\right)
\]
 for any embedding $\sigma$ of $\F$ into $\C$ (the subscript $k$
will occasionally be omitted to simplify the notation). The exterior
power $\Lambda^{N_{k}}\left(H^{0}(\mathcal{X},k\mathcal{K}_{(\mathcal{X},\mathcal{D})})\right)$
thus has rank one. We fix a non-trivial element in $\Lambda^{N_{k}}\left(H^{0}(\mathcal{X},k\mathcal{K}_{(\mathcal{X},\mathcal{D})})\right)$
that we note by $\det S^{(k)}.$ For example, $\det S^{(k)}$ can
be taken to be the $N_{k}-$fold exterior product of any $N_{k}$
elements $s_{1}^{(k)},...,s_{N_{k}}^{(k)}$ in $H^{0}(\mathcal{X},k\mathcal{K}_{(\mathcal{X},\mathcal{D})})$
that define a basis in $H^{0}(\mathcal{X},k\mathcal{K}_{(\mathcal{X},\mathcal{D})})\otimes_{\sigma}\C.$
Under the standard natural embedding of $\Lambda^{N_{k}}\left(H^{0}(X,k\mathcal{K}_{(X,\Delta)_{\sigma}})\right)$
into $H^{0}(\mathcal{X}^{N},k\mathcal{K}_{(X^{N},\Delta)_{\sigma}})$
we can identify the complexifications of $\det S^{(k)}$ with a holomorphic
section of $kK_{(X^{N_{k}},\Delta_{k})}$:

\begin{equation}
(\det S_{\sigma}^{(k)})(x_{1},x_{2},...,x_{N_{k}})=\det(s_{i}^{(k)}(x_{j})),\label{eq:def of Slater}
\end{equation}
Thus 
\[
\alpha_{\sigma}^{(k)}:=\left(\det S_{\sigma}^{(k)}\right)^{1/k}\otimes s_{\Delta}^{-1}
\]
 defines a multivalued meromorphic top form on $X^{N_{k}}$ (i.e.
a multivalued meromorphic section of $K_{X^{N_{k}}}).$ Set 
\[
\mathcal{Z}_{N_{k}}(\mathcal{X},\mathcal{D})_{\sigma}:=(\frac{i}{2})^{(N_{k}n)^{2}}\int_{X_{\sigma}^{N_{k}}}\alpha_{\sigma}^{(k)}\wedge\overline{\alpha_{\sigma}^{(k)}}\in\R,\,\,\,\,\,\mathcal{Z}_{N_{k}}(\mathcal{X},\mathcal{D}):=\prod_{\sigma}\mathcal{Z}_{N_{k}}(\mathcal{X},\mathcal{D})_{\sigma}
\]
 (the klt assumption ensures that $\mathcal{Z}_{N_{k}}(\mathcal{X},\mathcal{D})_{\sigma}<\infty).$
The product
\begin{equation}
\left(\sharp\frac{\Lambda^{N_{k}}\left(H^{0}(\mathcal{X},k\mathcal{K}_{(\mathcal{X},\mathcal{D})})\right)}{\mathcal{O}_{\F}(\det S^{(k)})}\right)^{1/k}\cdot\mathcal{Z}_{N_{k}}(\mathcal{X},\mathcal{D})\label{eq:product}
\end{equation}
 is, for any given $k,$ an invariant of $(\mathcal{X},\mathcal{D}),$
as follows directly from the product formula in $\F.$ 
\begin{rem}
\label{rem:free over Z}If $\F=\Q$ then $\Lambda^{N_{k}}\left(H^{0}(\mathcal{X},k\mathcal{K}_{(\mathcal{X},\mathcal{D})})\right)$
is a free $\Z-$module of rank one and thus taking $\det S^{(k)}$
to be a generator of $\Lambda^{N_{k}}\left(H^{0}(\mathcal{X},k\mathcal{K}_{(\mathcal{X},\mathcal{D})})\right)$
eliminates the first factor in the product
\end{rem}

\begin{thm}
\label{thm:period K pos}Let $(\mathcal{X},\mathcal{D})$ be an arithmetic
log pair over $\mathcal{O}_{\F}$ such that $\mathcal{K}_{(\mathcal{X},\mathcal{D})}$
is a relatively ample $\Q-$line bundle over $\mathcal{X}$ and assume
that $(X_{\Q},\Delta_{\Q})$ is klt. Then 
\[
\hat{h}_{\text{can }}(\mathcal{K}_{(\mathcal{X},\mathcal{D})})=-\lim_{k\rightarrow\infty}\frac{1}{2N_{k}}\log\left(\left(\sharp\frac{\Lambda^{N_{k}}\left(H^{0}(\mathcal{X},k\mathcal{K}_{(\mathcal{X},\mathcal{D})})\right)}{\mathcal{O}_{\F}(\det S^{(k)})}\right)^{-2/k}\cdot\mathcal{Z}_{N_{k}}(\mathcal{X},\mathcal{D})\right).
\]
\end{thm}

\begin{proof}
Fix $\sigma$ and write $X=X_{\sigma}.$ For any fixed continuous
metric $\left\Vert \cdot\right\Vert $ on $K_{(X,\Delta)}$ with positive
curvature current we can express 
\[
\mathcal{Z}_{N_{k}}(\mathcal{X},\mathcal{D})_{\sigma}=\int_{X^{N_{k}}}\left\Vert \det S^{(k)}\right\Vert ^{2/k}dV^{\otimes N_{k}}.
\]
 where $dV$ denotes the measure on $X$ corresponding to the metric
$\left\Vert \cdot\right\Vert $ (using the additive notation $\phi_{0}$
for the metric $\left\Vert \cdot\right\Vert $ this means that $dV=\mu_{\phi_{0}}$
in the notation of Section \ref{subsec:Local-representations-of}).
Indeed, in general, given $s\in H^{0}(X,kK_{(Y,\Delta)})$ and a volume
form $dV$ on $Y$ we can, locally on $Y,$ express
\[
\left\Vert s\right\Vert ^{2/k}dV:=|s|^{2/k}e^{-\phi_{0}}\cdot(\frac{i}{2})^{(\dim Y)^{2}}e^{\phi_{0}}dz\wedge d\bar{z}=(\frac{i}{2})^{(\dim Y)^{2}}(s^{1/k}dz)\wedge(\overline{s^{1/k}dz}).
\]
Next, fix a basis in $H^{0}(X,k(K_{X}+\Delta))$ which is orthonormal
wrt the scalar product $\left\langle \cdot,\cdot\right\rangle $ on
$H^{0}(X,k(K_{X}+\Delta))$ induced by ($\left\Vert \cdot\right\Vert ,dV)$
and denote by $\det S_{0}^{(k)}$ the corresponding section of $K_{(X^{N_{k}},\Delta_{k})},$
defined as in formula \ref{eq:def of Slater}. By basic linear algebra
$\det S^{(k)}=\det_{i,j\leq N_{k}}\left\langle s_{i}^{(k)},s_{j}^{(k)}\right\rangle \det S_{0}^{(k)}.$
Hence, 
\[
\frac{1}{N_{k}}\log\mathcal{Z}_{N_{k}}=\sum_{\sigma}\frac{1}{kN_{k}}\log\det_{i,j\leq N_{k}}\left\langle s_{i},s_{j}\right\rangle _{\sigma}+\int_{X^{N_{k}}}\left\Vert \det S_{0}^{(k)}\right\Vert ^{2/k}dV^{\otimes N_{k}}.
\]
 By the arithmetic Hilbert-Samuel formula \cite{g-s2,Zh0}, 
\begin{equation}
\frac{1}{kN_{k}}\log\left(\prod_{\sigma}\det_{i,j\leq N_{k}}\left\langle s_{i},s_{j}\right\rangle _{\sigma}\sharp\left(\frac{\Lambda^{N_{k}}\left(H^{0}(\mathcal{X},k\mathcal{K}_{(\mathcal{X},\mathcal{D})})\right)}{\mathcal{O}_{\F}(\det S^{(k)})}\right)^{-2}\right)\rightarrow-2\hat{h}(\mathcal{K}_{(\mathcal{X},\mathcal{D})},\left\Vert \cdot\right\Vert )\label{eq:hs K pos}
\end{equation}
 as $k\rightarrow\infty.$ Next, by the large deviation principle
in \cite[Thm 1.1]{berm1b} for $X$ non-singular and $\Delta=0$ and
\cite[Thm 4.3]{berm1c}, in general: for any given metric $\left\Vert \cdot\right\Vert $
on $K_{(X,\Delta)},$ 
\begin{equation}
-\frac{1}{N_{k}}\log\int_{X^{N_{k}}}\left\Vert \det S_{0}^{(k)}\right\Vert ^{2/k}dV^{\otimes N_{k}}\rightarrow\inf_{\mathcal{\mu\in P}(X)}F_{1}(\mu).\label{eq:conv towards inf F one}
\end{equation}
where $F_{1}(\mu)$ is the free energy type functional defined in
formula \ref{eq:def of free}. Hence, combining \ref{eq:hs K pos}
and \ref{eq:conv towards inf F one} gives, using the identities \ref{eq:arith Mab in terms of Mab}
and \ref{eq:Mab as free}, 
\[
-\lim_{k\rightarrow\infty}\frac{1}{2N_{k}}\log\left(\left(\sharp\frac{\Lambda^{N_{k}}\left(H^{0}(\mathcal{X},k\mathcal{K}_{(\mathcal{X},\mathcal{D})})\right)}{\mathcal{O}_{\F}(\det S^{(k)})}\right)^{-2/k}\cdot\mathcal{Z}_{N_{k}}(\mathcal{X},\mathcal{D})\right)=\inf_{\phi}\mathcal{M}_{\phi}(\mathcal{K}_{(\mathcal{X},\mathcal{D})},\phi),
\]
 where the inf ranges over all finite energy metrics $\phi$ on $K_{(X,\Delta)(\C)}.$
Invoking the variational principle in Prop \ref{prop:var princi metrics}
thus concludes the proof. 
\end{proof}

\subsubsection{Intermezzo: the case when $\mathcal{K}_{(\mathcal{X},\mathcal{D})}$
is semi-ample and Faltings' height}

Before moving on to the log Fano case we note that $\mathcal{Z}_{N_{k}}(\mathcal{X},\mathcal{D})$
is well-defined as soon as $k\mathcal{K}_{(\mathcal{X},\mathcal{D})}$
is effective, i.e. $N_{k}\geq1.$ In particular, if $\mathcal{K}_{(\mathcal{X},\mathcal{D})}$
is semi-ample, then $\mathcal{Z}_{N_{k}}(\mathcal{X},\mathcal{D})$
is well-defined for $k$ sufficiently divisible. For example, when
$k\mathcal{K}_{(\mathcal{X},\mathcal{D})}$ is trivial $-\frac{1}{2N_{k}}\log\mathcal{Z}_{N_{k}}(\mathcal{X},\mathcal{D})$
coincides, by definition, with Faltings' height \ref{eq:Falting height}.
There is thus no need to let $k$ tend to infinity in this case. In
general, when $\mathcal{K}_{(\mathcal{X},\mathcal{D})}$ is semi-ample
the proof of Theorem \ref{thm:period K pos} reveals, together with
the results described in \cite[Section 5.2]{berm1c}, that 
\[
-\lim_{k\rightarrow\infty}\frac{1}{2N_{k}}\log\left(\left(\sharp\frac{\Lambda^{N_{k}}\left(H^{0}(\mathcal{X},k\mathcal{K}_{(\mathcal{X},\mathcal{D})})\right)}{\mathcal{O}_{\F}(\det S^{(k)})}\right)^{-2/k}\cdot\mathcal{Z}_{N_{k}}(\mathcal{X},\mathcal{D})\right)=h_{\phi_{\text{can}}}(\mathcal{K}_{(\mathcal{X},\mathcal{D})}),
\]
 where $\phi_{\text{can}}$ is the volume-normalized metric on $\mathcal{K}_{(\mathcal{X},\mathcal{D})}$
introduced in \cite{ts,s-t}, whose curvature form is the pull-back
to $X$ of a canonical twisted Kähler-Einstein metric on the canonical
model of $X$ over $\C$ (i.e. the Proj of the canonical ring of $X$). 

\subsection{The case $-K_{(X,\Delta)}>0$}

Let now $(\mathcal{X},\mathcal{D})$ be an arithmetic log pair over
$\mathcal{O}_{\F}$ such that $-\mathcal{K}_{(\mathcal{X},\mathcal{D})}$
is a relatively ample $\Q-$line bundle over $\mathcal{X}.$ Given
a positive real number $k$ such that $-k\mathcal{K}_{(\mathcal{X},\mathcal{D})}$
is a bona fide line bundle (i.e. Cartier) we can, after replacing
$k$ with $-k,$ proceed as before. More precisely, we set 
\[
N_{k}:=\dim_{\R}\left(H^{0}(\mathcal{X},-k\mathcal{K}_{(\mathcal{X},\mathcal{D})})\otimes\R\right)
\]
 and define
\[
\mathcal{Z}_{N_{k}}(\mathcal{X},\mathcal{D})_{\sigma}:=(\frac{i}{2})^{(N_{k}n)^{2}}\int_{X_{\sigma}^{N_{k}}}\alpha_{k}\wedge\overline{\alpha_{k}},\,\,\,\,\alpha_{k}:=\left(\det S^{(k)}\right)^{-1/k}\otimes s_{\Delta}
\]
 where $\alpha_{k}$ still defines a meromorphic top form on a Zariski
open subset of $X^{N_{k}}.$ We then define $\mathcal{Z}_{N_{k}}(\mathcal{X},\mathcal{D})$
as the product over $\sigma$ of $\mathcal{Z}_{N_{k}}(\mathcal{X},\mathcal{D})_{\sigma}.$
However, in this case $\mathcal{Z}_{N_{k}}(\mathcal{X},\mathcal{D})_{\sigma}$
may diverge (even if $\mathcal{D}=0)$.
\begin{thm}
\label{thm:period K neg n one}Assume that $-K_{(X,\Delta)}>0$ and
that $n=1.$ Then $\mathcal{Z}_{N_{k}}(\mathcal{X},\mathcal{D})$
is finite for $k$ sufficiently large iff $(X,\Delta)$ is K-stable.
Moreover,
\[
\hat{h}_{\text{can }}(-\mathcal{K}_{(\mathcal{X},\mathcal{D})})=\lim_{k\rightarrow\infty}\frac{1}{2N_{k}}\log\left(\mathcal{Z}_{N_{k}}(\mathcal{X},\mathcal{D})\left(\sharp\frac{\Lambda^{N_{k}}\left(H^{0}(\mathcal{X},k\mathcal{K}_{(\mathcal{X},\mathcal{D})})\right)}{\mathcal{O}_{\F}(\det S^{(k)})}\right)^{2/k}\right).
\]
\end{thm}

\begin{proof}
Proceeding as in the proof of Theorem \ref{thm:period K pos}, but
replacing $k$ with $-k$ yields, if $\mathcal{Z}_{N_{k}}<\infty,$
\[
\frac{1}{N_{k}}\log\mathcal{Z}_{N_{k}}(\mathcal{X},\mathcal{D})_{\sigma}=-\frac{1}{kN_{k}}\log\det_{i,j\leq N_{k}}\left\langle s_{i},s_{j}\right\rangle +\int_{X_{\sigma}^{N_{k}}}\left\Vert \det S_{0}^{(k)}\right\Vert ^{-2/k}dV^{\otimes N_{k}}
\]
 for any given metric $\left\Vert \cdot\right\Vert $ on $-K_{(X,\Delta)},$
where $dV$ denotes the corresponding measure on $X.$ Now assume
that $n=1$ and $(X,\Delta)$ is K-stable. By \cite[Thm 4.1]{berm6b}
this equivalently means that $\mathcal{Z}_{N_{k}}<\infty$ for $k$
sufficiently large. Moreover, by \cite[Thm 4.4]{berm6b}, 
\[
-\frac{1}{N_{k}}\log\int_{X^{N_{k}}}\left\Vert \det S_{0}^{(k)}\right\Vert ^{2/k}dV^{\otimes N_{k}}\rightarrow\inf_{\mathcal{\mu\in P}(X)}F_{-1}(\mu)
\]
 Hence, invoking the Hilbert-Samuel formula and the identities \ref{eq:arith Mab in terms of Mab}
concludes the proof, precisely as in the case $K_{(X,\Delta)}>0.$ 
\end{proof}
For a general relative dimension $n$ a notion of \emph{Gibbs stability
}is introduced in \cite{berm1c}, which - in the arithmetic present
setup - amounts to the finiteness of $\mathcal{Z}_{N_{k}}(\mathcal{X},\mathcal{D})$
for $k$ sufficiently large. It is conjectured in \cite{berm1c,berm6b}
that $(\mathcal{X},\mathcal{D})$ is Gibbs stable iff $(X,\Delta)$
is K-stable (the ``only if'' direction is established in \cite{f-o}).
Moreover, under the following (a priori) stronger assumption: 
\begin{equation}
\left(\sharp\frac{\Lambda^{N_{k}}\left(H^{0}(\mathcal{X},k\mathcal{K}_{(\mathcal{X},\mathcal{D})})\right)}{\mathcal{O}_{\F}(\det S^{(k)})}\right)^{-\pm2/k}\cdot\mathcal{Z}_{N_{k}}(\mathcal{X},\mathcal{D})\leq C^{N_{k}}\label{eq:unif bound}
\end{equation}
it was pointed out in \cite{berm6b} that the convergence in Theorem
\ref{thm:period K neg n one} holds under a certain zero-free hypothesis,
discussed in the following section.
\begin{rem}
\label{rem:normal with pi}It is sometimes convenient to use a different
normalization, where $\alpha_{k}\wedge\overline{\alpha_{k}}$ is replaced
by $\pi{}^{-nN}\alpha_{k}\wedge\overline{\alpha_{k}}.$ These two
different normalizations are analogous to the two different normalizations
for Faltings' height of abelian varieties appearing in the literature
(\cite{fa00} vs. \cite{de}). Then the right hand side in Theorems
\ref{thm:period K pos}, \ref{thm:period K neg n one} gets replaced
by $\pm\hat{h}_{\text{can}}(\pm\mathcal{K}_{(\mathcal{X},\mathcal{D})})+\frac{n}{2}\log\pi.$
This is the height of $\pm\hat{h}_{\text{can}}(\pm\mathcal{K}_{(\mathcal{X},\mathcal{D})})$
computed wrt the Kähler-Einstein metric on $\pm K_{(X,\Delta)}$ giving
volume $\pi{}^{n}$ to $X.$ In all the explicit formulas that we
have been able to compute (e.g. Theorem \ref{thm:explicit intro})
this normalization has the effect of removing $\pi$ from the explicit
formulas.
\end{rem}

\subsection{Real-analyticity and the zero-free hypothesis}

Consider a linear family $(\mathcal{X},\mathcal{D}_{\boldsymbol{w}})$
of log pairs with coefficients $\boldsymbol{w}\in\R^{m},$ as defined
in Section \ref{subsec:Variations-of-}. This means that $\mathcal{K}_{(\mathcal{X},\mathcal{D}_{\boldsymbol{w}})}\cong s(\boldsymbol{w})\mathcal{L}_{0},$
where $s$ is an affine function of $\boldsymbol{w}.$ Given a positive
integer $l$ set $k:=ls^{-1}$ which is thus negative when $s<0.$
By definition, $k\mathcal{K}_{(\mathcal{X},\mathcal{D}_{\boldsymbol{w}})}\cong l\mathcal{L}_{0},$
giving 
\[
H^{0}(\mathcal{X},k\mathcal{K}_{(\mathcal{X},\mathcal{D}_{\boldsymbol{w}})})\cong H^{0}(\mathcal{X},l\mathcal{L}_{0}).
\]
 Hence, denoting by $N$ the dimension of $H^{0}(\mathcal{X},l\mathcal{L}_{0})\otimes\R$
and by $\det S$ the corresponding section over $X^{N}$ (both depending
only $l)$ we can express 
\[
\mathcal{Z}_{N}(\mathcal{X},\mathcal{D}_{\boldsymbol{w}})_{\sigma}:=(\frac{1}{2})^{(Nn)^{2}}\int_{X_{\sigma}^{N}}\left|\left(\det S_{\sigma}\right)^{s(\boldsymbol{w})/l}\otimes s_{1}^{w_{1}}\otimes\cdots s_{m}^{w_{m}}\right|^{2}.
\]
 For a fixed positive integer $l$ this function is manifestly real-analytic
(and log convex) wrt $\boldsymbol{w}\in\R^{m}$ in the open region
where $\mathcal{Z}_{N}(\mathcal{X},\mathcal{D})<\infty.$ More precisely,
allowing complex coefficients, $\boldsymbol{w}\in\C^{m}$, the corresponding
function $\mathcal{Z}_{N}(\mathcal{X},\mathcal{D}_{\boldsymbol{w}})$
is holomorphic in the tube domain in $\C^{m}$ over the open subset
$\{\mathcal{Z}_{N}(\mathcal{X},\mathcal{D}_{\boldsymbol{w}})<\infty\}\Subset\R^{m}.$
In \cite{berm6b} \emph{a ``zero-free hypothesis}'' is introduced,
which in the present arithmetic setup may be formulated as follows:

\begin{equation}
\exists\text{\ensuremath{\Omega\subset\C^{m}}}\text{:}\,\,\mathcal{Z}_{N}(\mathcal{X},\mathcal{D}_{\boldsymbol{w}})\neq0\,\text{in \ensuremath{\Omega,} }\label{eq:zerofree}
\end{equation}
where $\Omega$ is assumed to be a connected open subset of $\C^{m}$
independent of $N$ (i.e on $l)$ and contained in the tube-domain
$\{\mathcal{Z}_{N}(\mathcal{X},\mathcal{D}_{\boldsymbol{w}})<\infty\}.$ 
\begin{prop}
\label{prop:zerofree}Assume that $-\mathcal{K}_{(\mathcal{X},\mathcal{D})}$
is relatively ample and that the uniform bound \ref{eq:unif bound}
holds. If $(\mathcal{X},\mathcal{D})$ contained in a linear family
$(\mathcal{X},\mathcal{D}_{\boldsymbol{w}})$ containing some log
pair $(\mathcal{X},\mathcal{D}_{\boldsymbol{w}_{0}})$ such that $\mathcal{K}_{(\mathcal{X},\mathcal{D}_{\boldsymbol{w}_{0}})}$
is relatively ample, then the convergence in Theorem \ref{thm:period K neg n one}
holds under the condition that the zero-free hypothesis \ref{eq:zerofree}
holds. 
\end{prop}

\begin{proof}
This follows from arguments in \cite{berm6b}, which go as follows.
First, using basic properties of holomorphic functions and convexity,
after passing to a subsequence, the following limit holds uniformly
on compact subsets of $\Omega,$ as $l\rightarrow\infty:$ 
\[
-\frac{1}{2N_{k}}\log\left(\left(\sharp\frac{\Lambda^{N_{k}}\left(H^{0}(\mathcal{X},k\mathcal{K}_{(\mathcal{X},\mathcal{D})})\right)}{\mathcal{O}_{\F}(\det S^{(k)})}\right)^{-2s/l}\cdot\mathcal{Z}_{N}(\mathcal{X},\mathcal{D})\right)\rightarrow g(\boldsymbol{w})
\]
 for some holomorphic function $g$ on $\Omega$ (indeed, by assumption,
$g$ is a uniform limit of uniformly bounded holomorphic functions
on $\Omega$). But Theorem \ref{thm:period K pos} implies that $g=\pm\hat{h}(\pm\mathcal{K}_{(\mathcal{X},\mathcal{D}_{\boldsymbol{w}})})$
on $\Omega\cap\R^{m}\cap\{s>0\}.$ Hence, by uniqueness of real-analytic
extensions, it follows from Prop \ref{prop:real anal} that $g=\pm\hat{h}(\pm\mathcal{K}_{(\mathcal{X},\mathcal{D}_{\boldsymbol{w}})})=f$
on all of $\Omega\cap\R^{m},$ which concludes the proof. 
\end{proof}

\subsection{Synthesis on arithmetic log surfaces }

When $n=1$ combining Theorems \ref{thm:period K pos}, \ref{thm:period K neg n one}
yields:
\begin{thm}
\label{thm:periods n one both signs}Let $(\mathcal{X},\mathcal{D})$
be an arithmetic log pair over $\mathcal{O}_{\F}$ of relative dimension
one such that $\pm\mathcal{K}_{(\mathcal{X},\mathcal{D})}$ is a relatively
ample $\Q-$line bundle over $\mathcal{X}$ and assume that $(X_{\Q},\Delta_{\Q})$
is klt. Then 
\[
\pm\hat{h}_{\text{can }}(\pm\mathcal{K}_{(\mathcal{X},\mathcal{D})})=-\lim_{k\rightarrow\infty}\frac{1}{2N_{k}}\log\left(\left(\sharp\frac{\Lambda^{N_{k}}\left(H^{0}(\mathcal{X},\pm k\mathcal{K}_{(\mathcal{X},\mathcal{D})})\right)}{\mathcal{O}_{\F}(\det S^{(k)})}\right)^{-\pm2/k}\cdot\mathcal{Z}_{N_{k}}(\mathcal{X},\mathcal{D})\right).
\]
\end{thm}

\section{\label{sec:The-canonical-height}The canonical height of log pairs
on $\P_{\Z}^{1}$ and the Hurwitz zeta function}

In this section we will, in particular, prove Theorem \ref{thm:explicit intro}.
Recall that $\mathcal{D}^{o}$ denotes the divisor on $\P_{\Z}^{1}$
defined as the Zariski closure of the divisor $\Delta_{\Q}$ on $\P_{\Q}^{1}$
supported at $\{0,1,\infty\}$ with coefficients $\boldsymbol{w}=(w_{1},w_{2},w_{3})$
contained in the convex domain $C\subset\R^{3}$ defined by the weight
conditions \ref{eq:weight cond intro} (i.e. $(\P^{1},\Delta_{\boldsymbol{w}})$
is K-semistable). Denote by $f(\boldsymbol{w})$ the function 
\begin{equation}
f(\boldsymbol{w}):=\frac{1-\log(\pi\frac{V}{2})}{2}-\frac{\gamma(0,\frac{V}{2})-\sum_{i=1}^{3}\gamma(w_{i}-\frac{V}{2},w_{i})}{V},\,\,\,V:=-2\text{+}\sum_{i=1}^{3}w_{i}\label{eq:def of f text}
\end{equation}
defined in the interior of $C$ when $V>0,$ where $\gamma(a,b)$
is defined by formula \ref{eq:def of gamma} (note that $V$ is the
degree of $K_{(\P^{1},\Delta)}$.)
\begin{lem}
\label{lem:gamma}Given $a,b\in]0,1[,$ 
\[
\gamma(a,b)=\int_{a}^{b}\log l(x)\mathrm{d}x,\,\,\,\,l(x):=\frac{\Gamma(x)}{\Gamma(1-x)},\,\,\Gamma(x):=\int_{0}^{\infty}t^{x-1}e^{-t}dt.
\]
\end{lem}

\begin{proof}
The formula follows directly from the well-known fact that $\zeta(-1,t)+\zeta'(-1,t)+\frac{(t-1)}{2}\log(2\pi)$
is a primitive of $\log(\Gamma(t))$ on $]0,1[$ \cite[formula 3.11]{c-s}.
\end{proof}
The previous lemma reveals that $f$ is real-analytic when $V>0.$
Furthermore, Theorem \ref{thm:periods n one both signs} will imply
that $f$ extends real-analytically to all of the interior of $C.$
We extend $f$ to a finite function on the subset of the boundary
of $C$ where $V\neq0,$ by declaring its value to be the limit of
its values along any affine segment $I$ in the interior of $C$ reaching
the boundary. 

\subsection{The case $K_{(\P^{1},\Delta)}>0$}

It is enough to consider case when $\Delta$ is klt, i.e. $w_{i}<1,$
by the continuity in Prop \ref{prop:conc}. Since 
\[
k\mathcal{K}_{(\P_{\Z}^{1},\mathcal{D}^{o})}\simeq kV\mathcal{O}(1),\,\,\,V:=V(K_{(\P^{1},\Delta)})
\]
 the free $\Z-$module $H^{0}(\mathcal{X},k\mathcal{K}_{(\P_{\Z}^{1},\mathcal{D}^{o})})$
may be identified with the space of all homogeneous polynomials of
degree $kV$ on $\C^{2}$ with integer coefficients. We fix the standard
basis $s_{1},...,s_{N_{k}}$ of monomials in the latter free $\Z-$module
and denote by $\det S^{(k)}$ the corresponding generator of $\Lambda^{N_{k}}\left(H^{0}(\mathcal{X},k\mathcal{K}_{(\mathcal{X},\mathcal{D})})\right)$,
as in Remark \ref{rem:free over Z}. Denote by $p_{1},...,p_{m}$
the irreducible components of $\Delta$ and assume that $z=\infty$
at $p_{m},$ where $z$ denotes the affine coordinate on the standard
affine piece $\C$ of $\P_{\C}^{1}.$ Identifying $\det S^{(k)}$
with the Vandermonde determinant $\prod_{i<j\leq N_{k}}(z_{i}-z_{j})$
(where $N_{k}=kV+1)$ we can thus express
\begin{equation}
\mathcal{Z}_{N}=\int_{\C^{N}}\left(\prod_{i\neq j}\left|z_{i}-z_{j}\right|\right)^{\frac{V}{N-1}}\prod_{i\leq N,j\leq m-1}\left|z_{i}-p_{j}\right|^{-2w_{i}}\prod_{i}\frac{i}{2}dz_{i}\wedge d\bar{z}_{i},\label{eq:Z N as integral of Vandermonde}
\end{equation}
 where we have, for simplicity, dropped the subindex $k$ in the notation
$N_{k}$ (\cite[Lemma 4.3]{berm6b}). For $m=3$ and $(p_{1},p_{2})=(0,1)$
the integral appearing in the right hand side of the previous formula
is known as the Dotsenko-Fateev integral (and can be viewed as a Selberg
integral over the field $\C$ \cite{fu-zhu}). By \cite[Formula B.9]{d-f}
(and \cite[formula 3.1]{fu-zhu}) it may be explicitly computed in
terms of the function $l(x)$ appearing in Lemma \ref{lem:gamma}:
\begin{equation}
\mathcal{Z}_{N}=N!\left(\frac{\pi}{l(\frac{1}{2}\frac{V}{N-1})}\right)^{N}\prod_{j=0}^{N-1}\frac{l(\frac{j+1}{2}\frac{V}{N-1})}{l(w_{1}-\frac{j}{2}\frac{V}{N-1})l(w_{2}-\frac{j}{2}\frac{V}{N-1})l(w_{3}-\frac{j}{2}\frac{V}{N-1})},.\label{eq:DF formula}
\end{equation}
Hence, by Theorem \ref{thm:periods n one both signs},
\begin{align}
-\hat{h}_{\text{can}}(\mathcal{K}_{(\mathcal{X},\Delta)}) & =\lim_{N\rightarrow\infty}\frac{1}{2N}\log\mathcal{Z}_{N}\nonumber \\
= & \lim_{N\rightarrow\infty}\frac{1}{2N}\left(\log N!+N\log(\pi)-N\log(l(\frac{1}{2}\frac{V}{N-1}))\right)+\nonumber \\
 & \frac{1}{V}\sum_{j=0}^{N-1}\log(l(\frac{j+1}{2}\frac{V}{N-1}))\frac{V}{2N}-\frac{1}{V}\sum_{k=1}^{3}\sum_{j=0}^{N-1}\log l(w_{k}-\frac{j}{2}\frac{V}{N-1})\frac{V}{2N}\label{eq: height limit as riemann sums neg curv}
\end{align}
Using Stirling's approximation and the fact that the gamma function
has a simple pole with residue 1 at 0 gives
\begin{align*}
\frac{1}{2N}(\log N!-N\log(l(\frac{1}{2}\frac{V}{N-1}))) & =\\
\frac{1}{2}\left(\log N-1-\log(\Gamma(\frac{V}{2}\frac{1}{N-1}))+\log(\Gamma(1-\frac{V}{2}\frac{1}{N-1})\right) & +O(N^{-1}\log N)\\
\frac{1}{2}\left(\log N-1-\log(\frac{2}{V}(N-1)+O(1))+\log(\Gamma(1-\frac{V}{2}\frac{1}{N-1})\right) & +O(N^{-1}\log(N)\\
\rightarrow_{N\rightarrow\infty} & \frac{1}{2}(\log\frac{V}{2}-1).
\end{align*}
All in all, recognizing the sums over $j$ in \ref{eq: height limit as riemann sums neg curv}
as either right or left Riemann sums, this proves $\hat{h}_{\text{can}}(\mathcal{K}_{(\P_{\Z}^{1},\mathcal{D}^{o})})=f(\boldsymbol{w}).$

\subsection{The case $-K_{(\P^{1},\Delta)}>0$}

It is enough to consider the case when $(\P^{1},\Delta)$ is K-stable
(i.e. the case when $\mathbf{w}$ is contained in the interior of
$C$), by the continuity in Prop \ref{prop:conc}. Since $\left|V\right|$
is the volume (degree) of $-K_{(\P^{1},\Delta)},$
\[
-k\mathcal{K}_{(\P_{\Z}^{1},\mathcal{D}^{o})}\simeq k\left|V\right|\mathcal{O}(1),\,\,\,\left|V\right|=2-\sum_{i=1}^{m}w_{i}.
\]
 Using that $\left|V\right|=-V$ formula \ref{eq:DF formula} now
yields 
\[
\mathcal{Z}_{N}=N!\left(\frac{\pi}{-l(-\frac{1}{2}\frac{\left|V\right|}{N-1})}\right)^{N}\prod_{j=0}^{N-1}\frac{-l(-\frac{(j+1)}{2}\frac{\left|V\right|}{N-1})}{l(w_{1}+\frac{j}{2}\frac{\left|V\right|}{N-1})l(w_{2}+\frac{j}{2}\frac{\left|V\right|}{N-1})l(w_{3}+\frac{j}{2}\frac{\left|V\right|}{N-1})}.
\]
 Hence, proceeding precisely as before, gives
\[
\frac{1}{2}h_{\text{can}}(-\mathcal{K}_{(\P_{\Z}^{1},\mathcal{D}^{o})})=\frac{\left|V\right|}{2}(\log\frac{\left|V\right|}{2}+\log\pi-1)+\int_{-\frac{\left|V\right|}{2}}^{0}\log(-l(x))\mathrm{d}x-\sum_{k=1}^{3}\int_{w_{k}}^{w_{k}+\frac{\left|V\right|}{2}}\log l(x)\mathrm{d}x.
\]
 Finally, exploiting that $\Gamma(x+1)=x\Gamma(x)$ the integral over
$[-\left|V\right|/2,0]$ may be rewritten as
\[
-\int_{0}^{\frac{\left|V\right|}{2}}\left(\log l(x)-2\log(x)\right)\mathrm{d}x=-\int_{0}^{\frac{\left|V\right|}{2}}\log l(x)\mathrm{d}x-\left|V\right|\log(\frac{\left|V\right|}{2})+\left|V\right|
\]
so that in total
\[
\hat{h}_{\text{can}}(-\mathcal{K}_{(\P_{\Z}^{1},\mathcal{D}^{o})})=\frac{1}{2}(-\log\frac{\left|V\right|}{2}+\log\pi+1)-\frac{1}{\left|V\right|}\int_{0}^{\frac{\left|V\right|}{2}}\log(l(x))\mathrm{d}x-\sum_{k=1}^{3}\frac{1}{\left|V\right|}\int_{w_{k}}^{w_{k}+\frac{\left|V\right|}{2}}\log l(x)\mathrm{d}x.
\]

\subsubsection{Real-analyticity }

Let us give three different proofs that $\pm\hat{h}_{\text{can}}(\pm\mathcal{K}_{(\P_{\Z}^{1},\mathcal{D}^{o})})$
is real-analytic in the interior of $C.$ First, this is a special
case of Prop \ref{prop:real anal}. Secondly, let us show directly
from Theorem \ref{thm:explicit intro} that $\pm\hat{h}_{\text{can}}(\pm\mathcal{K}_{(\P_{\Z}^{1},\mathcal{D}^{o})})$
is real-analytic. By Theorem \ref{thm:explicit intro}
\[
\pm\hat{h}_{\text{can}}(\pm\mathcal{K}_{(\P_{\Z}^{1},\mathcal{D}^{o})})=\frac{1}{2}(-\log(\pi)-\int_{0}^{1}\log(\frac{Vt}{2}l(\frac{Vt}{2})\mathrm{)d}t+\sum_{k=1}^{3}\int_{0}^{1}\log l(w_{k}-\frac{Vt}{2})\mathrm{d}t).
\]
Recalling that the gamma function has a simple pole at 0, so that
$\Gamma(x)x$ is real-analytic and positive on $(-1,\infty)$, it
is not to hard to see that the expression above is a real-analytic
function of the weights $w_{i}$ in the interior of $C.$ The third
proof of the real-analyticity exploits that the assumptions in Prop
\ref{prop:zerofree} are satisfied in this case, since $\mathcal{Z}_{N}$
is a product of Gamma-functions (see the end of 4.3). As a consequence,
$\pm\hat{h}_{\text{can}}(\pm\mathcal{K}_{(\P_{\Z}^{1},\mathcal{D}^{o})})$
is the restriction of a uniform limit of holomorphic functions on
$\Omega$ and thus real-analytic. 

\subsection{A case with more than three points}

Assume given $(w_{0},w_{1},w_{\infty})\in]0,1[^{3}.$ Consider the
divisor $\tilde{\Delta}$ on $\P^{1}$ supported at the points $\left\{ 0,1,-1,\infty\right\} $
with weights $(w_{0},w_{1},w_{-1},w_{\infty}).$ Assume that $K_{(\P^{1},\tilde{\Delta})}>0.$
Denote by $\Delta$ the divisor on $\P^{1}$ supported at the points
$\left\{ 0,1,\infty\right\} $ with weights $(1+\frac{w_{0}-1}{2},w_{1},1+\frac{w_{\infty}-1}{2})$
(in particular, if $\Delta$ has ramification $m$ at $p_{0},$ then
$\tilde{\Delta}$ has ramification $2m).$ Denote by $\tilde{\mathcal{D}}$
and $\mathcal{D}$ the Zariski closures in $\P_{\Z}^{1}$ of $\tilde{\Delta}$
and $\Delta,$ respectively. 
\begin{prop}
\label{prop:four points}The following formula holds
\[
\hat{h}_{\text{can}}(\mathcal{K}_{(\P_{\Z}^{1},\tilde{\mathcal{D}})})=\hat{h}_{\text{can}}(\mathcal{K}_{(\P_{\Z}^{1},\mathcal{D})})+\frac{1}{2}\log2=f(1+\frac{w_{0}-1}{2},w_{1},1+\frac{w_{\infty}-1}{2})+\frac{1}{2}\log2
\]
a
\end{prop}

\begin{proof}
Using the standard isomorphism $\mathcal{K}_{(\P_{\Z}^{1},\mathcal{D})})\simeq V(K_{(\P^{1},\Delta)})\mathcal{O}(1)$
over $\Z$ we can identify a given metric $\phi$ on $K_{(\P^{1},\Delta)}$
with a metric on $(K_{(\P^{1},\Delta)})\mathcal{O}(1).$ Consider
the standard map $F:\P^{1}\rightarrow\P^{1}$ of degree $2,$ which
in the standard affine coordinate is given by $y=x^{2}.$ We will
use the same notation $F$ for its standard lift satisfying $F^{*}\mathcal{O}(1)\simeq2\mathcal{O}(1).$
Then $F^{*}\phi$ defines a metric on $2V(K_{(\P^{1},\Delta)})\mathcal{O}(1)$
which (as before) may be identified with a metric on $\mathcal{K}_{(\P_{\Z}^{1},\tilde{\mathcal{D}})})$
(using that $2V(K_{(\P^{1},\Delta)})=V(K_{(\P^{1},\tilde{\Delta})}).$
By basic functoriality of normalized heights $\hat{h}(\mathcal{O}(1),\psi)=\hat{h}(F^{*}\mathcal{O}(1),F^{*}\psi)$
for any metric $\psi$ on $\mathcal{O}(1).$ In particular, $\hat{h}(V\mathcal{O}(1),\phi)=\hat{h}(F^{*}(V\mathcal{O}(1)),F^{*}\phi).$
Hence, it will be enough to show that

\[
\int_{\P^{1}}\mu_{F^{*}\phi}=2^{-1}\int_{\P^{1}}\mu_{\phi}.
\]
To this end first observe that

\[
dx=2^{-1}y^{-1/2}dy,\,\,\,(y-1)(y+1)=(x-1).
\]
 Hence, 
\[
\mu_{F^{*}\phi}=2^{-2}F^{*}\mu_{\phi}.
\]
 Since the map $F$ has degree $2$ it follows that 
\[
\int_{\P^{1}}\mu_{F^{*}\phi}=\int_{\P^{1}}2^{-2}F^{*}\mu_{\phi}=2\int_{\P^{1}}2^{-2}\mu_{\phi}=2^{-1}\int_{\P^{1}}\mu_{\phi},
\]
 as desired.
\end{proof}
\begin{rem}
A similar formula holds when the two points $\{1,1\}$ in the support
of $\tilde{\Delta}$ are replaced by the $d$ points defined as the
$d$ roots of unity (by taking the map $F$ in the proof above to
be defined by $y=x^{d}).$
\end{rem}

\section{Sharp bounds on $\P_{\Z}^{1}$}

In this section we will, in particular, prove Theorem \ref{thm:sharp bounds intro}.
We continue with the notations from Section \ref{sec:The-canonical-height}.
But we start with the following refinement of the conjectural Fujita
type inequality \ref{eq:Fuj} in the present case:
\begin{thm}
\label{thm:refined bounds intro}Let $(\mathcal{X},\mathcal{D};\mathcal{L})$
be a polarized arithmetic log surface $(\mathcal{X},\mathcal{D};\mathcal{L})$
over $\Z$ with $\mathcal{X}$ normal such that the complexification
$X$ of of $\mathcal{X}$ equals $\P^{1}$ and the complexification
$L$ of $\mathcal{L}$ equals either $K_{(\P^{1},\Delta)}$ or $-K_{(\P^{1},\Delta)}$.
Assume that $\Delta$ is supported on at most three points and that
$(X,\Delta)$ is K-semistable (i.e. the weight conditions \ref{eq:weight cond intro}
hold). Then
\[
\mathcal{\hat{M}}_{(\mathcal{X},\mathcal{D})}(\overline{\mathcal{L}})\geq f(\boldsymbol{w})
\]
and if $(X,\Delta)$ is K-stable, then equality holds iff $(\mathcal{X},\mathcal{D})=(\P_{\Z}^{1},\mathcal{D}^{o})$
and $\mathcal{L}=\pm\mathcal{K}_{(\P_{\Z}^{1},\mathcal{D}^{o})}.$
As a consequence, if $\pm\mathcal{K}_{(\mathcal{X},\mathcal{D})}$
is relatively ample, then
\[
\pm\hat{h}_{\text{can}}(\mathcal{\overline{\pm K}}_{(\mathcal{X},\mathcal{D})})\geq f(\boldsymbol{w})
\]
 and if $-\mathcal{K}_{(\mathcal{X},\mathcal{D})}$ is relatively
ample, then
\[
-\hat{h}(\mathcal{\overline{-K}}_{(\mathcal{X},\mathcal{D})})\geq f(\boldsymbol{w})
\]
for any volume-normalized continuous metric on $-K_{(\P^{1},\Delta)}.$ 
\end{thm}

\begin{proof}
The first inequality follows directly from combining the variational
principles for metrics and models in Prop \ref{prop:var princi metrics}
and Prop \ref{prop:optimal model for log P one}, respectively, with
Theorem \ref{thm:explicit intro}. In particular, taking the inf over
all psh metrics of finite energy yields the second inequality. The
third inequality then follows from Lemma \ref{lem:sup over all cont phi},
since $n=1.$
\end{proof}
In fact, the inequalities in the previous theorem hold more generally
when $\Z$ is replaced by $\mathcal{O}_{\F}$ (using Prop \ref{prop:optimal model for log P one}
and the fact that the normalized height is invariant under base-change).
In particular, the first inequality yields the following explicit
expression for the normalized modular invariant, defined in Section
\ref{subsec:Odaka's-modular-invariant},
\begin{equation}
\hat{\mathcal{M}}(\P_{\F}^{1},\Delta_{\F};\pm K_{(\P_{\Q}^{1},\Delta_{\Q})})=f(\boldsymbol{w}),\label{eq:normal mod invariant of log pairs on p one}
\end{equation}
when $\Delta_{\F}$ is supported on three points in $\P_{\F}^{1}.$

\subsection{Proof of Theorem \ref{thm:sharp bounds intro}}

We first establish the following refinement of the first inequality
in Theorem \ref{thm:sharp bounds intro}:
\begin{thm}
\label{thm: linear bound fano}The following inequality holds
\[
\hat{h}(-\mathcal{K}_{(\P_{\Z}^{1},\mathcal{D}^{o})})\geq\frac{1}{2}(1+\log\pi)+\frac{1}{4}(1+\log\frac{3}{4})\sum_{k=1}^{3}w_{k}
\]
In particular, $\hat{h}(-\mathcal{K}_{(\mathcal{X},\Delta)})\geq\hat{h}(-\mathcal{K}_{\mathbb{P}_{\mathbb{Z}}^{1}})>0$.
\end{thm}

\begin{proof}
Set $g(\boldsymbol{w})=-\pm\hat{h}(\pm\mathcal{K}_{(\P_{\Z}^{1},\mathcal{D}^{o})})$
and first consider the case when $-K_{(\P^{1},\Delta)}>0,$ so that
$\left|V\right|$ is the volume (degree) of $-K_{(\P^{1},\Delta)}).$
In this case, $g=\hat{h}(\pm\mathcal{K}_{(\P_{\Z}^{1},\mathcal{D}^{o})})$
and, using Theorem \ref{thm:explicit intro}, 
\begin{align*}
\frac{\mathrm{\partial}g}{\mathrm{\partial}w_{i}} & =\frac{1}{2V}-\frac{1}{\left|V\right|}\int_{0}^{\frac{\left|V\right|}{2}}\log l(x)\mathrm{d}x-\sum_{k=1}^{3}\frac{1}{\left|V\right|}\int_{w_{k}}^{w_{k}+\frac{\left|V\right|}{2}}\log l(x)\mathrm{d}x\\
 & +\frac{1}{2\left|V\right|}\log l(\frac{\left|V\right|}{2})-\frac{1}{2\left|V\right|}\log l(w_{i}+\frac{\left|V\right|}{2})+\frac{1}{\left|V\right|}\log l(w_{i})\\
 & +\sum_{k=1,k\neq i}^{3}\frac{1}{2\left|V\right|}\log l(w_{k}+\frac{\left|V\right|}{2}).
\end{align*}
Next we compute the limit of the gradient at zero along the curve
$w(t)=(t,t,t)$, i.e. 
\[
\lim_{t\rightarrow0}\frac{\mathrm{d}}{\mathrm{d}t}g(w(t))=3\lim_{t\rightarrow0}\frac{\mathrm{d}}{\mathrm{d}w_{i}}g(w)|_{w_{t}}=\frac{3}{4}(1+\lim_{t\rightarrow0}\log(l(1-\frac{3t}{2})l(t)^{2}l(1-\frac{t}{2}))=\frac{3}{4}(1+\log\frac{3}{4}).
\]
By Prop \ref{prop:conc}, $g(w)$ is, in general, convex. Hence, along
the curve $w(t)$ we have $g(w(t))\geq g(0)+t\lim_{t\rightarrow0}\frac{\mathrm{d}}{\mathrm{d}t}g(w(t))$.
Furthermore, as $g$ is symmetric in the weights and convex, we have
$g(w)\geq g(w(t))$ for any $w$ where $\sum_{k=1}^{3}w_{k}=\sum_{k=1}^{3}w_{k}(t)=3t$.
Putting it all together we have shown,
\[
g(w)\geq\frac{1}{2}(1+\log\pi)+\frac{1}{4}(1+\log\frac{3}{4})\sum_{k=1}^{3}w_{k},
\]
as desired. 
\end{proof}
Finally, we establish the following refinement of the second inequality
in Theorem \ref{thm:sharp bounds intro}:
\begin{thm}
The following inequality holds when $K_{(\P^{1},\Delta)}$ is semi-ample:
\[
\hat{h}(\mathcal{K}_{(\P_{\Z}^{1},\mathcal{D}^{o})})\leq-\frac{1}{2}\log(\pi)+\frac{3}{2}\log\frac{\Gamma(\frac{2}{3})}{\Gamma(\frac{1}{3})}+\frac{3}{4}(\gamma+\frac{1}{2}(\frac{\Gamma'(2/3)}{\Gamma(2/3)}+\frac{\Gamma'(1/3)}{\Gamma(2/3)}))(\sum_{k=1}^{3}w-2)
\]
equality holds for the weights $(2/3,2/3,2/3)$. In particular, $\hat{h}(\mathcal{K}_{(\P_{\Z}^{1},\mathcal{D}^{o})})<0.$ 
\end{thm}

\begin{proof}
We have for the gradient of $\hat{h}:=\hat{h}(\mathcal{K}_{(\mathcal{X},\Delta)})$,
with $V$ denoting the degree of $K_{(\P^{1},\Delta)},$ which, by
assumption is non-negative,
\begin{align*}
\frac{\mathrm{d}}{\mathrm{d}w_{i}}\hat{h}(w)= & -\frac{1}{2V}+\frac{1}{V^{2}}\int_{0}^{\frac{V}{2}}\log l(x)\mathrm{d}x-\frac{1}{V^{2}}\sum_{k=1}^{3}\int_{w_{k}-\frac{V}{2}}^{w_{k}}\log l(x)\mathrm{d}x\\
 & -\frac{1}{2V}\log l(\frac{V}{2})+\frac{1}{V}\log l(w_{i})-\frac{1}{2V}\log l(w_{i}-\frac{V}{2})\\
 & +\sum_{k=1,k\neq i}^{3}\frac{1}{2V}\log l(w_{k}-\frac{V}{2}).
\end{align*}
We will compute the limit of the gradient as $t\rightarrow\frac{2}{3}_{+}$,
along the curve $w(t):=(t,t,t)$. We begin by making a few preparatory
calculations. First, as $V\rightarrow0$,
\begin{align*}
\frac{1}{V^{2}}\int_{0}^{\frac{V}{2}}\log l(x)\mathrm{d}x & =\frac{1}{V^{2}}\int_{0}^{\frac{V}{2}}\log(\frac{1}{x}-\gamma+\mathcal{O}(x))+\gamma x+\mathcal{O}(x^{2})\mathrm{d}x\\
 & =\frac{1}{V^{2}}\int_{0}^{\frac{V}{2}}-\log(x)-2\gamma x+\mathcal{O}(x^{2})\mathrm{d}x=-\frac{1}{2V}\log\frac{V}{2}+\frac{1}{2V}-\frac{\gamma}{4}
\end{align*}
where we have used the Laurent series of $\Gamma$ around 0 and $\log\Gamma$
around 1. Next, as $t\rightarrow\frac{2}{3}$,
\begin{align*}
-\frac{3}{(3t-2)^{2}}\int_{t-(3t-2)/2}^{t}\log l(x)\mathrm{d}x & =-\frac{3}{(3t-2)^{2}}\int_{t-(3t-2)/2}^{t}\log l(\frac{2}{3})+(\log l)'(\frac{2}{3})(x-\frac{2}{3})+\mathcal{O}((x-\frac{2}{3})^{2})\mathrm{d}x\\
 & =-\frac{3}{2(3t-2)}\log l(\frac{2}{3})-\frac{1}{8}(\log l)'(\frac{2}{3})+\mathcal{O}(t-\frac{2}{3})
\end{align*}
and

\begin{align*}
-\frac{1}{2V}\log l(\frac{V}{2}) & =-\frac{1}{2V}\log(\frac{2}{V}-\gamma+\mathcal{O}(V))+\frac{\gamma}{4}+\mathcal{O}(V)\\
 & =\frac{1}{2V}\log\frac{V}{2}+\frac{\gamma}{2}+\mathcal{O}(V)
\end{align*}
Thus, in total
\begin{align*}
\lim_{t\rightarrow\frac{2}{3}_{+}}\frac{\mathrm{d}}{\mathrm{d}t}\hat{h}(w(t))= & 3\lim_{t\rightarrow\frac{2}{3}_{+}}\frac{\mathrm{d}}{\mathrm{d}w_{i}}\hat{h}(w)|_{w_{t}}\\
= & 3\lim_{t\rightarrow\frac{2}{3}_{+}}\frac{\gamma}{4}-\frac{3}{2(3t-2)}\log l(\frac{2}{3})-\frac{1}{8}(\log l)'(\frac{2}{3})\\
 & +\frac{1}{3t-2}\log l(t)+\frac{1}{2(3t-2)}\log l(t-\frac{3t-2}{2})\\
= & 3\lim_{s\rightarrow0_{+}}-\frac{3}{2s}\log l(\frac{2}{3})-\frac{1}{8}(\log l)'(\frac{2}{3})\\
 & +\frac{1}{s}\log l(\frac{s}{3}+\frac{2}{3})+\frac{1}{2s}\log l(-\frac{s}{6}+\frac{2}{3})\\
= & \frac{3\gamma}{4}+\frac{3}{8}(\log l)'(\frac{2}{3})\\
= & \frac{3}{4}(\gamma+\frac{1}{2}(\frac{\Gamma'(2/3)}{\Gamma(2/3)}+\frac{\Gamma'(1/3)}{\Gamma(2/3)})).
\end{align*}
We also want to evaluate the height (or rather take the limit) for
the weights $(2/3,2/3,2/3)$, which is an easy variation of the above
calculation, 
\[
\lim_{t\rightarrow\frac{2}{3}_{+}}\hat{h}(w(t))=-\frac{1}{2}\log(\pi)+\frac{3}{2}\log l(\frac{2}{3})
\]
By a similar argument as in Theorem \ref{thm: linear bound fano},
using the concavity of the height, we have thus shown 
\[
\hat{h}\leq-\frac{1}{2}\log(\pi)+\frac{3}{2}\log\frac{\Gamma(\frac{2}{3})}{\Gamma(\frac{1}{3})}+\frac{3}{4}(\gamma+\frac{1}{2}(\frac{\Gamma'(2/3)}{\Gamma(2/3)}+\frac{\Gamma'(1/3)}{\Gamma(2/3)}))(\sum_{k=1}^{3}w-2).
\]
\end{proof}

\section{Specific values of canonical heights}

In this section we continue with the case when $\mathcal{X}=\P_{\Z}^{1}$
and $\mathcal{D}^{o}$ is the Zariski closure of the divisor $\Delta$
on $\P_{\Q}^{1}$ supported on the three points $\{0,1,\infty\}.$
We consider only the ``orbifold/cusp case'' where the coefficients
$w_{i}$ of $\Delta$ are of the form $w_{i}=1-1/m_{i}$ for $m_{i}\in\N\cup\{\infty\}$,
where $m_{i}$ are called \emph{ramification indices.} The formulas
in Table 1 (Section \ref{subsec:Specific-values-of}) are obtained
by simplifying the explicit expression $f(\boldsymbol{w)}$ appearing
in Theorem \ref{thm:explicit intro}. Moreover, we also compute the
canonical height in some log Fano cases. We will provide the complete
calculation only for the simplest cases. The calculations in the remaining
cases are similar, but since they are somewhat lengthy they are merely
outlined (in order to avoid computational mistakes, we have numerically
verified the end results to machine precision, using standard implementations
of the expression $f(\boldsymbol{w})$).

\subsection{\label{subsec:spec log K pos}The case when $K_{(\P^{1},\Delta)}>0$}

Set 
\[
\hat{h}:=\hat{h}_{\text{can }}(\mathcal{K}_{(\P_{\Z}^{1},\mathcal{D}^{o})})+\log\frac{\pi V}{2}=\frac{1}{2}-\frac{\gamma(0,\frac{V}{2})-\sum_{i=1}^{3}\gamma(w_{i}-\frac{V}{2},w_{i})}{V},
\]
 using, in the second equality, Theorem \ref{thm:explicit intro}.
In the application to Shimura curves, considered in Section \ref{sec:Applications-to-Shimura},
$\hat{h}$ is the normalized height of $\mathcal{K}_{(\P_{\Z}^{1},\mathcal{D})}$
with respect to the Petersson metric. 
\begin{prop}
For the ramification indices $(2,3,\infty)$ 
\[
\hat{h}=-\frac{\zeta'(-1)}{\zeta(-1)}-\frac{1}{2}-\frac{1}{4}\log(12).
\]
\end{prop}

\begin{proof}
Denote as before $F(x)=\zeta(-1,x)\text{+\ensuremath{\zeta'(-1,x)}}.$
By Theorem \ref{thm:explicit intro} and \ref{eq:def of gamma} we
have (using also that $F(0)=F(1)$ interpreted correctly)
\begin{align*}
\hat{h}= & \frac{1}{2}+\frac{1}{V}(-F(1/12)-F(11/12)+F(0)+F(12/12)\\
 & +F(6/12)+F(6/12)-F(5/12)-F(7/12)\\
 & +F(8/12)+F(4/12)-F(7/12)-F(5/12)\\
 & +F(12/12)+F(0)-F(11/12)-F(1/12))\\
= & 1/2+\frac{1}{V}(-2F(1/12)+F(4/12)-2F(5/12)+2F(6/12)\\
 & -2F(7/12)+F(8/12)-2F(11/12)+F(12/12))
\end{align*}
To relate the linear combination of Hurwitz zeta functions to the
Riemann zeta function, we use the $\textit{multiplication theorem}$
\begin{equation}
k^{s}\zeta(s)=\sum_{i=1}^{k}\zeta(s,i/k)\label{eq:multiplication formula}
\end{equation}
repeatedly for appropriate values of $k$ and end up with
\begin{align*}
 & -2\zeta(s,1/12)+\zeta(s,4/12)-2\zeta(s,5/12)+2\zeta(6/12)\\
 & -2\zeta(s,7/12)+\zeta(s,8/12)-2\zeta(s,11/12)+\zeta(s,12/12)\\
 & =(-2\cdot12^{s}+2\cdot6^{s}+2\cdot4^{s}+3^{s}+1)\zeta(s)
\end{align*}
Using this together with the definition \ref{eq:def of gamma} of
$F$ leads to
\begin{align*}
\hat{h}= & \frac{1}{2}+\frac{1}{V}((-2\cdot12^{-1}+2\cdot6^{-1}+2\cdot4^{-1}+3^{-1}+1)\zeta(-1)\\
 & +(-2\log(12)12^{-1}+2\log(6)6^{-1}+2\log(4)4^{-1}+\log(3)3^{-1})\zeta(-1)\\
 & +(-2\cdot12^{-1}+2\cdot6^{-1}+2\cdot4^{-1}+3^{-1}+1)\zeta'(-1)\\
= & \frac{1}{2}+\frac{1}{V}(-\frac{2}{12}-\frac{\log(12)}{24}+2\zeta'(-1))\\
= & -\frac{\zeta'(-1)}{\zeta(-1)}-\frac{1}{2}-\frac{1}{4}\log(12).
\end{align*}
\end{proof}
For the ramification indices $(6,2,6)$ the calculation is similar
to the previous case. Next, for the case of ramification indices $(4,4,4)$
the calculation also proceeds in a similar way, but now using that

\[
\zeta_{\mathbb{Q}(\sqrt{2})}(s)=\zeta(s)8^{-s}(\zeta(s,\frac{1}{8})-\zeta(s,\frac{3}{8})-\zeta(s,\frac{5}{8})+\zeta(\frac{7}{8})),
\]
by the standard factorization formula for $\zeta_{\F}(s),$ when $\F$
is an abelian Galois extension of $\mathbb{Q}$. From the explicit
formula in Theorem \ref{thm:explicit intro} for $\hat{h}$ we get
a linear combination of Hurwitz zeta functions rather than a linear
combinations of products of Hurwitz zeta functions as in the equation
above. But, after differentiating and evaluating at $-1$, we can
still use the multiplication theorem for the Hurwitz zeta function
a number of times on the derivative terms, while evaluating the rest
in terms of explicit rational numbers. This uses the well known relation
between values of the Hurwitz zeta function at $-1$ and the second
Bernoulli polynomial. Indeed, 
\[
\zeta(-1,a)=-\frac{B_{2}(a)}{2}.
\]

Finally, for the rest of the cases, i.e. $(3,3,6),\ (5,5,5),\ (6,6,6),\ (2,4,12),\ (7,7,7)$
and $(9,9,9)$, the strategy is the same as for the case of $(4,4,4)$,
noting that the respective number fields are all abelian Galois extensions. 

\subsection{\label{subsec:spec log Fano}The case when $-K_{(\P^{1},\Delta)}>0$}

In this section we will use the shorthand $\hat{h}_{\text{can }}:=\hat{h}_{\text{can }}(-\mathcal{K}_{(\P_{\Z}^{1},\mathcal{D})}).$
We start by verifying that Theorem \ref{thm:explicit intro} recovers
some simple cases of Fano orbifolds, where the canonical height has
previously been computed.

\subsubsection{The case when $\Delta$ is supported on two points. }

In this case we may, by symmetry, assume that $w_{2}=0.$ In this
case the K-semistability assumption implies that $w_{1}=w_{3}.$ By
\cite[Lemma 3.2]{a-b2}, 

\[
\hat{h}_{\text{can }}=\frac{1}{2}(1+\log\pi-\log\frac{V}{2})
\]
(when all weights vanish this specializes to the well-known formula
for the canonical height of $\P_{\Z}^{1}).$ In order for the previous
formula to be consistent with the previous theorem it must be that
\[
\gamma(0,\frac{V}{2})+\gamma(1-\frac{V}{2},1)=0.
\]
 Let us give a direct proof of this vanishing, using a symmetry argument.
Setting $\lambda=V/2$ and $g(t)=\log\Gamma(x)$ the left hand side
in the previous formula may be expressed as the integral over $[-\lambda/2,\lambda/2]$
of the function 
\[
\left(g(x+\lambda/2)-g(-x+\lambda/2)\right)+\left(g(x+1-\lambda/2)-g(-x+1-\lambda/2)\right),
\]
 which is odd (since both terms are). Hence, the integral over $[-\lambda/2,\lambda/2]$
indeed vanishes. 

\subsubsection{The case of ramification indices $(2,2,2)$ and the Fermat curve
$\mathcal{X}_{2}$ of degree two.}

Let us next show that when all weights $w_{i}$ equal $1/2$ 
\begin{equation}
\hat{h}_{\text{can}}=\frac{1}{2}\left(1+\log\pi\right)+\frac{1}{2}\log2\label{eq:canonical height for two two two}
\end{equation}
 by expressing $\hat{h}_{\text{can}}(\P_{\Z}^{1},\mathcal{D})$ in
terms of the canonical height of the Zariski closure $\mathcal{X}_{2}$
in $\P_{\Z}^{2}$ of the Fermat curve of degree two:
\[
\hat{h}_{\text{can }}=\hat{h}_{\text{can }}(-\mathcal{K}_{\mathcal{X}_{2}})+\log2
\]
(as follows from realizing $(\P_{\Z}^{1},\mathcal{D})$ as a Galois
cover of $\mathcal{X}_{2}$ as in \cite[Section 5.2]{a-b2}). Formula
\ref{eq:canonical height for two two two} thus follows from 
\[
2\hat{h}_{\text{can }}(-\mathcal{K}_{\mathcal{X}_{2}})=1+\log\pi-\log2,
\]
which can be deduced from the height formula for quadrics in \cite{c-m}
(or by noting that $\mathcal{X}_{2}$ is the blow-up of $\P_{\Z}^{1}$
of a closed point on the fiber over the prime $(2)$). Since, in this
case, $V=1/2,$ formula \ref{eq:canonical height for two two two}
is, by Theorem \ref{thm:explicit intro}, equivalent to the identity
\[
\gamma(0,\frac{1}{4})+3\gamma(\frac{1}{2},\frac{3}{4})=\frac{1}{4}\log2.
\]
This identity can, indeed, be verified using the multiplication formula
\ref{eq:multiplication formula}. Indeed, applying this formula for
$k=2,4$ one easily finds the above relation.

\subsubsection{New cases}

We next consider some cases where the canonical height has not been
computed before:

\begin{table}[th]
\begin{tabular}{cc}
\toprule 
$(m_{1},m_{2},m_{3})$ & $\hat{h}_{\text{can }}-\frac{1}{2}(1+\log\pi)$\tabularnewline
\midrule
\midrule 
$(2,2,3)$ & $-\frac{1}{6}\log2+\frac{2}{3}\log3$\tabularnewline
\midrule 
$(2,2,4)$ & $\frac{3}{4}\log2$\tabularnewline
\midrule 
$(2,3,3)$ & $\frac{1}{2}\log2+\frac{1}{8}\log3$\tabularnewline
\midrule 
$(2,3,4)$ & $\frac{7}{12}\log2+\frac{1}{8}\log3$\tabularnewline
\bottomrule
\end{tabular}\label{table: log Fano curves}\bigskip{}
\caption{Ramification indices and the corresponding normalized height for some
log Fano orbifolds.}
\end{table}
In general, any K-stable Fano orbifold curve can have ramification
indices $(m_{1},m_{2},m_{3})$ from either the infinite list $(2,2,r),r\geq2$
or the exceptional list $(2,3,3),(2,3,4),(2,3,5).$ 
\begin{prop}
For the ramification indices $(2,3,3)$ 
\[
\hat{h}_{\text{can }}=\frac{1}{2}+\frac{1}{2}\log\pi+\frac{1}{8}\log48.
\]
\end{prop}

\begin{proof}
Denote as before $F(x)=\zeta(-1,x)\text{+\ensuremath{\zeta'(-1,x)}}.$
Then 
\begin{align*}
\hat{h}_{\text{can }}= & \frac{1}{2}(-\log\frac{V}{2}+\log\pi+1)-\frac{1}{V}(\\
 & F(1/12)+F(11/12)-F(0)-F(1)\\
 & +F(7/12)+F(5/12)-2F(6/12)\\
 & +2F(9/12)+2F(3/12)-2F(8/12)-2F(4/12)).
\end{align*}
By repeatedly using the multiplication theorem \ref{eq:multiplication formula}
we find the following identity for the relevant linear combination
of Hurwitz zeta functions.
\begin{align*}
-\zeta(s,1/12)-2\zeta(s,3/12)+2\zeta(s,4/12)\\
-\zeta(s,5/12)+2\zeta(s,6/12)-\zeta(s,7/12)\\
+2\zeta(s,8/12)-2\zeta(s,9/12)-\zeta(s,11/12)\\
=(-12^{s}+6^{s}-4^{s}+2\cdot3^{s}+3\cdot2^{s}-2)\zeta(s) & .
\end{align*}
This allows us to compute $\hat{h}$, and quite remarkably, since
\[
-12^{-1}+6^{-1}-4^{-1}+2\cdot3^{-1}+3\cdot2^{-1}-2=0,
\]
there is no $\zeta'(-1)$ appearing in the expression, which is in
total
\[
\hat{h}_{\text{can }}=\frac{1}{2}+\frac{1}{2}\log\pi+\frac{1}{8}\log48.
\]
 The rest of the cases in Table 2 are computed in a similar manner. 

In all cases considered above, $\hat{h}_{\text{can }}(-\mathcal{K}_{(\P_{\Z}^{1},\mathcal{D}^{o})})-\hat{h}_{\text{can }}(-\mathcal{K}_{(\P_{\Z}^{1}})$
is a sum of terms $q(p)\log p$ for primes $p$ and $q(p)\in\Q.$
As next shown, this is always the case:
\end{proof}
\begin{prop}
\label{prop:Belyie log Fano}For any ramification indices $(m_{1},m_{2},m_{3})$
such that the corresponding log pair $(\P_{\Z}^{1},\mathcal{D}^{o})$
is Fano, i.e. $-\mathcal{K}_{(\P_{\Z}^{1},\mathcal{D}^{o})}$ is relatively
ample,
\[
\hat{h}_{\text{can }}(-\mathcal{K}_{(\P_{\Z}^{1},\mathcal{D}^{o})})=\hat{h}_{\text{can }}(-\mathcal{K}_{(\P_{\Z}^{1}})+\sum_{p}q(p)\log p,
\]
 where $p$ ranges over a finite number of primes and $q(p)\in\Q.$ 
\end{prop}

\begin{proof}
By the ADE-classification of log Fano orbifolds over $\C,$ the orbifold
$(\P_{\C}^{1},\Delta)$ over $\C,$ induced by\textbf{ }$(\P_{\Z}^{1},\mathcal{D}^{o}),$
coincides with the orbifold induced from an action on $\P^{1}$ by
a finite group $G\subset SU(2)$ \cite[Chapter 8]{ki}. By Hurwitz
formula, $f^{*}(-K_{(\P_{\C}^{1},\Delta)})=-K_{\P_{\C}^{1}}.$ Moreover,
since $G$ preserves the Fubini-Study metric $\omega_{\text{FS}}$
on $\P^{1}$ its push forward $f_{*}\omega_{\text{FS}}$ is a Kähler-Einstein
metric $(\P_{\C}^{1},\Delta).$ Fixing a Kähler-Einstein metric $\phi_{\text{ }}$
on $-K_{(\P_{\C}^{1},\Delta)}$ this means that $f^{*}\phi_{\text{ }}$
defines a Kähler-Einstein metric on $-K_{\P_{\C}^{1}}.$ Next, since
the branching locus of the corresponding quotient morphism $f:\P^{1}\rightarrow\P^{1}/G$
is contained in $\{0,1,\infty\}$ (i.e. it is a Belyi function), $f$
is defined over a number field $\F$ (in fact, there is an explicit
formula for $f$ going back to Klein \cite[Section 4.1]{kur}). As
a consequence, there exists a regular projective model $\mathcal{Y}$
of $\P_{\F}^{1}$ over $\mathcal{O}_{F}$ and generically finite morphisms
$g_{1}$ and $g_{2}$ from\emph{ $\mathcal{Y}$} to $\P_{\mathcal{O}_{\F}}^{1}$
such that, on the generic fiber $f=g_{2}\circ(g_{1})^{-1}$ and 
\[
g_{1}^{*}(-\mathcal{K}_{\P_{\F}^{1}})=-(\mathcal{K}_{\mathcal{Y}}+\mathcal{R}_{1}),\,\,\,g_{1}^{*}(-\mathcal{K}_{\P_{\F}^{1}})=-(\mathcal{K}_{\mathcal{Y}}+\mathcal{R}_{2})
\]
 for $\Q-$divisors $\mathcal{R}_{i}$ on $\mathcal{Y}.$ It thus
follows from the proof of Lemma \ref{lem:local h} that 
\[
\hat{h}_{\text{\ensuremath{\phi} }}(-\mathcal{K}_{(\P_{\Z}^{1},\mathcal{D}^{o})})-\hat{h}_{f^{*}\phi}(-\mathcal{K}_{(\P_{\Z}^{1}})=\sum_{p}a(p)\log p
\]
 for a finite number of primes $p$ and $a(p)\in\Q.$ Finally, since
the integral of the measure attached to $\phi$ coincides with the
integral of the measure attached to $f^{*}\phi,$ up to multiplication
by the degree of $f,$ this proves the proposition.
\end{proof}
This leads one to wonder if the canonical height of any K-semistable
Fano orbifold is always a rational number mod $\log(\pi^{\Q}\N^{\Q})$?
For example, this is in line with previous explicit height formulas
on Fano varieties wrt Kähler-Einstein metrics, which - as far as we
know - all concern \emph{homogeneous} Fano varieties \cite{ma2,c-m,k-k,ta1,ta2,ta3}.

\subsection{\label{subsec:extension of the explicit formula}Extension to other
arithmetic triangle groups (?)}

Let us come back to the case of log pairs satisfying $\mathcal{K}_{(\P_{\Z}^{1},\mathcal{D}^{o})}(\C)>0,$
for given ramification indices $(m_{1},m_{2},m_{3}).$ Assume that
$\Gamma(m_{1},m_{2},m_{3})$ satifies the following ``arithmetic''
condition: $\Gamma(m_{1},m_{2},m_{3})$ is a subgroup of $\Gamma^{(+)}(B,\mathcal{O}_{B})$
of finite index, where $\Gamma^{(+)}(B,\mathcal{O}_{B})$ is a triangle
group corresponding to a quaternion algebra $B$ over a totally real
number field $\F.$ This means that $\Gamma^{(+)}(B,\mathcal{O}_{B})$
is also of the form $\Gamma(n_{1},n_{2},n_{3})$ and that $\H/\Gamma^{(+)}(B,\mathcal{O}_{B})$
is one of the components of the complex points of a Shimura curve
$(X_{\F},\Delta_{\F}$) (see Section \ref{subsec:genus zero}). In
this case it follows from Yuan's formula \ref{eq:yuan intro} that
\begin{equation}
\hat{h}_{\text{Pet }}\left(\mathcal{K}_{(\P_{\Z}^{1},\mathcal{D}^{o})}\right)=-\frac{1}{2}-\frac{1}{[\F:\Q]}\frac{\zeta'_{\F}(-1)}{\zeta{}_{\F}(-1)}_{\text{ }}+\sum_{p}q(p)\log p,\label{eq:h pet in terms of log p}
\end{equation}
where $p$ ranges over a finite number of primes and $q(p)\in\Q.$
Indeed, the quotient map $f:\H/\Gamma(m_{1},m_{2},m_{3})\rightarrow\H/\Gamma(n_{1},n_{2},n_{3})$
preserves the Peterson metric (see formula \ref{eq:Poincare}). Moreover,
since $f$ is a Belyi function, it is defined over $\bar{\Q}.$ Hence,
formula \ref{eq:h pet in terms of log p} follows from combining the
argument in the proof of Proposition \ref{prop:Belyie log Fano} with
Lemma \ref{lem:local h} and formula \ref{eq:yuan intro}.

However, computing the numbers $q(p)$ explicitly would require an
explicit knowledge of the arithmetic geometry of the canonical model
$(\mathcal{X},\mathcal{D}),$ as well as the corresponding Belyi function.
In contrast, for the cases in Table 1 (which all satisfy the arithmetic
condition above) $q(p)$ is computed explicitely using the formula
in Theorem \ref{thm:explicit intro}. Accordingly, it would be interesting
to know if Theorem \ref{thm:explicit intro} could be used to compute
$q(p)$ explicitly in \emph{all} these arithmetic cases. While there
are probably a few remaining cases where the procedure in Section
\ref{subsec:spec log K pos} works, we found several cases indicating
that the relations coming from the multiplication identity for the
Hurwitz zeta function are not enough to establish an explicit formula
of the form \ref{eq:h pet in terms of log p}. It may still be that
some more complicated identities could be leveraged. For instance,
$\Gamma(m,m,m)$ satifies the arithmetic condition in question precisely
when $m\in\{4,5,6,7,8,9,12,15\}$ and then $\F=\mathbb{Q}(\cos(\pi/m))$
\cite{tak}. If it would be enough to apply the multiplication identity
for the Hurwitz zeta function to extend Table 1 to the case $m=8,$
then the resulting formula would only involve a $\log p-$term for
$p=2$ (as in formula \ref{eq:fermat formula}). However, numerical
investigations indicate that the naive height of the rational coefficient
$c_{2}$ in front of $\log2$ would have to be very large ($\gtrapprox30000$). 

\section{\label{sec:Applications-to-Shimura}Applications to Shimura curves}

\subsection{\label{subsec:The-setup-of Shim}Setup }

We start by recalling the setup in \cite{yu1,y-z}. Let $\F$ be a
totally real number field and $\Sigma$ a finite set of places of
$\F$ of odd cardinality, containing all the infinite places of $\F.$
Let $\B$ be a totally definite incoherent quaternion algebra over
the adele ring $\A$ associated to $\F.$ To a compact open subgroup
$U\Subset\B_{f}^{\times}$ is attached a \emph{Shimura curve} $X_{U}$
over $\F.$ This is a non-singular projective curve over $\F,$ which
may be defined as a course moduli scheme \cite[Section 1.2.1]{y-z}.
Its complex points $X_{U}(\C)$ may be represented as follows. Fix
an infinite place $\sigma$ of $\F$ and denote by $B$ the indefinite
quaternion algebra over $\F$ with ramification locus $\Sigma-\{\sigma\}$
(denoted by $\Sigma_{f}$). Then $X_{\F}(\C)$ is the compactification
of
\[
B^{\times}\setminus\H^{\pm}\times\B_{f}^{\times}/U\,\,\left(\simeq\bigsqcup_{\sigma}\H/\Gamma_{\sigma}\right)
\]
for a finite number of appropriate discrete subgroups $\Gamma_{\sigma}$
of $SL(2,\R).$ More precisely, this quotient construction induces
a log pair $(X_{U},\Delta)$ such that $K_{(X,\Delta)}$ is ample,
where $\Delta_{\F}$ is the orbifold/cusp divisor appearing as the
branching divisor, plus the cusps. Denote by $X$ the Shimura curve
corresponding to a \emph{maximal }compact open subgroup $U\Subset\B_{f}^{\times}.$
By \cite[Section 4.2]{y-z}, it has a canonical integral model $\mathcal{X}$
over $\mathcal{O}_{\F},$ which is a projective flat, normal and $\Q-$factorial
arithmetic surface over $\mathcal{O}_{\F}.$ Briefly, the model $\mathcal{X}$
is defined as follows, locally over the base $\text{Spec}\mathcal{O}_{\F}.$
First, to an appropriate compact open subgroup $U'$ of $\B_{f}^{\times}$
is attached a regular model $\mathcal{X}_{U'}$ of $X_{U'}$ that
is stable over $\mathcal{O}_{\F},$ in the sense of Deligne-Mumford.
Then the scheme $\mathcal{X}$ is defined as the quotient of $\mathcal{X}_{U'}$
by the finite group $U'/U.$ In particular, there is a finite morphism
\begin{equation}
g:\,\mathcal{X}_{U'}\rightarrow\mathcal{X}.\label{eq:finite morphism absolute}
\end{equation}
 Moreover, as shown in \cite[Section 4.2]{y-z}, $K_{(X,\Delta)}$
admits a canonical relatively ample model $\mathcal{L}$ over $\mathcal{O}_{\F},$
dubbed the \emph{Hodge bundle}. It may be defined as the norm $N_{g}(\mathcal{K}_{\mathcal{X}_{U'}})$
of the relative canonical line bundle of $\mathcal{X}_{U'}$ under
$g,$ divided by the degree of $g.$ As a consequence, we can express
$\mathcal{L}$ as the log canonical line bundle of a canonical effective
divisor $\mathcal{D}$ on $\mathcal{X}:$ 
\[
\mathcal{L}=\mathcal{K}_{(\mathcal{X},\mathcal{D})},\,\,\,\,\,\mathcal{D}:=(\deg g)^{-1}N_{g}(\mathcal{R}),\,\,\,\,\mathcal{R}:=(\mathcal{K}_{\mathcal{X}_{U'}}-g^{*}\mathcal{K}_{\mathcal{X}}),
\]
where $\mathcal{R}$ is an effective divisor on $\mathcal{X}_{U'}$
(using that $g$ is a ramified cover in the sense of \cite[Def 2.39]{ko}).

The complex points of $\mathcal{L}$ may be identified with $K_{(X,\Delta)}(\C),$
which is endowed with the \emph{Petersson metric.} Using the normalization
adopted in \cite{yu1}, this is the metric $\phi_{\text{Pet}}$ on
$K_{(X,\Delta)}(\C)$ which pulls back to the Poincaré metric on the
canonical line bundle of the upper half-plane $\H$ under the uniformization
maps $\H\rightarrow(X_{\sigma}(\C),\Delta_{\sigma})$ (ramified along
$\Delta_{\sigma}),$ where the Poincaré metric is defined by
\begin{equation}
\left\Vert d\tau\right\Vert _{\text{Pet}}:=2\text{Im \ensuremath{(\tau)} },\,\,\,\,\,\,\H=\{\text{Im}(\tau)>0\}\Subset\C.\label{eq:Poincare}
\end{equation}

\begin{example}
\label{exa:course mod space for Q}When $\F=\Q$ it is shown in \cite[Lemma 2.1]{yu2}
that $\mathcal{D}$ is the Zariski closure of $\Delta.$ Moreover,
the scheme $\mathcal{X}$ is the coarse moduli scheme of the moduli
stack parametrizing all Abelian schemes over $\Z$ of relative dimension
$2$ with a special action by a maximal order $\mathcal{O}_{B}$ in
$B$ (see \cite{k-r-y} and \cite[Sections 2.1, 2.2]{yu2} with $n(U)=1$).
More precisely, when $B=M_{2}(\Q)$ the course moduli space has to
be ``compactified'' and then $X_{\Q}$ is the classical modular
curve over $\Q$ \cite{d-r}. In general, when $\F=\Q,$ there is
a universal abelian scheme $\pi:\mathcal{A}\rightarrow\mathcal{X}$
and, by \cite[Thm 1.1]{yu2}, the normalized height of $\overline{\mathcal{L}}$
coincides with normalized height of the Hodge bundle $\pi_{*}\mathcal{K}_{\mathcal{A}/\mathcal{X}}\rightarrow\mathcal{X}$
plus $2^{-1}\log d_{B}$ (where $d_{B}$ denotes the discriminant
of $B)$, when the Hodge bundle is endowed with the Faltings metric,
normalized in the following way: $\left\Vert \alpha\right\Vert _{x}^{2}:=(2\pi)^{-2}\left|\int_{\pi^{-1}(x)(\C)}\alpha\wedge\bar{\alpha}\right|.$
\end{example}

\subsubsection{\label{subsec:genus zero}The case when $X$ has geometric genus
zero }

Now specialize to the case when $X_{\bar{\Q}}\cong\P_{\bar{\Q}}^{1}$
and the corresponding divisor $\Delta$ is supported at most three
points, up to taking finite covers. By the classification result in
\cite{tak}, there are $19$ different classes of such quaternionic
Shimura curves, corresponding to $13$ different totally real fields
$\F.$ The ramification indices of the corresponding divisors are
explicitly given in \cite[Table 3]{tak}. Let us recall the classical
terminology in \cite{tak}, since it different than the one in \cite{yu1}.
There is an isomorphism $\rho_{1}$ of the quaternion algebra $B$
into $M_{2}(\R).$ The image in $M_{2}(\R)$ of the group of all elements
in a maximal order $\mathcal{O}_{B}$ in $A$ whose reduced norm is
a totally positive element in $\F^{\times}$ is denoted by $\Gamma^{(+)}(B,\mathcal{O}_{B}).$
In the terminology of \cite{yu1} the quotient $\H/\Gamma^{(+)}(B,\mathcal{O})$
coincides with one connected component of $X_{\F}(\C).$ Indeed, as
shown in \cite[ Section 3.1.1]{y-z-z} the complex points $X(\C)$
may (up to compactifying the cusps) be decomposed in connected components
of the form $\H/\Gamma_{h}$ where $\Gamma_{h}:=B_{+}^{\times}\cap hUh^{-1},$
where $B_{+}$denotes the elements in $B$ with positive reduced norm
and $h$ ranges over a some elements including the identity $e.$
Since $U$ is assumed maximal it can be taken to be $\mathcal{\hat{O}}_{B}(:=\mathcal{O}_{B}\times\hat{\mathcal{O}}_{\F}).$
Thus, when $h$ is the identity $e$ we get $\Gamma_{e}=B_{+}^{\times}\cap\mathcal{O}_{B}^{\times}=\Gamma^{(+)}(B,\mathcal{O}_{B}).$ 

\subsection{Proofs of Theorem \ref{thm:Shimura intro}, \ref{thm:Shimura not Q intro}}

We continue with the case when the Shimura curve $X_{\F}$ has geometric
genus zero. Equivalently, there is a finite base change $\F\hookrightarrow\F'$
such that $X_{\F'}\cong\P_{\F'}^{1}.$ Moreover, after perhaps increasing
$\F',$ we may assume that the irreducible components of $\Delta_{\F'}(:=\Delta_{\F}\otimes\F')$
are defined by $\F'-$points. The scheme $\mathcal{X}\otimes_{\mathcal{O}_{\F}}\mathcal{O}_{\F'}$
is still normal and $\Q-$Gorenstein (see \cite[Section 4.1]{yu1}). 

In the cases considered below we will show that the optimal model
of $(X_{\F'},\Delta_{\F'})$ over $\mathcal{O}_{\F'}$ is of the form
$(\P_{\mathcal{O}_{\F'}}^{1},\mathcal{D}^{o})$ for a divisor $\mathcal{D}^{o}$
on $\P_{\mathcal{O}_{\F'}}^{1}.$ It follows from Lemma \ref{lem:local h}
that, for any fixed metric on $K_{(X,\Delta)(\C)},$
\begin{equation}
\frac{1}{[\F':\Q]}\left(h(\overline{\mathcal{K}_{(\mathcal{X},\mathcal{D})}\otimes_{\mathcal{O}_{\F}}\mathcal{O}_{\F'}})-h(\overline{\mathcal{K}_{(\P_{\mathcal{O}_{\F}'}^{1},\mathcal{D}^{o})})}\right)=\sum_{p}h(p)\log p,\label{eq:difference h after base change in terms of q}
\end{equation}
 for a finite number of prime numbers $p$ and rational numbers $h(p),$
independent of the choice of $\F'.$ The number $h(p)$ may be geometrically
expressed as follows. Fix any normal model $\mathcal{Y}$ of $X_{\F'}$
over $\mathcal{O}_{\F'},$ dominating both $\mathcal{X}\otimes_{\mathcal{O}_{\F}}\mathcal{O}_{\F'}$
and $\P_{\mathcal{O}_{\F'}}^{1}.$ Denote by $\mathfrak{p}_{i}$ the
prime ideals in $\mathcal{O}_{\F'}$ over $p.$ Then 

\[
h(p)=\frac{1}{[\F':\Q]}\sum_{i}h(\mathfrak{p}_{i})f_{i},\,\,\,\,N(\mathfrak{p}_{i})=:p^{f_{i}}
\]
 where $h(\mathfrak{p}_{i})$ is the sum of intersection numbers on
the fiber $\mathcal{Y}_{\mathfrak{p}_{i}}$ defined in Lemma \ref{lem:local h},
for $\mathcal{L}'=\mathcal{K}_{(\mathcal{X},\mathcal{D})}\otimes_{\mathcal{O}_{\F}}\mathcal{O}_{\F'}$
and $\mathcal{L}=\mathcal{K}_{(\P_{\mathcal{O}_{\F}'}^{1},\mathcal{D}^{o})}).$
Note that $h(\mathfrak{p}_{i})\geq0.$ Indeed, since $\mathcal{K}_{(\mathcal{X},\mathcal{D})}\otimes_{\mathcal{O}_{\F}}\mathcal{O}_{\F'}$
is the log canonical line bundle of the log pair $(\mathcal{X},\mathcal{D})\otimes_{\mathcal{O}_{\F}}\mathcal{O}_{\F'}$
(using that $\mathcal{K}_{\mathcal{X}}\otimes_{\mathcal{O}_{\F}}\mathcal{O}_{\F'}=\mathcal{K}_{\mathcal{X}\otimes_{\mathcal{O}_{\F}}\mathcal{O}_{\F'}}),$
the non-negativity of $h(\mathfrak{p}_{i})$ follows from Remark \ref{rem:local ineq}.

We will compute $h(p)$ for some Shimura curves. By the uniqueness
of prime factorization it will be enough to compute the left hand
side in formula \ref{eq:difference h after base change in terms of q}. 

\subsubsection{Height formulas}

All heights will be computed wrt the Petersson metric on $K_{(X,\Delta)(\C)},$
denoted by $h_{\text{Pet}}.$ 
\begin{lem}
\label{lem:The-volume-of Pet}The volume of the measure $\mu$ corresponding
to the Petersson metric on $K_{(X,\Delta)}$ is equal to $\pi V/2:$
\[
\int_{X}\mu=\pi V/2,\,\,\,V:=V(K_{(X,\Delta)}).
\]
In particular, , $\hat{h}_{\text{Pet}}(\mathcal{K}_{(\P_{\Z}^{1},\mathcal{D}^{o})})=\hat{h}_{\text{can }}(\mathcal{K}_{(\P_{\Z}^{1},\mathcal{D}^{o})})+\frac{1}{2}\log(\pi V/2).$ 
\end{lem}

\begin{proof}
First observe that, in general, if $\mu$ is the measure corresponding
to a finite energy metric $\phi$ on $K_{(X,\Delta)},$ then 
\[
\int_{X-\Delta}dd^{c}\phi=V(K_{(X,\Delta)})
\]
 (using that $dd^{c}\phi$ does not charge finite subsets). Now, let
$\mu$ by the measure induced by Petersson metric on $K_{(X,\Delta)}.$
As recalled above this means that $\mu=\mu_{\phi}$ where, locally,
$\phi:=-\log(\left\Vert d\tau\right\Vert ^{2}):=-\log((2y)^{2}).$
Note that $dd^{c}\phi=\frac{1}{\pi}\mu$ on $X-\Delta.$ Indeed, 
\[
dd^{c}\phi:=\frac{1}{\pi}\frac{i}{2}\partial\bar{\partial}\phi:=\frac{1}{\pi}(\frac{\partial}{\partial z}\frac{\partial}{\partial\overline{z}}\phi)dxdy=\frac{2}{\pi}\frac{1}{(2y)^{2}}dxdy=:\frac{2}{\pi}e^{\phi}dxdy=:\frac{2}{\pi}\mu
\]
All in all this means that $V(K_{(X,\Delta)})=\int_{X-\Delta}dd^{c}\phi=\int_{X-\Delta}\frac{2}{\pi}\mu,$
proving the desired formula.
\end{proof}
Theorem \ref{thm:explicit intro} thus implies the following corollary,
where $\mathcal{D}^{o}$ denotes the divisor on $\P_{\Z}^{1}$ defined
as the Zariski closure of the divisor $\Delta_{\Q}$ on $\P_{\Q}^{1}$
supported on $\{0,1,\infty\}$ with weights $w_{i}\in[0,1].$
\begin{cor}
The following formula holds when $K_{(\P^{1},\Delta)}$ is ample:
\[
\hat{h}_{\text{Pet}}(\mathcal{K}_{(\P_{\Z}^{1},\mathcal{D}^{o})})=\frac{1}{2}-\frac{\gamma(0,\frac{V}{2})-\sum_{i=1}^{3}\gamma(w_{i}-\frac{V}{2},w_{i})}{V}
\]
\end{cor}

\subsubsection{\label{subsec:modular curve}The case $\F=\Q,$ $\Sigma_{f}=\emptyset$}

Let us show how to recover height formula \ref{eq:yuan intro} in
the case $\Sigma_{f}=\emptyset$ from Theorem \ref{thm:explicit intro}.
By \cite{d-r} the corresponding canonical model $\mathcal{X}$ is
isomorphic to $\P_{\Z}^{1}$ over $\Z$ (under the morphism defined
by the $j-$invariant) and $\Delta_{\Q}(:=\mathcal{D}\otimes_{\Z}\Q)$
is supported on the three points $0,1728$ and $\infty$ in $\P_{\Q}^{1}$
with ramification indices $2,3$ and $\infty,$ respectively. In general,
if $a\in\Z$ the divisor $\mathcal{D}_{a}$ on $\P_{\Z}^{1}$ defined
as the Zariski closure of the divisor on $\P_{\Q}^{1}$ supported
at the points $0,a$ and $\infty$ (with given weights $w_{0},w_{1}$
and $w_{\infty}$) satisfies
\[
\hat{h}_{\text{Pet}}(\mathcal{K}_{(\mathcal{X},\mathcal{D}_{a})})=\hat{h}_{\text{Pet}}(\mathcal{K}_{(\mathcal{X},\mathcal{D}_{1})})-\left(\frac{\sum_{i\leq\infty}w_{i}}{2}-\sum_{i<\infty}w_{i}\right)\log a.
\]
This follows Theorem \ref{thm:periods n one both signs}, using the
change of variables $z_{i}=a\zeta_{i}$ in the integral formula \ref{eq:Z N as integral of Vandermonde}
(but it can also be shown directly using scheme theory). In the present
case $\mathcal{D}=\mathcal{D}_{1728},$ i.e. $a=1728.$ The bracket
above thus becomes $-\frac{1}{12}.$ Since $1728=12^{3}(=2^{6}3^{3})$
it follows that 
\[
\hat{h}_{\text{Pet}}(\mathcal{K}_{(\mathcal{X},\mathcal{D}_{a})})=\hat{h}_{\text{Pet}}(\mathcal{X},\mathcal{D}_{1})+\frac{1}{12}\log(12^{3})(\implies h(2)=\frac{1}{2},\,\,\,h(3)=\frac{1}{4}.
\]
A new proof of formula \ref{eq:yuan intro} is thus obtained by invoking
the formula for ramification indices $(2,3,\infty)$ in Table 1. 
\begin{rem}
\label{rem:not stable log pair}From the identity $\mathcal{D}=\mathcal{D}_{1728}$
one sees directly that the reduction mod $p$ of $(\P_{\Z}^{1},\mathcal{D})$
is log canonical iff the prime $p$ is not in $\{2,3\}$ and that
$\mathcal{K}_{(\P_{\Z}^{1},\mathcal{D})}$ is isomorphic to $\mathcal{K}_{(\P_{\Z}^{1},\mathcal{D}^{o})}$
precisely over the complement in $\P_{\Z}^{1}$ of the fibers over
$(2)$ and ($3).$ This is consistent (as it must) with the fact,
shown above, that $h(p)$ vanishes iff $p$ is not in $\{2,3\}$ (see
Remark \ref{rem:local ineq}). 
\end{rem}

\subsubsection{The case $\F=\Q,\Sigma_{f}=\{2,3\}$ (proof of Theorem \ref{thm:Shimura intro})}

Now consider the case when the indefinite quaternion algebra $B$
over $\Q$ has discriminant $6,,$ i.e. it is ramified at $2$ and
$3.$ According to a result attributed to Ihara, the corresponding
Shimura curve $X_{\Q}$ is the subscheme of $\P_{\Q}^{2}$ cut out
by $x_{0}^{2}+3x_{1}^{2}+x_{2}^{2}$ (see \cite[Section 3.1]{el}
for a proof). In particular, $X_{\Q}$ admits a $\Q(\sqrt{-3})-$point,
e.g. $[1:(\sqrt{-3})^{-1}:0].$ It follows that $X_{\Q}\otimes\Q(\sqrt{-3})$
is isomorphic to $\P_{\Q(\sqrt{-3})}^{1}$ (by stereographic projection
through any $\F-$point). Furthermore, by \cite[Section 3.1]{el}
and \cite[Table 3]{tak}, setting $\F:=\Q(\sqrt{3},i)$ the corresponding
divisor $\Delta_{\Q}\otimes\F$ is supported at four $\F-$points
with ramification indices $(3,3;2,2).$ Moreover, as explained in
\cite[Section 3.1]{el}, the cross ratio of the corresponding pair
of two points is $-1.$ Denote by $(\P_{\mathcal{O}_{\F}}^{1},\mathcal{D}^{o})$
the corresponding unique optimal model over $\mathcal{O}_{\F},$ furnished
by Lemma \ref{lem:optimal model for three and four pts} and Prop
\ref{prop:optimal model for log P one}. 

We will compute the left hand side in formula \ref{eq:difference h after base change in terms of q}
wrt the Petersson metric.
\begin{lem}
The following formula holds,
\[
\hat{h}_{\text{Pet}}(\overline{\mathcal{K}_{(\P_{\mathcal{O}_{\F}}^{1},\mathcal{D}^{o})})}=-\frac{\zeta'(-1)}{\zeta(-1)}-\frac{1}{2}-\left((\frac{1}{6}-\frac{1}{2})\log2-\frac{1}{8}\log3\right),
\]
 which, combined with formula \ref{eq:yuan intro}, gives
\[
\hat{h}(\overline{\mathcal{K}_{(\mathcal{X},\mathcal{D})}})-\hat{h}_{\text{}}(\overline{\mathcal{K}_{(\P_{\mathcal{O}_{\F}}^{1},\mathcal{D}^{o})})}=\frac{11}{12}\log2+\frac{7}{8}\log3
\]
\end{lem}

\begin{proof}
By Prop \ref{prop:four points}
\[
\hat{h}_{\text{Pet}}(\overline{\mathcal{K}_{(\P_{\mathcal{O}_{\F}}^{1},\mathcal{D}^{o})})}=\hat{h}_{\text{Pet}}(\overline{\mathcal{K}_{(\P_{\Z}^{1},\mathcal{D}')})}+\frac{1}{2}\log2,
\]
 where $\mathcal{D}'$ is the Zariski closure of the divisor supported
at $(0,1,\infty)$ with ramification indices $(6,2,6).$ Hence, by
Table 1, 
\[
\hat{h}_{\text{Pet}}(\overline{\mathcal{K}_{(\P_{\mathcal{O}_{\F}}^{1},\mathcal{D}^{o})})}=-\frac{\zeta'(-1)}{\zeta(-1)}-\frac{1}{2}-\left(\frac{1}{6}\log2-\frac{1}{8}\log3\right)+\frac{1}{2}\log2,
\]
Combining this result with Yuan's formula \ref{eq:yuan intro} for
$\frak{\mathfrak{p}}=(2)$ and $\frak{\mathfrak{p}}=(3)$ in $\Z$
reveals that 
\[
\hat{h}(\overline{\mathcal{K}_{(\mathcal{X},\mathcal{D})}})-\hat{h}_{\text{}}(\overline{\mathcal{K}_{(\P_{\mathcal{O}_{\F}}^{1},\mathcal{D}^{o})})}=\frac{3\cdot2-1}{4}\log2+\frac{3\cdot3-1}{4(3-1)}\log3+\left((\frac{1}{6}-\frac{1}{2})\log2-\frac{1}{8}\log3\right)=
\]
\[
=\left(\frac{3\cdot2-1}{4}+\frac{1}{6}-\frac{1}{2}\right)\log2+\left(\frac{3\cdot3-1}{4(3-1)}-\frac{1}{8}\right)\log3=\frac{11}{12}\log2+\frac{7}{8}\log3
\]
\end{proof}
Since the normalized height is invariant under base change we have
$\hat{h}(\overline{\mathcal{K}_{(\mathcal{X},\mathcal{D})}\otimes_{\Z}\mathcal{O}_{\F}})=\hat{h}(\overline{\mathcal{K}_{(\mathcal{X},\mathcal{D})}}).$
Hence, setting $\hat{h}(p):=h(p)/(2K_{(X_{\F},\Delta_{\F})}\cdot X_{\F})$
the previous lemma gives $\hat{h}(2)=\frac{11}{12}$ and $\hat{h}(3)=\frac{7}{8}$
(using uniqueness of prime factorization in $\Z).$ Since $h(p)=\hat{h}(p)\cdot2/3$
this means that $h(2)=$11/18 and $h(3)=7/12.$ 

\subsubsection{The quaternion algebra over $\Q(\sqrt{3})$ ramified over $3$ (proof
of Theorem \ref{thm:Shimura not Q intro})}

Now consider the quaternion algebra over $\Q(\sqrt{3})$ that is only
ramified at the unique prime ideal $\frak{\mathfrak{p}}_{3}$ in $\mathcal{O}_{\Q(\sqrt{3})}$
containing $3.$ In fact, $\frak{\mathfrak{p}}_{3}=(\sqrt{3}).$ Indeed,
$(3)$ is the square of the ideal $(\sqrt{3}),$ which has norm $N(\mathfrak{p_{3}})=3.$
As a consequence, the contribution from prime ideals in formula \ref{eq:yuan intro}
for $\hat{h}(\overline{\mathcal{K}_{(\mathcal{X},\mathcal{D})}})$
is 
\[
\frac{3N(\frak{\mathfrak{p}})-1}{4(N(\frak{\mathfrak{p}})-1)}\sum_{\frak{\mathfrak{p}}}\log N(\frak{\mathfrak{p}})=\log3,
\]
Moreover, by \cite[Table 3]{tak}, $(\mathcal{X},\mathcal{D})\otimes\bar{\Q}$
is isomorphic to $(\P_{\Z}^{1},\mathcal{D}^{o})\otimes\bar{\Q,}$
where $\mathcal{D}^{o}$ is the divisor appearing in Theorem \ref{thm:explicit intro}
with ramification indices $(2,4,12).$ Fix a finite field extension
$\F$ of $\Q(\sqrt{3})$ such that $(\mathcal{X},\mathcal{D})\otimes\F$
is isomorphic to $(\P_{\Z}^{1},\mathcal{D}^{o})\otimes\F.$ Combining
formula \ref{eq:yuan intro} with Theorem \ref{thm:explicit intro}
thus yields, using Table 1, 
\[
\hat{h}(\overline{\mathcal{K}_{(\mathcal{X},\mathcal{D})}})-\hat{h}_{\text{}}(\overline{\mathcal{K}_{(\P_{\Z}^{1},\mathcal{D}^{o})})}=\frac{1}{2}\log3+\frac{5}{3}\log2+\frac{7}{16}\log3.
\]
 Note that $1/2+7/16=15/16.$ Since the normalized height is invariant
under base change, it follows that
\[
\hat{h}(\overline{\mathcal{K}_{(\mathcal{X},\mathcal{D})}\otimes\mathcal{O}_{\F}})-\hat{h}_{\text{}}(\overline{\mathcal{K}_{(\P_{\mathcal{O}_{\F}}^{1},\mathcal{D}^{o})})}=\hat{h}_{2}\log2+\hat{h}_{3}\log3,\,\,\,\hat{h}_{2}=\frac{5}{3},\,\hat{h}_{3}=\frac{15}{16}.
\]
 Since $h(p)=\hat{h}(p)\cdot(2K_{(X_{\F},\Delta_{\F})}\cdot X_{\F})$
and $K_{(X_{\F},\Delta_{\F})}\cdot X_{\F}=1/2+3/4+1/12-2=1/6$ we
deduce that

\[
h(2)=2\cdot(1/6)\cdot(5/3)=\frac{5}{9},\,\,\,\,h(3)=2\cdot(1/6)\cdot(15/16)=\frac{15}{48}.
\]

\subsubsection{The quaternion algebra over $\Q(\sqrt{6})$ ramified over $2$}
\begin{thm}
\label{thm:sqrt six}Consider the quaternion algebra over $\mathbb{Q}(\sqrt{6})$
ramified over the unique prime ideal $\mathfrak{p}_{2}$ containing
$2$ and denote by $(\mathcal{X},\mathcal{D})$ the canonical model
over \emph{$\mathcal{O}_{\mathbb{Q}(\sqrt{6})}$} of the corresponding
Shimura curve $(X_{\mathbb{Q}(\sqrt{6})},\Delta_{\mathbb{Q}(\sqrt{6})}).$\emph{
}Fix a finite field extension\emph{ $\F$ }of $\mathbb{Q}(\sqrt{6})$
such that $X_{\mathbb{Q}(\sqrt{6})}\otimes\F$ is isomorphic to $\P_{\F}^{1}$
and $\Delta_{\F}$ is supported on three $\F-$points. Then the optimal
model of $(X_{\mathbb{Q}(\sqrt{6})},\Delta_{\mathbb{Q}(\sqrt{6})})\otimes\F$
over $\mathcal{O}_{\mathbb{Q}(\sqrt{6})}$ is given by $(\P_{\mathcal{O}_{\F}}^{1},\mathcal{D}^{o}),$
where $\mathcal{D}^{o}$ denotes the Zariski closure of the divisor
on $\P_{\F}^{1}$ supported on $\{0,1,\infty\}$ having the same ramification
indices $(3,4,6)$ as the divisor $\Delta_{\F}.$ Moreover $h(p)=0$
unless $p=2$ or $p=3$ and 
\[
h(2)=\frac{43}{144},\ h(3)=\frac{3}{32}.
\]
\end{thm}

\begin{proof}
The unique prime ideal $\mathfrak{p}_{2}$ of $\mathbb{Q}(\sqrt{6})$
containing $2$ is given by $(2+\sqrt{6})$, which has norm $N(\mathfrak{p}_{2})=2$.
As such, the contribution coming from the prime ideals in \ref{eq:yuan intro}
is given by
\[
\frac{3N(\frak{\mathfrak{p}})-1}{4(N(\frak{\mathfrak{p}})-1)}\sum_{\frak{\mathfrak{p}}}\log N(\frak{\mathfrak{p}})=\frac{7}{4}\log2.
\]

By \cite[Table 3]{tak}, $(\mathcal{X},\mathcal{D})\otimes\bar{\Q}$
is isomorphic to $(\P_{\Z}^{1},\mathcal{D}^{o})\otimes\bar{\Q,}$
where $\mathcal{D}^{o}$ is the divisor appearing in Theorem \ref{thm:explicit intro}
with ramification indices $(3,4,6).$ Combining formula \ref{eq:yuan intro}
and row 7 in Table 1 yields,
\[
\hat{h}(\overline{\mathcal{K}_{(\mathcal{X},\mathcal{D})}})-\hat{h}_{\text{}}(\overline{\mathcal{K}_{(\P_{\Z}^{1},\mathcal{D}^{o})})}=\frac{1}{2}\frac{7}{4}\log2+\frac{9}{16}\log3+\frac{11}{12}\log2.
\]
\end{proof}
Since the normalized height is invariant under base change, it follows
that
\[
\hat{h}(\overline{\mathcal{K}_{(\mathcal{X},\mathcal{D})}\otimes\mathcal{O}_{\F}})-\hat{h}_{\text{}}(\overline{\mathcal{K}_{(\P_{\mathcal{O}_{\F}}^{1},\mathcal{D}^{o})})}=\hat{h}_{2}\log2+\hat{h}_{3}\log3,\,\,\,\hat{h}_{2}=\frac{43}{24},\,\hat{h}_{3}=\frac{9}{16}.
\]
 Since $h(p)=\hat{h}(p)\cdot(2K_{(X_{\F},\Delta_{\F})}\cdot X_{\F})$
and $K_{(X_{\F},\Delta_{\F})}\cdot X_{\F}=2/3+3/4+5/6-2=1/12$ we
deduce that

\[
h(2)=2\cdot(1/12)\cdot(43/24)=\frac{43}{144},\ h(3)=2\cdot(1/12)\cdot(9/16)=\frac{3}{32}.
\]

\subsection{\label{subsec:Implications-for-wild}Implications for wild ramification
and intersections over special places}

As recalled in Section \ref{subsec:The-setup-of Shim}, the canonical
integral model $\mathcal{X}$ of a quaternionic Shimura curve, comes,
locally over the base $\text{Spec \ensuremath{\mathcal{O}_{\F}}}$,
with a finite morphism from a regular scheme $\mathcal{X}'$ to $\mathcal{X}$
(formula \ref{eq:finite morphism absolute}), induced by the action
of a finite group $G$ on $\mathcal{X}'.$ The morphism induces an
effective divisor $\mathcal{D}$ on $\mathcal{X}.$ Consider now $\text{\ensuremath{\mathfrak{p}}\ensuremath{\ensuremath{\in}}Spec \ensuremath{\mathcal{O}_{\F}}}$
which is \emph{split}, i.e. $\mathfrak{p}$ is not in the ramification
locus of the quaternion algebra $B.$ Denote by $\kappa$ the residue
field of $\mathfrak{p}$ and by $\overline{\kappa}$ its algebraic
closure. Both $\mathcal{X}'_{\mathfrak{p}}\otimes_{\kappa}\overline{\kappa}$
and $\mathcal{X}{}_{\mathfrak{p}}\otimes_{\kappa}\overline{\kappa}$
are smooth \cite[Section 4.1]{yu1}. Moreover, by \cite[Prop 4.1]{yu1},
the restricted finite morphism 
\begin{equation}
g:\,\,\mathcal{X}'_{\mathfrak{p}}\rightarrow\mathcal{X}_{\mathfrak{p}}\label{eq:restricted finite map}
\end{equation}
is unramified at the generic points of $\mathcal{X}'_{\mathfrak{p}}.$
This means that the restricted finite morphism \ref{eq:restricted finite map}
is a ramified cover in the sense of \cite[Def 2.39]{ko}. Accordingly,
the restriction of $\mathcal{D}$ to $\mathcal{X}_{\frak{\mathfrak{p}}}$
defines a divisor on $\mathcal{X}_{\frak{\mathfrak{p}}}$ that we
shall denote by $\mathcal{D_{\frak{\mathfrak{p}}}}.$ In general,
a ramified cover is called called \emph{tame} at a given prime divisor
$P'$ on $\mathcal{X}'_{\mathfrak{p}}$ if the characteristic of the
residue field of $P'$ does not divide the ramification index of $g$
along $P'.$ We will say that $g$ has \emph{wild ramification} if
the ramification is not tame at all prime divisors $P'$ on $\mathcal{X}'_{\mathfrak{p}}.$
Theorems \ref{thm:Shimura not Q intro}, \ref{thm:sqrt six} imply
the following 
\begin{cor}
\emph{\label{cor:wild}}When $\mathcal{X}$ is the canonical model
in Theorem \ref{thm:Shimura not Q intro} the log pair $(\mathcal{X}_{\frak{\mathfrak{p}_{2}}},\mathcal{D_{\frak{\mathfrak{p}_{2}}}})$
is not log stable, i.e. it is not log canonical (lc). As a consequence,
the ramified cover \ref{eq:restricted finite map} has wild ramification
over $\mathfrak{p}_{2}$ and some of the irreducible components of
the divisor $\mathcal{D}$ on $\mathcal{X}$ coincide, when restricted
to the fiber of $\mathcal{X}$ over $\mathfrak{p}_{2}.$ Moreover,
when $\mathcal{X}$ is the canonical model in Theorem \ref{thm:sqrt six},
the corresponding result holds over $\mathfrak{p}_{3}$.
\end{cor}

\begin{proof}
Denote by \emph{$\frak{\mathfrak{p}}$ }a prime ideal appearing in
the statement of the corollary and by $\F$ the totally real field
in question. Since \emph{$\frak{\mathfrak{p}}$} is split \cite[Prop 4.1]{yu1}
shows, as recalled above, that \ref{eq:restricted finite map} is
a ramified cover. Let us first show that the \emph{log pair $(\mathcal{X}_{\frak{\mathfrak{p}}},\mathcal{D_{\frak{\mathfrak{p}}}})$
is} \emph{not lc.} Assume, in order to get a contradiction, that $(\mathcal{X}_{\frak{\mathfrak{p}}},\mathcal{D_{\frak{\mathfrak{p}}}})$
is lc. Take a finite field extension $\F'$ of $\F$ to which Theorem
\ref{thm:Shimura not Q intro} applies and fix a prime ideal $\mathfrak{p}'$
in $\mathcal{O}_{\F'}$ over $\mathfrak{p}.$ Then the restriction
of $(\mathcal{X},\mathcal{D})\otimes_{\mathcal{O}_{\F}}\mathcal{O}_{\F'}$
to the fiber over $\mathfrak{p}',$ that we denote by $(\mathcal{X}_{\mathfrak{p}'},\mathcal{D}_{\mathfrak{p}'})$,
is also lc. Indeed, in general, as recalled in Section \ref{subsec:Singularities-of-log},
if $X$ is a normal scheme of dimension one over a perfect field,
then $(X,D)$ is lc iff $w_{i}\leq1$ for all coefficients $w_{i}$
of $\mathcal{D}.$ Since any finite field is perfect this applies
to $(\mathcal{X}_{\mathfrak{p}},\mathcal{D}_{\frak{\mathfrak{p}}}).$
Hence, decomposing $\mathcal{D}_{\frak{\mathfrak{p}}}=\sum w_{i}D_{\frak{\mathfrak{p}}}^{(i)}$
where $D_{\frak{\mathfrak{p}}}^{(i)}$ is a prime divisor on $X$
we have $w_{i}\leq1.$ Next, since the residue field $\F_{\frak{\mathfrak{p}}}$
is perfect, $D_{\frak{\mathfrak{p}}}^{(i)}\otimes_{\F_{\frak{\mathfrak{p}}}}\F_{\frak{\mathfrak{p}'}}$
is a sum of \emph{distinct} irreducible divisors $D_{\frak{\mathfrak{p}}}^{(i)}\otimes_{\F_{\frak{\mathfrak{p}}}}\F_{\frak{\mathfrak{p}'}}=\sum_{j}D_{\frak{\mathfrak{p}'}}^{(i,j)}.$
As a consequence, the coefficients of $\mathcal{D}_{\frak{\mathfrak{p}'}}$
are at most $1,$ showing that $(\mathcal{X}_{\mathfrak{p}'},\mathcal{D}_{\mathfrak{p}'})$
is indeed lc. But this implies that $h(\mathfrak{p}')=0,$ where $h(\mathfrak{p}')$
is defined in formula \ref{eq:height difference intro-1}, comparing
$(\mathcal{X},\mathcal{D})\otimes_{\mathcal{O}_{\F}}\mathcal{O}_{\F'}$
with the optimal model appearing in Theorem \ref{thm:Shimura not Q intro}.
Indeed, since $(\mathcal{X}_{\mathfrak{p}'},\mathcal{D}_{\mathfrak{p}'})$
is lc the vanishing $h(\mathfrak{p}')=0$ follows from Prop \ref{prop:cond one and two}.
Finally, the vanishing of $h(\mathfrak{p}')$ for all prime ideals
$\mathfrak{p}'$ over $\mathfrak{p}$ implies that $h(p)=0,$ which
contradicts Theorem \ref{thm:Shimura not Q intro}. Next, to show
the statement about wild ramification, first observe that, since \ref{eq:restricted finite map}
is a ramified cover we have
\[
g^{*}\left(\mathcal{K}_{\mathcal{X}_{\mathfrak{p}}}+\mathcal{D_{\frak{\mathfrak{p}}}}\right)=\mathcal{K}_{\mathcal{X}'_{\mathfrak{p}}}.
\]
Assume, to get a contradiction, that $g$ does not have wild ramification
over $\mathfrak{p}.$ This implies, since $(\mathcal{X}'_{\mathfrak{p}},0)$
is lc (and even klt) that $(\mathcal{X}_{\mathfrak{p}},D_{\frak{\mathfrak{p}}})$
is lc, by a Hurwitz type formula (see \cite[Cor 2.43]{ko}). This
is a contradiction. Likewise, if the irreducible components of $\mathcal{D_{\frak{\mathfrak{p}}}}$
were all distinct, then the coefficients of $\mathcal{D_{\frak{\mathfrak{p}}}}$
would all be of the form $1-1/m_{i}$ for positive integers $m_{i}$
(since $\mathcal{D}$ is the Zariski closure of an orbifold divisor
on the generic fiber). Thus $(\mathcal{X}_{\mathfrak{p}},D_{\frak{\mathfrak{p}}})$
would be klt, contradicting that it is not even lc.
\end{proof}
In general, when $\mathfrak{p}$ is split, $\mathcal{X}_{\mathfrak{p}}$
is isomorphic to $\mathcal{X}'_{\mathfrak{p}}/G,$ by \cite[Prop 4.1]{yu1}.
The previous corollary also applies to the classical case when $B=M_{2}(\Q),$
where $\mathcal{X}$ is the compactification of the coarse moduli
space of elliptic curves over $\Z.$ In this case all $\mathfrak{p}$
are split. The non-vanishing \ref{eq:h for modular curve} thus implies
that \ref{eq:restricted finite map} has wild ramification over $\mathfrak{p}=(2)$
and $\mathfrak{p}=(3).$ This also follows from classical results
about elliptic curves. Indeed, for $p=2$ and $p=3$ there exist elliptic
curves $E$ over $\Z/(p)$ such that $\sharp(\text{Aut \ensuremath{(E)/\{\pm1\})}}$
is $12$ or $6,$ respectively. These elliptic curves give rise to
local ramification indices for the morphism \ref{eq:restricted finite map}
of order $12$ and $6,$ respectively, which are thus divided by $p.$

\section{Application to twisted Fermat curves}

In this Section we will, in particular, prove Theorem \ref{thm:Fermat intro}.
Given integers $a_{i}$ consider the subscheme $\mathcal{X}_{a}$
of $\P_{\Z}^{n+1}$ cut out by the homogeneous polynomial $\sum_{i=0}^{n+1}a_{i}x_{i}^{d}.$
This scheme will be denoted by $\mathcal{X}_{1}$ in the case $a_{i}=1.$
\begin{prop}
The following formula holds when $\pm K_{X_{a}}$ is ample (i.e when
$\pm(d-(n+2))>0$)
\[
h_{\text{can }}(\mathcal{K}_{\mathcal{X}_{a}})=h_{\text{can }}(\mathcal{K}_{\mathcal{X}_{1}})+(\frac{\left|n+2-d\right|}{(n+1)}\pm1)d^{-1}\sum_{i}\log(|a_{i}|).
\]
\end{prop}

\begin{proof}
The case when $-K_{X}>0$ is the content of \cite[formula 5.5]{a-b2}
(applied to $k=n+2-d$). The proof in the case when $K_{X}>0$ is
essentially the same, but then $k$ in \cite[Lemma 5.3]{a-b2} is
taken as $d-(n+2)$ (by adjunction) and the minus sign in $\pm$ results
from change in sign in front of $\log\int\mu_{\phi}$ (see \cite[Lemma 5.4]{a-b2}).
\end{proof}
It follows that
\[
h_{\text{can }}(\mathcal{X}_{a})\leq h_{\text{can }}(\mathcal{X}_{1})\:\text{when \ensuremath{-K_{X_{a}}>0,\,\,\,\,h_{\text{can }}(\mathcal{X}_{1})\leq h_{\text{can }}(\mathcal{X}_{a}),\,\text{when }}\ensuremath{K_{X_{a}}}}>0
\]
Equivalently, by Prop \ref{prop:var princi metrics}, this means that
\[
\inf_{\psi}\mathcal{M}_{\mathcal{X}_{1}}(\pm\mathcal{K}_{\mathcal{X}_{1}},\psi)\leq\inf_{\psi}\mathcal{M}_{\mathcal{X}_{a}}(\pm\mathcal{K}_{\mathcal{X}_{a}},\psi).
\]

Now we specialize to $n=1.$ Given a positive integer $m$ consider
the divisor $\mathcal{D}$ on $\P^{1}$ supported on $\{0,1,\infty\}$
with coefficients $(1-1/m).$
\begin{lem}
\label{prop:red to log hyper}Denote by $\mathcal{X}$ the Fermat
hypersurface of a given degree $m$ $(>2).$ Then, 
\[
\hat{h}_{\text{can }}(\mathcal{K}_{\mathcal{X}})=\hat{h}_{\text{can }}(\mathcal{K}_{(\P^{1},\mathcal{D})})+\frac{1}{2}\log\frac{V(X)}{V(\P^{1},\Delta)}\,\,\,\,\left(\frac{V(X)}{V(\P^{1},\Delta)}=m^{2}\right)
\]
\end{lem}

\begin{proof}
This is shown exactly as in the Fano case in \cite[Prop 5.6]{a-b2},
but now the last term comes with a different sign (due to the sign
difference in the definition of $\hat{h}_{\text{can }}(\mathcal{X})$). 
\end{proof}

\subsection{Proof of Theorem \ref{thm:Fermat intro}}

The first formula in Theorem \ref{thm:Fermat intro} follows directly
from combining the previous proposition and lemma with Theorem \ref{thm:explicit intro}.
Next, fix $m$ and $a$ and set $\mathcal{X}_{a}=\mathcal{X}_{a}^{(m)}.$
By \cite{d-m}, there exists a stable model $\mathcal{X}^{s}$ for
$\mathcal{X}_{a}\otimes_{\Z}\mathcal{O}_{\F}$ over $\mathcal{O}_{\F}$
for some number field $\F.$ Since the base change of a stable model
is still a stable model \cite[Section 1.5]{l-l}, we may as well assume
that $\F$ contains all $a_{i}^{1/m}.$ Thus $X_{a}\otimes_{\Q}\F$
is isomorphic to $X_{1}\otimes_{\Q}\F$ over $\F,$ showing that $\mathcal{X}^{s}$
is also a stable model for $\mathcal{X}_{1}\otimes_{\Z}\mathcal{O}_{\F}.$
Hence, to prove the inequalities in Theorem \ref{thm:Fermat intro},
it will - by the first formula in Theorem \ref{thm:Fermat intro}
(combined with Theorem \ref{thm:sharp bounds intro} and Lemma \ref{lem:sup over all cont phi})
- be enough to show that

\begin{equation}
\hat{h}_{\text{can }}(\mathcal{K}_{\mathcal{X}^{s}})\leq\hat{h}_{\text{can }}(\mathcal{K}_{\mathcal{X}_{1}}).\label{eq:height ine pf Fermat}
\end{equation}
But, by Cor \ref{cor:stable model}, $\hat{h}_{\text{can }}(\mathcal{K}_{\mathcal{X}^{s}})\leq\hat{h}_{\text{can }}(\mathcal{K}_{\mathcal{X}_{1}\otimes_{\Z}\mathcal{O}_{\F}}).$
Since $\mathcal{K}_{\mathcal{X}_{1}\otimes_{\Z}\mathcal{O}_{\F}}$
is isomorphic to $\mathcal{K}_{\mathcal{X}_{1}}\otimes_{\Z}\mathcal{O}_{\F}$
(by the adjunction formula) and the normalized height is invariant
under base change this proves the inequality \ref{eq:height ine pf Fermat}. 

\subsection{\label{subsec:The-Arakelov-metric}The Arakelov vs the Kähler-Einstein
metric (proof of Cor \ref{cor:Arak })}

Let $X_{\Q}$ be a non-singular projective curve of degree $m$ in
$\P_{\Q}^{2}$ and denote by $X$ its complex points. Assume that
$K_{X}$ is ample, i.e. $m\geq4.$ Denote by $g_{X}$ the genus of
$X.$ The Arakelov metric on $K_{X}$ may be defined as the metric
which turns the adjunction formula into an isometry \cite{fa84}. 

The first inequality in Cor \ref{cor:Arak } follows directly from
combining Theorem \ref{thm:Fermat intro} with the following bound
(using that $V=1-3/m)$:
\begin{lem}
\label{lem:Arak vs KE}For any given model $\mathcal{X}$ of $X$
over $\Z$
\[
\hat{h}_{\text{Ar}}(\mathcal{X})\leq\hat{h}_{\text{can }}(\mathcal{X})+\frac{1}{2}\log\pi+\frac{1}{2}\frac{4\log((m-1)(m-2)-2)+1}{(m-1)(m-2)/2-1}+\frac{1}{2}\log((m-1)(m-2)/2-1).
\]
\end{lem}

This bound follows from results in \cite{j-k1,j-k2}, as next explained.
First recall that, by \ref{eq:change of metrics formula for height},
\[
2\hat{h}_{\text{\ensuremath{\psi_{\text{Ar}}} }}(\mathcal{X})-2\hat{h}_{\text{can }}(\mathcal{X})=\frac{\mathcal{E}(\psi_{\text{Ar}},\psi_{\text{KE}})}{2V(K_{X})},
\]
 where $\psi_{\text{KE}}$ denotes the unique volume-normalized Kähler-Einstein
metric on $X.$ Comparing with the notation in \cite[Section 2.1]{j-k2},
$\mu_{\text{hyp}}:=4\pi(g_{X}-1)\mu_{\psi_{\text{KE}}},$ where $\mu_{\psi_{\text{KE}}}$
denotes the measure on $X$ corresponding to $\psi_{\text{KE}}.$
Denoting by $\psi_{\text{hyp}}$ the Kähler-Einstein metric on $K_{X}$
corresponding to $\mu_{\text{hyp}},$  \cite[Prop 4.5]{j-k2} thus
yields the following bound:
\[
\frac{\mathcal{E}(\psi_{\text{Ar}},\psi_{\text{hyp }})}{2V(K_{X})}:=\int_{X}\left(\psi_{\text{Ar}}-\psi_{\text{hyp}}\right)\frac{(dd^{c}\psi_{\text{hyp}}+dd^{c}\psi_{\text{Ar}})}{2V(K_{X})}\leq-\frac{c_{X}-1}{g_{X}-1}-\log4,
\]
 where $c_{X}$ is the finite part of the logarithmic derivative of
the Selberg zeta function at $s=1,$ defined before \cite[formula 2.8]{j-k2}.
Since $\psi_{\text{hyp }}:=\psi_{\text{KE}}+\log(4\pi(g_{X}-1))$
this means that 
\[
\frac{\mathcal{E}(\psi_{\text{Ar}},\psi_{\text{KE}})}{2V(K_{X})}\leq\frac{-c_{X}+1}{g_{X}-1}-\log4+\log(4\pi(g_{X}-1)).
\]
 Next, we recall that, by \cite[Thm 3.3]{j-k1}, 
\[
-c_{X}\leq4\log(2g_{X}-2)
\]
 (since $X$ has no cusps, nor elliptic points; compare \cite[Section 2.1]{j-k1}).
Hence, 
\begin{equation}
2\hat{h}_{\text{\ensuremath{\psi_{\text{Ar}}} }}(\mathcal{X})-2\hat{h}_{\text{can }}(\mathcal{X})\leq\frac{4\log(2g_{X}-2)+1}{g_{X}-1}+\log(\pi(g_{X}-1))\label{eq:difference height Arakelov and can}
\end{equation}
 and since $2g_{X}=(m-1)(m-2)$ this proves Lemma \ref{lem:Arak vs KE}.

Finally, to prove the second inequality in Cor \ref{cor:Arak } recall
that it was shown in the course of the proof of Theorem \ref{thm:sharp bounds intro}
that $f(t,t,t)$ is decreasing in $t.$ In the present case $t=1-1/m$
where $m\geq4.$ Moreover, since $f(1-1/4,1-1/4,1-1/4)$ is expressed
in Table 1 (for the ramification indices $(4,4,4)$) this concludes
the proof (the fact that $\epsilon_{m}$ is decreasing can be shown
by  elementary methods).

\subsection{\label{subsec:Comparison-with-Parhin's}Comparison with Parshin's
inequality in the geometric case}

Let $\mathcal{B}$ be a complex projective curve of genus $g_{\mathcal{B}}$
and $\mathcal{X}$ a complex projective surface with a morphism $\mathcal{X}\rightarrow\mathcal{B}$
such that the relative canonical line bundle $\mathcal{K}_{\mathcal{X}/\mathcal{B}}$
is relatively ample. Assume that $\mathcal{X},\mathcal{B}$ and the
generic fiber $X$ of $\mathcal{B}$ are regular and denote by $s$
the number of singular fibers of $\mathcal{X}\rightarrow\mathcal{B}.$
By \cite{pa,v}, the following geometric analog of Parshin's proposed
arithmetic inequality \ref{eq:parshin intro} holds:
\begin{equation}
\hat{h}(\mathcal{K}_{\mathcal{X}/\mathcal{B}}):=\frac{\mathcal{K}_{\mathcal{X}/\mathcal{B}}\cdot\mathcal{K}_{\mathcal{X}/\mathcal{B}}}{2\deg(K_{X})}\leq\max(0,g_{\mathcal{B}}-1)+\frac{1}{2}s,\label{eq:geom ineq for semistable}
\end{equation}
 when $\mathcal{X}\rightarrow\mathcal{B}$ is\emph{ (semi)-stable}.
This is a consequence of the Miyaoka--Yau inequality for $\mathcal{X}$.
Moreover, by \cite{tan0}, 
\begin{equation}
\hat{h}(\mathcal{K}_{\mathcal{X}/\mathcal{B}})\leq\max(0,g_{\mathcal{B}}-1)+\frac{3}{2}s,\label{eq:geom ineq for minimal}
\end{equation}
 when $\mathcal{X}\rightarrow\mathcal{B}$ is merely\emph{ relatively
minimal. }

In the arithmetic case the role of $g_{\mathcal{B}}$ is played by
$\log\left|D_{\F}\right|$ and the role of $s$ is played by $\sum_{\frak{\mathfrak{p}\text{\ensuremath{\text{bad}}}}}\log N(\frak{\mathfrak{p}})$
(see the discussion in \cite{v}). In particular, the case $g_{\mathcal{B}}=0$
corresponds to the case when $\F=\Q.$ Coming back to case of the
Zariski closure $\mathcal{X}^{(m)}$ in $\P_{\Z}^{2}$ of the Fermat
curve $X^{(m)}$ over $\Q$ of degree $m,$ recall that, by Cor \ref{cor:Arak },
\[
\hat{h}_{\text{can}}(\overline{\mathcal{K}_{\mathcal{X}^{(m)}}})<0+\log m,\,\,\,\,\hat{h}_{\text{Ar}}(\overline{\mathcal{K}_{\mathcal{X}^{(m)}}})<0+2\log m.
\]
As a consequence, the corresponding inequalities also hold when $\mathcal{X}^{(m)}$
is replaced by a stable model or a relatively minimal model (By Cor
\ref{cor:stable model} and Prop \ref{prop:minimal model}). Now specialize
to the case when $m$ is square-free. Then the role of $\log m$ is
played by $s$ in the geometric case. Thus, the inequality for $\hat{h}_{\text{can}}(\overline{\mathcal{K}_{\mathcal{X}_{\text{min}}^{(m)}}})$
is actually better than the inequality one would obtain from the geometric
inequality \ref{eq:geom ineq for minimal}, would it translate to
the arithmetic setup (since $1<3/2).$ 

In view of the previous discussion it seems natural to ask if, in
general, the direct analog of the geometric inequality \ref{eq:geom ineq for semistable}
holds for the volume-normalized Kähler-Einstein metric on $K_{X}$?
This would imply Parshin's inequality \ref{eq:parshin intro} for
the Arakelov metric on $K_{X}$ with explicit constants (using the
inequality \ref{eq:difference height Arakelov and can}). For example,
when $\Q=\F$ one would get $c_{1}=1/2,c_{2}=0$ and $c_{0}$ explitely
bounded from above by $\log\deg K_{X}.$ More generally, consider
a projective regular curve $X$ over a number field $\F,$ endowed
with a divisor $\Delta$ such that $K_{(X,\Delta}>0$ and fix a finite
field extension $\F'$ of $\F$ such that $(X,\Delta)_{\F}\otimes\F'$
admits a relatively stable model $(\mathcal{X},\mathcal{D})$ over
$\mathcal{O}_{\F'}$ (as discussed before the statement of Lemma \ref{lem:optimal model for three and four pts}).
Does the following inequality hold,
\[
\hat{h}(\mathcal{K}_{(\mathcal{X},\mathcal{D})})\leq\max(0,\log\left|D_{\F}\right|-1)+\frac{1}{2}\sum_{\frak{\mathfrak{p}\text{\ensuremath{\text{bad}}}}}\log N(\frak{\mathfrak{p}})?
\]
 The case when $\F=\F'=\Q,$ $X=\P_{\Q}^{1}$ and $\Delta$ has three
irreducible components follows from the second inequality in Theorem
\ref{thm:refined bounds intro} (there are no bad $\mathfrak{p}$
in this case).

\end{document}